\documentclass[english,a4paper,11pt,oneside]{article}

\usepackage{enumitem}
\usepackage{hyperref}
\usepackage{graphicx}
\usepackage[utf8]{inputenc}
\usepackage{amsmath, amsthm, amssymb}
\usepackage[english]{babel}
\usepackage{marvosym}
\usepackage{csquotes}

\usepackage{color}
\usepackage{varwidth}
\usepackage{svg}
\usepackage{tikz}
\usepackage[backend=biber, hyperref=true, sorting=nyt, style=alphabetic]{biblatex}
\bibliography{Literatur}
\newcommand{\hhookrightarrow}{\lhook\mkern-3mu\relbar\mkern-12mu\hookrightarrow}
\newtheorem{theorem}{Theorem}[section]
\theoremstyle{definition} 
\newtheorem{definition}[theorem]{Definition} 
\theoremstyle{definition} 
\newtheorem{lemma}[theorem]{Lemma} 
\theoremstyle{definition} 
\newtheorem{remark}[theorem]{Remark}
\theoremstyle{definition} 
\newtheorem{corollary}[theorem]{Corollary} 
\theoremstyle{definition}
 
\theoremstyle{definition} 
\newtheorem{assumption}[theorem]{Assumption} 
\theoremstyle{definition} 

\title{Optimal control of reaction-diffusion systems with hysteresis}
\author{Christian Münch\footnote{Department of Mathematics - M6, Technical University of Munich, Boltzmannstr. 3, 85747 Garching, Germany. christian.muench@ma.tum.de}}

\begin{document}
	\maketitle

%	\pagenumbering{roman}
%	\pagestyle{headings}
%	%%%% Zusammenfassung in deutscher Sprache
%	
%	\newpage
%	\tableofcontents
%	\newpage
%	
%	%%%% Page numbering restarts here
%	\pagenumbering{arabic}
%	\pagestyle{headings}
	
	\pagenumbering{arabic}  
	
\begin{abstract}
	This paper is concerned with the optimal control of hysteresis-reaction-diffusion systems.
	We study a control problem with two sorts of controls, namely distributed control functions, or controls which act on a part of the boundary of the domain.
	The state equation is given by a reaction-diffusion system with the additional challenge that the reaction term includes a scalar stop operator.
	We choose a variational inequality to represent the hysteresis. 
	In this paper, we prove first order necessary optimality conditions. In particular, under certain regularity assumptions, we derive results about the continuity properties of the adjoint system. For the case of distributed controls, we improve the optimality conditions and show uniqueness of the adjoint variables. We employ the optimality system to prove higher regularity of the optimal solutions of our problem. Finally, we derive regularity properties of the value function of a perturbed control problem when the set of controls is restricted.
	The specific feature of rate-independent hysteresis 
	%	in the reaction term of a reaction-diffusion system 
	in the state equation leads to difficulties concerning the analysis of the solution operator.
	Non-locality in time of the Hadamard derivative of the control-to-state operator complicates the derivation of an adjoint system.
\end{abstract}

\begin{center}
	Keywords: Optimal control, reaction-diffusion, semilinear parabolic evolution problem, hysteresis operator, stop operator, global existence, solution operator, Hadamard differentiability, optimality conditions, adjoint system.
	
	MSC subject class: 49J20, 47J40, 35K51
\end{center}

\section{Introduction}

In this paper, we derive an adjoint system for the optimal control problem
\begin{equation}
\min_{u \in U_i} J(y,u) := \frac{1}{2}  \|y-y_d\|_{U_1}^2+ \frac{\kappa}{2}\| u \|_{U_i}^2 \label{opt_control_ control_problem}
\end{equation} 
subject to 
\begin{alignat}{2}
\dot{y}(t) +(A_p y)(t) &= f(y(t),z(t))  + (B_i u)(t)&&\  \text{in } \mathbb{W}_{\Gamma_D}^{-1,p}(\Omega) \text{ for }t\in (0,T), \label{state_equ_y}\\
y(0)&= 0 &&\ \text{in } \mathbb{W}_{\Gamma_D}^{-1,p}(\Omega),\notag\\
z&=\mathcal{W}[Sy],\label{state_equ_z}
%(\dot{z}(t)-S\dot{y}(t))(z(t)-\xi) &\leq 0 &&\ \text{for } \xi\in [a,b] \text{ and } t\in (0,T),\label{state_equ_z}\\
%z(t)&\in [a,b] &&\ \text{for } t\in[0,T],\notag\\
%z(0)&=z_0\in [a,b]\notag.
\end{alignat} 
where $\mathbb{W}_{\Gamma_D}^{-1,p}(\Omega)$ is a product of dual spaces, see e.g. \cite[(16)-(18)]{muench} for the existence theory of problem \eqref{opt_control_ control_problem}-\eqref{state_equ_z} and related references therein.
We consider either spatially
distributed controls in the space
$
U_1:= \mathrm{L}^2 \left((0,T);[\mathrm{L}^2(\Omega)]^m\right)
$, or controls which act on given Neumann boundary parts $\Gamma_{N_j}$, $j\in\{1,\ldots m\}$, of the state space, i.e. controls in
$
U_2:= \mathrm{L}^2 \left((0,T); \prod_{j=1}^m \mathrm{L}^2(\Gamma_{N_j},\mathcal{H}_{d-1})\right).
$
The operators
$B_1 : [\mathrm{L}^2(\Omega)]^m \rightarrow \mathbb{W}_{\Gamma_D}^{-1,p}(\Omega)$ and 
$B_2 : \prod_{j=1}^m \mathrm{L}^2(\Gamma_{N_j},\mathcal{H}_{d-1}) \rightarrow \mathbb{W}_{\Gamma_D}^{-1,p}(\Omega)$ are continuous and $A_p$ is an unbounded diffusion operator on the space $\mathbb{W}_{\Gamma_D}^{-1,p}(\Omega)$. With $i\in\{1,2\}$, we identify $B_i$ with the corresponding continuous operators from $U_i$ into $\mathrm{L}^2 \left((0,T);\mathbb{W}_{\Gamma_D}^{-1,p}(\Omega)\right)$ which act pointwise in time, i.e. we write $(B_i u)(t) = B_i(u(t))$ for all $t\in (0,T)$. In the same way we identify $(A_p y)(t)$ with $A_p(y(t))$ for functions $y:(0,T)\rightarrow \mathbb{W}_{\Gamma_D}^{-1,p}(\Omega)$.
Moreover, $S$ projects $y$ to a scalar valued function. In particular,
$\mathcal{W}$ is a scalar stop operator and it is well-known (see e.g. \cite{visintin2013differential}, \cite{brokate2013optimal}) that
$\mathcal{W}$ is represented by the solution operator $z=\mathcal{W}[v]$ of the variational inequality
\begin{alignat}{2}
(\dot{z}(t)-\dot{v}(t))(z(t)-\xi) \leq 0& &&\ \text{for } \xi\in [a,b] \text{ and a.e. } t\in (0,T),\label{stop1}\\
z(t)\in [a,b]& &&\ \text{for } t\in[0,T],\ z(0)=z_0.\label{stop2}
\end{alignat} 
For $i\in\{1,2\}$, we denote by $G$ the operator, which maps $B_iu$ to the unique solution $y$ of \eqref{state_equ_y}-\eqref{state_equ_z}, see \cite[Theorem 3.1]{muench}.
Note that $y=G(B_iu)$ is a function of time with values in a product of dual spaces.

%Optimal control of reaction-diffusion systems has been analyzed before in various ways.

%In \cite{barthel2010optimal} a class of optimal boundary control problems with control constraints of box type is discussed. The state equation is given by a reaction-diffucion system modelling a chemical reaction. Necessary and sufficient (second order) optimality conditions are derived and numerical examples are given. 
%The reaction term is smooth (chemical reaction).
%
%Second order optimality conditions for similar problems have been derived in \cite{griesse2003diss}. They serve to prove sensitivity results with respect to parameters of a class of control problems which are subject to reaction-diffusion systems. Stability properties of the lagrange multipliers and the active sets are shown. Parametric sensitivity derivatives are finally computed numerically (cite introduction).

%Parametric sensitivity analysis for optimal control problems of a \cite{griesse2006parametric} Griesse Volkwein: extension of the diss.

Optimal control of (systems of) partial differential equations has extensively been analyzed in the literature before. 

In particular, optimal control problems with state equations of \emph{semilinear parabolic} type are part of the well-known monograph \cite{troeltzsch2010optimal} and the early work \cite{Bonnans85}. Further studies in this direction are the subject of \cite{Raymond98} and \cite{Casas97}. We also refer to \cite{rehberg_optimal} for a control problem with parabolic state equation and rough boundary conditions like in our setting.

Early studies in the field of optimal control of \emph{reaction-diffusion systems} and in particular in the direction of parameter sensitivity analysis have been performed in \cite{griesse2003diss} and were further established in \cite{griesse2006parametric} and several more papers. Optimality conditions for a similar problem were also derived in \cite{barthel2010optimal}.

%Most of the challenges in the derivation of optimality conditions in all the works mentioned so far are due to difficulties like state constraints or dynamic boundary conditions. 
The \emph{non-linearities} in all the works mentioned so far are mostly \emph{smooth} enough to obtain a (twice) continuously differentiable control-to-state operator, so that first and many times also second order optimality conditions could be derived.

In the literature, there are only few results available concerning optimal control of \emph{infinite-dimensional rate-independent processes}.
For a class of energetically driven processes, existence of optimal controls for problems of this type has first been
studied in \cite{rindler2008optimal} and \cite{rindler2009approximation}. 
Subsequently, the results were applied to (thermal) control problems in the field of shape memory
materials in \cite{eleuteri2013thermal} and \cite{eleuteri2014}. No optimality conditions are given in these works.
Optimal control of a problem of static plasticity in the infinite-dimensional setting is the subject of \cite{herzog2012c} and \cite{herzog2013b}. The results were used in \cite{herzog2014} to numerically solve a quasi-static control problem by time-discretization.
Optimality conditions for time-continuous, infinite-dimensional, rate-independent control problems of quasi-static plasticity type could be derived in \cite{wachsmuth2012optimal}, \cite{wachsmuth2015optimal}, \cite{wachsmuth2016optimal} by means of time-discretization.
Another time-continuous, infinite-dimensional optimal control problem of a rate-independent system, which is represented by its energetic formulation, is addressed in \cite{stefanelli2016optimal}. With help of viscous regularization, a necessary optimality condition is derived.

To our knowledge, the first results for optimal control of \emph{hysteresis} have been achieved in \cite{brokate1987optimale,brokate1988optimal,brokate91optimal}.
% in the ODE-setting
%, see also \cite{brokate91optimal} for the English translation.
%Shortly afterwards, in \cite{brokate1988optimal}, 
Necessary optimality conditions for the optimal control of an ODE-system with hysteresis were established.
An adjoint system was derived by a time discretization approach.
Optimal control of sweeping processes has been studied in \cite{castaing2014some}, \cite{colombo2012optimal} and \cite{colombo2016optimal}.
%The non-linearity in \cite{brokate1988optimal} was twice continuously differentiable.
In \cite{brokate2013optimal}, first order optimality conditions for a control problem of an ODE-system with hysteresis of (vectorial) stop type were derived. The stop operator is represented in form of a variational inequality.
%(rather than by a differential inclusion).
The main challenge with the stop operator (as with all hysteresis operators) is the fact that hysteresis acts non-local in time so that the state $y(t)$ at each time $t\in (0,T]$ depends on the whole background $(0,t)$. Moreover, the stop operator is not differentiable in the classical sense and so the control-to-state can not be expected to be so either.
Regularization techniques were used in order to derive an optimality system.
%the variational inequality which determines the stop operator was replaced by a regularization. Accordingly, $\varepsilon$-dependent regularized control problems were analyzed. Convergence of the corresponding optimal solutions and the adjoint systems finally yielded an optimality system of the original problem.
Several of the ideas of this approach are useful also for us.
To handle a reaction-diffusion system requires additional work though. Firstly, the state vector $y:[0,T] \rightarrow \mathbb{W}_{\Gamma_D}^{-1,p}(\Omega)$ in \eqref{state_equ_y} is a function with values in an infinite-dimensional space and secondly, the non-linearity $f$ in our case is not necessarily continuously differentiable but only locally Lipschitz continuous and directionally differentiable. Therefore, techniques as in \cite{meyeroptimal} are required. Particularly, since the domain $\Omega$ has a rough boundary, we have to consider a product of dual spaces for the domain of the diffusion operator $A_p$.

The existing literature provides only few rigorous results in the field of control of \emph{hysteresis-reaction-diffusion systems}, especially when it comes to optimal control of such systems.
In \cite{cavaterra2002automatic}, automatic control problems governed by reaction-diffusion systems with feedback control of relay switch and Preisach type have been studied. Global existence and uniqueness of solutions were proven.
Closed-loop control of a reaction-diffusion system coupled with ordinary differential inclusions has been considered in \cite{dudziuk2011closed}. A feedback law for the case with a finite number of control devices was derived.

Necessary conditions for the optimal control of (general) \emph{non-smooth semilinear parabolic equations} have been established in \cite{meyeroptimal}. In particular, the non-linearity is merely locally Lipschitz continuous and directionally differentiable so that the control-to-state operator is not differentiable in the classical sense. Regularization techniques have been used to derive an adjoint system. No hysteresis is considered in this paper.
Nevertheless, a modification of the approach in \cite{meyeroptimal} is applicable for the problem at hand.
In particular, we include ideas from \cite{brokate2013optimal} and adapt the proof to apply to non-localities in time such as hysteresis.
We refer to the references in \cite{meyeroptimal} for a good overview of further contributions dealing with optimal control of non-smooth parabolic equations.

In this paper, we are interested in the optimal control of \emph{non-smooth reaction-diffusion systems with hysteresis}. In particular, a scalar stop operator enters the non-linearity $f$.
The function $f$ is assumed to be locally Lipschitz continuous and directionally differentiable. Additionally, the domain $\Omega$ satisfies minimal smoothness assumptions.

The outline of the paper is as follows:

In Section~2, we introduce the framework for the rest of the work and collect results from the literature.
Subsection~2.3 contains the main assumption and notation.
%Section 4. and Section 5. repeat the results of \cite{muench} on the regularity of the solution operator for \eqref{state_equ_y}-\eqref{state_equ_z} and the existence of an optimal control for \eqref{opt_control_ control_problem}.

Our first main interest is to derive an adjoint system and first order necessary optimality conditions for problem \eqref{opt_control_ control_problem}-\eqref{state_equ_z}.

In Section~3, we introduce a family of regularized control problems with $\varepsilon$-dependent state equations and derive adjoint systems as well as optimality conditions for those. 
In particular, we regularize $f$ and the stop operator $\mathcal{W}$ in dependence of the parameter $\varepsilon>0$ and replace the original control problem by a regularized one. 
The corresponding control-to-state operator $u\mapsto G_\varepsilon(B_iu)$, $i\in\{1,2\}$, and the regularization $y\mapsto Z_\varepsilon(Sy)$ of $\mathcal{W}[S\cdot]$ are G\^{a}teaux-differentiable and we obtain optimal solutions $\overline{u}_\varepsilon$, $\overline{y}_\varepsilon = G_\varepsilon(B_i\overline{u}_\varepsilon)$ and $\overline{z}_\varepsilon = Z_\varepsilon(S\overline{y}_\varepsilon)$ of the regularized problems.
We investigate in the limit $\varepsilon\rightarrow 0$ and use standard arguments to derive
% a subsequence
%$(\overline{u}_{\varepsilon_k},\overline{y}_{\varepsilon_k},\overline{z}_{\varepsilon_k})$ which converges to 
a solution $(\overline{u},\overline{y},\overline{z})$ of the original problem.
% as $k\rightarrow \infty$.
It still remains difficult to derive adjoint systems $(p_\varepsilon,q_\varepsilon)$ already for the regularized problems.
The main result of Section~3 is Theorem~\ref{Thm:opt_syst_reg} which contains the evolution equations of $p_\varepsilon$ and $q_\varepsilon$ and the adjoint equation which provides a relation between $(p_\varepsilon,q_\varepsilon)$ and $\overline{u}_\varepsilon$ and $\overline{u}$.

In Section~4, we perform the key step towards an optimality system of \eqref{opt_control_ control_problem}-\eqref{state_equ_z} by driving the regularization parameter to zero. We exploit the adjoint systems $(p_\varepsilon,q_\varepsilon)$ to derive necessary optimality conditions for problem \eqref{opt_control_ control_problem}-\eqref{state_equ_z}.
While the evolution equation for $p$
%$
%=\lim\limits_{k\rightarrow \infty}p_{\varepsilon_k}$ 
follows rather straight forward, the adjoint variable $q$
%=\lim\limits_{k\rightarrow \infty}q_{\varepsilon_k}$ 
which belongs to $\overline{z}$ has lower regularity, similar as in optimal control problems with implicit state constraints of the form of variational inequalities.
The function $q$ is contained in the space $\mathrm{BV}(0,T)$ of functions with bounded total variation in $[0,T]$, and instead of a time derivative we obtain a measure $dq\in \mathrm{C}([0,T])^*$.
%We can show uniform-in-$\varepsilon$ for $\dot{q}_\varepsilon$ only in the space $\mathrm{L}^1(0,T)$. 
%Consequently, we have to pass to weak-$*$ convergence of $\dot{q}_{\varepsilon_k}$ in $\mathrm{C}([0,T])^*$ and obtain only a limit measure $dq$ instead of a weak derivative of $q$. The limit $q$ itself is shown to be contained in the space $\mathrm{BV}(0,T)$ of functions with bounded total variation in $[0,T]$.
In order to complete our knowledge about the optimality system, we investigate in studying $q$ and $dq$. Indeed, we reveal a lot of the properties of $q$ and the corresponding measure $dq$. There remains an abstract measure $d\mu\in \mathrm{C}([0,T])^*$ on which $dq$ depends and which we cannot fully characterize. Moreover, $d\mu$ appears in the optimality conditions for problem \eqref{opt_control_ control_problem}-\eqref{state_equ_z}. Still, we are able to prove that $d\mu$ has its support only in a part of $[0,T]$. 
With an additional regularity assumption on $S\overline{y}$, we can characterize the measure $d\mu$ also in most of the parts where it does not vanish. 
The first main results of Section~4 are Theorem~\ref{Thm:opt_control_boundary_opt_cond_limit_adjoint} and Corollary~\ref{Cor:opt_cond_general}, which contain the existence of an adjoint system and optimality conditions for problem \eqref{opt_control_ control_problem}-\eqref{state_equ_z} for $i\in \{1,2\}$.
After having established the optimality system for the general problem \eqref{opt_control_ control_problem}-\eqref{state_equ_z}, $i\in\{1,2\}$, we continue to improve the optimality conditions for the particular case of distributed control functions, i.e. for $i=1$, see Corollary~\ref{Cor:opt_control_distrib_opt_cond} below.
Moreover, in Corollary~\ref{Cor:Uniqueness_dist_contr}. we show uniqueness of $p$, $q$ and $d\mu$ for $i=1$. In the we make explicit use of the surjectivity of $B_1$ which implies that the operator $B_1^*$ in the adjoint equation is one-to-one. 
These together are the second main result of Section~4.
%Corollary~\ref{Cor:opt_control_distrib_opt_cond} contains our results about the improved optimality conditions for the problem with $i=1$. Uniqueness of the adjoint system for $i=1$ is shown in Corollary~\ref{Cor:Uniqueness_dist_contr}.
%These together are the second main result of Section~4.

In Section~5,
we prove higher regularity of the optimal control $\overline{u}$ and the optimal state $\overline{y}$ by means of the adjoint equation and the continuity properties of the adjoint variables, see Theorem~\ref{Thm:HigherRegularity} below. An example for a case in which Theorem~\ref{Thm:HigherRegularity} can be applied is given in Remark~\ref{Rem:Example_higher_regularity}.

Finally, in Section~6 we study a perturbed problem similar to \eqref{opt_control_ control_problem}-\eqref{state_equ_z}. In particular, in Theorem~\ref{Thm:ValueFunction} we prove regularity results for the corresponding value function.

All our results are applicable for more general spaces of control functions $U=\mathrm{L}^2\bigl((0,T);\tilde{U}\bigr)$, as long as there exists a continuous operator $B:\tilde{U} \rightarrow \mathbb{W}_{\Gamma_D}^{-1,p}(\Omega)$. Also $J(y,u)$ can be exchanged by a general differentiable functional $J(y,u,z)$ if the corresponding reduced cost function remains coercive in $u\in U$.
Moreover, $A_p$ can be replaced by a semi-linear parabolic operator which satisfies maximal parabolic regularity on the space $\mathbb{W}_{\Gamma_D}^{-1,p}(\Omega)$.
We focus on the two particular control problems for $U_1$ and $U_2$ and on the operator $A_p$ in order to give an illustration.

\underline{Notation:}\\
We write $\mathcal{L}(X,Y)$ for the space of linear operators between spaces $X$ and $Y$ and $\mathcal{L}(X)$ for the space of linear operators on $X$.
We also abbreviate the duality in $X$ by
$
\langle x,y \rangle_{X^*,X} =
\langle x,y \rangle_{X}.
$
$c>0$ denotes a generic constant which is adapted in the course of the paper.
In Banach space valued evolution equations like \eqref{state_equ_y} we sometimes omit the range space if the latter is clear from the context, i.e. we only write "for $t\in (0,T)$".

\section{Preliminaries and assumptions}
We introduce the setting for the rest of the work, collect results from the literature and state the main assumption. 

\subsection{Sobolev spaces including homogeneous Dirichlet boundary conditions}
\begin{definition}
	\cite[Definition 2.1][$I$-sets]{muench}
	For $0<I\leq d$ and a closed set $M\subset \mathbb{R}^d$ let $\rho$ denote the restriction of the $I$-dimensional Hausdorff measure $\mathcal{H}_I$ to $M$. Then we call $M$ an $I$-set if there are constants $c_1,c_2>0$ such that
	\begin{equation*}
	c_1 r^I \leq \rho\left( B_{\mathbb{R}^d}(x,r)\cap M \right) \leq c_2 r^I
	\end{equation*}
	for all $x$ in $M$ and $r\in ]0,1[$.
\end{definition}

\begin{assumption}[Domain]\label{Ass:domain}
	\cite[Assumption 2.3 and Assumption 4.11]{rehbergsystems} or \cite[Assumption 2.2 and Assumption 2.6]{muench}
	For some given $d\geq 2$, the domain $\Omega\subset\mathbb{R}^d$ is bounded and $\overline{\Omega}$ is a $d$-set.
	For $j\in \{1,\ldots,m\}$ the Neumann boundary part $\Gamma_{N_j}\subset \partial\Omega$ is open and $\Gamma_{D_j}=\partial\Omega\backslash\Gamma_{N_j}$ is a $(d-1)$-set.
	%\end{assumption}
	%We need the following assumption for each of the $m$ components \cite[cf.][Assumption 4.11]{rehbergsystems}:
	%\begin{assumption}\label{Ass:existence_extension}
	%	In the setting of Assumption \ref{Ass:domain} we suppose f
	For any $x\in \overline{\Gamma_{N_j}}$ there is an open neighborhood $U_x$ of $x$ and a bi-Lipschitz mapping $\phi_x$ from $U_x$ onto a cube in $\mathbb{R}^d$ such that $\phi_x(\Omega\cap U_x)$ equals the lower half of the cube and such that $\partial\Omega\cap U_x$ is mapped onto the top surface of the lower half cube.
\end{assumption}

%The setting is the same as in \cite[Section 2]{muench}. All Sobolev spaces are defined on a bounded domain $\Omega\subset \mathbb{R}^d$ with $d\geq 2$. 
%The boundary regularity is defined in Assumption~\ref{Ass:domain} below.
We only consider real valued functions.
For each component $j\in \{1,\ldots, m\}$ of the space of vector valued functions, see Definition~\ref{Def:Sobolev_space}, we decompose the boundary $\partial\Omega$ into the corresponding Dirichlet part $\Gamma_{D_j}$ and the Neumann boundary $\Gamma_{N_j}:=\partial\Omega\backslash \Gamma_{D_j}$, see Assumption~\ref{Ass:domain}.
The cases $\Gamma_{D_j}=\emptyset$ and $\Gamma_{D_j}=\partial\Omega$ are not excluded.
%The assumed condition on $\Gamma_{D_j}$ requires the definition of an I-set \cite[cf.][Definition 2.1]{rehbergsystems}.
%\begin{definition}
%	For $0<I\leq d$ and a closed set $M\subset \mathbb{R}^d$ let $\rho$ denote the restriction of the $I$-dimensional Hausdorff measure $\mathcal{H}_I$ to $M$. Then we call $M$ an $I$-set if there are constants $c_1,c_2>0$ such that
%	\begin{align*}
%	c_1 r^I \leq \rho\left( B_{\mathbb{R}^d}(x,r)\cap M \right) \leq c_2 r^I
%	\end{align*}
%	for all $x$ in $M$ and $r\in ]0,1[$.
%\end{definition}
%The assumption on the domain in our setting is based on the following definition.

We define Sobolev spaces which include the Dirichlet boundary conditions for our state equation.
\begin{definition}[Sobolev spaces]\label{Def:Sobolev_space}
	\cite[Definition 2.4]{rehbergsystems} or \cite[Definition 2.4]{muench}
	For $\Omega$ from Assumption~\ref{Ass:domain} and $p\in [1,\infty)$ we denote by 
	$\mathrm{W}^{1,p}(\Omega)$ 
	the usual Sobolev space on $\Omega$.
	If $M$ is a closed subset of $\overline{\Omega}$ we define \begin{equation*}
	\mathrm{W}_\mathrm{M}^{1,p}(\Omega) := \overline{\lbrace \psi|_\Omega:\, \psi\in\mathrm{C}_0^\infty(\mathbb{R}^d),\, \mathrm{supp}(\psi)\cap \mathrm{M}=\emptyset \rbrace},
	\end{equation*} 
	where the closure is taken in the space $\mathrm{W}^{1,p}(\Omega)$.
		In the case $p\in(1,\infty)$ we denote by $p'$ the H{\"o}lder conjugate of $p$.
	Moreover, we write
	\begin{equation*}
	\mathrm{W}_\mathrm{M}^{-1,p}(\Omega) := \left[\mathrm{W}_\mathrm{M}^{1,p'}(\Omega)\right]^*
	\end{equation*}
	for the dual space $\mathrm{W}_\mathrm{M}^{1,p'}(\Omega)$.
	In the vectorial setting we introduce the product space
	\begin{equation*}
	\mathbb{W}_{\Gamma_D}^{1,p}(\Omega):= \prod\limits_{j=1}^m \mathrm{W}_{\Gamma_{D_j}}^{1,p}(\Omega)
	\end{equation*}
	and for $p\in(1,\infty)$ we denote by $\mathbb{W}_{\Gamma_D}^{-1,p}(\Omega)$ the (componentwise) dual space of   
	$
	\mathbb{W}_{\Gamma_D}^{1,p'}(\Omega)
	$.
\end{definition}
%\begin{remark}\label{Rem:embedding_w1p_lp}
%	Under Assumption~\ref{Ass:existence_extension} it can be shown that the embeddings
%	
%	$\mathrm{W}_{\Gamma_{D_j}}^{1,p}(\Omega)\hookrightarrow \mathrm{L}^p(\Omega)$ are compact \cite[cf.][Remark 3.2]{rehbergsystems}. The proof is almost equal to the proofs of \cite[Part II,\, 5.6.1,\, Theorem 2]{evans} and \cite[Part II,\,5.7,\, Theorem 1]{evans}.
%\end{remark}

\subsection{Operators and their properties}\label{sec:elliptic-regularity-for-systems-and-fractional-power-spaces}
In this section we precisely define the operators $A_p$ in equation~\eqref{state_equ_y}, see Definition~\ref{Def:vector_Sobolev_space}. We apply results from the literature to assure that $A_p$ satisfies the properties which we need for the analysis of \eqref{state_equ_y}-\eqref{state_equ_z} for particular values of $p$ to be chosen, see \cite[Section 6]{rehbergsystems} or \cite[Subsection 2.2]{muench}.
\begin{definition}[Diffusion operator]\label{Def:vector_Sobolev_space}
%	With Assumption \ref{Ass:domain} and Assumption \ref{Ass:existence_extension} and $p\in [1,\infty[$ we define the Sobolev space of vector valued functions by the product space
%	\begin{align*}
%	\mathbb{W}_{\Gamma_D}^{1,p}(\Omega):= \prod\limits_{j=1}^m \mathrm{W}_{\Gamma_{D_j}}^{1,p}(\Omega)
%	\end{align*}
%	and denote its dual by $\mathbb{W}_{\Gamma_D}^{-1,p'}(\Omega)$. 
%	
	For $p\in(1,\infty)$ we define the continuous operators 
	\begin{equation*}
	\begin{aligned}
	\mathcal{L}_p:\mathbb{W}_{\Gamma_D}^{1,p}(\Omega)=\prod\limits_{j=1}^m \mathrm{W}_{\Gamma_{D_j}}^{1,p}(\Omega)\rightarrow \mathrm{L}^p(\Omega, \mathbb{R}^{md}),&& \mathcal{L}_p(u):= \mathrm{vec}(\nabla u)=(\nabla u_1,\ldots, \nabla u_m)^\intercal
	\end{aligned}
	\end{equation*}
	and
\begin{equation*}
	\begin{aligned}
	I_p:\mathbb{W}_{\Gamma_D}^{1,p}(\Omega)\rightarrow \mathbb{W}_{\Gamma_D}^{-1,p}(\Omega),
	&& \langle I_p u,v \rangle_{\mathbb{W}_{\Gamma_D}^{1,p'}(\Omega)}:=\int_\Omega u\cdot v\, dx
	&& \forall v\in \mathbb{W}_{\Gamma_D}^{1,p'}(\Omega).
	\end{aligned}
\end{equation*}
	With given diffusion coefficients $d_1,\ldots, d_m > 0$ we define the corresponding diffusion matrix in $\mathbb{R}^{md\times md}$ by
	$D=\mathrm{diag}(d_1,\ldots, d_1,\ldots, d_m,\ldots, d_m)$.

	For $p\in(1,\infty)$ we set
	\begin{equation*}
	\begin{aligned}
	\mathcal{A}_p: \mathbb{W}_{\Gamma_D}^{1,p}(\Omega) \rightarrow \mathbb{W}_{\Gamma_D}^{-1,p}(\Omega),
	&& \mathcal{A}_p:= \mathcal{L}_{p'}^* D \mathcal{L}_p
	\end{aligned}
	\end{equation*}
	and define the unbounded operator
	
	\begin{equation*}
	\begin{aligned}
	A_p: \mathrm{dom}(A_p) = \mathrm{ran}\left(I_p\right)\subset\mathbb{W}_{\Gamma_D}^{-1,p}(\Omega) \rightarrow \mathbb{W}_{\Gamma_D}^{-1,p}(\Omega),
	&& A_p:=\mathcal{A}_p I_p^{-1}.
	\end{aligned}
	\end{equation*}
	The set $\mathrm{ran}\left(I_p\right)$ stands for the range of $I_p$. The domain $\mathrm{dom}(A_p)$ is equipped with the graph norm.
%	If $\theta \geq 0$ is given then the fractional power spaces $X^\theta:= \mathrm{dom}([A_p + 1]^\theta)\subset \mathbb{W}_{\Gamma_D}^{-1,p}(\Omega)$ and the unbounded operators $[A_p + 1]^\theta$ are well-defined \cite[cf.][Chapter 1]{henry} with $X^0=\mathbb{W}_{\Gamma_D}^{-1,p}(\Omega)$.
%	$X^\theta$ is equipped with the norm
%	$\|x\|_{X^\theta}= \|(A_p+1)^\theta x\|_{\mathbb{W}_{\Gamma_D}^{-1,p}(\Omega)}.$
\end{definition}

%The following result is shown in \cite[Theorem 5.6 and Theorem 5.12]{rehbergsystems}:
%\begin{theorem}\label{Thm:elliptic_regularity_for_systems}
%	In the setting of Definition~\ref{Def:vector_Sobolev_space} and Definition~\ref{Def:A_p} there exists an open interval $\mathrm{J}$ around $2$ such that for all $p\in \mathrm{J}$ the operators $\mathcal{A}_p + I_p$ are topological isomorphisms between $\mathbb{W}_{\Gamma_D}^{1,p}(\Omega)$ and $\mathbb{W}_{\Gamma_D}^{-1,p}(\Omega)$.
%	
%	There is a constant $c>0$ such that for all $p\in \mathrm{J}$ and $\lambda\in \mathbb{C}_+:=\lbrace z\in \mathbb{C}: \mathrm{Re}z\geq 0 \rbrace$ the estimate 
%	\begin{align*}
%	\Vert (A_p+1+\lambda)^{-1} \Vert_{\mathcal{L}(\mathbb{W}_{\Gamma_D}^{-1,p}(\Omega))} \leq  \frac{c}{1+|\lambda|}
%	\end{align*}
%	holds true and $-A_p$ generates an analytic semigroup of operators on $\mathbb{W}_{\Gamma_D}^{-1,p}(\Omega)$.
%\end{theorem}

We introduce the notion of maximal parabolic regularity as in \cite[Definition 2.12]{muench}.
\begin{definition}[Maximal parabolic regularity]\label{Def:maximal parabolic regularity}
	For $p,q\in (1,\infty)$ and $(t_0,T)\subset \mathbb{R}$ we say that $A_p$ satisfies maximal parabolic $\mathrm{L}^{q}((t_0,T);\mathbb{W}_{\Gamma_D}^{-1,p}(\Omega))$-regularity if for all
	$g\in \mathrm{L}^q\left((t_0,T);\mathbb{W}_{\Gamma_D}^{-1,p}(\Omega)\right)$ there is a unique solution $y\in\mathrm{W}^{1,q}((t_0,T);\mathbb{W}_{\Gamma_D}^{-1,p}(\Omega))\cap \mathrm{L}^{q}((t_0,T);\mathrm{dom}(A_p))$ of the equation
	\begin{equation*}
	\dot{y}+A_p y = g,\ y(t_0)=0.
	\end{equation*}
	The time derivative is taken in the sense of distributions \cite[Definition 11.2]{squareroot_problem}.

	For $t\in [0,T]$ we abbreviate

	$
	Y_q:= \mathrm{W}^{1,q}((0,T);\mathbb{W}_{\Gamma_D}^{-1,p}(\Omega))\cap \mathrm{L}^{q}((0,T);\mathrm{dom}(A_p))$,
	$
	Y_{q,t}:= \{y\in Y_q:\ y(t)=0\}
	$
	and 
	
	$
	Y^*_{q,t}:= \{ y\in \mathrm{W}^{1,q}(0,T;[\mathrm{dom}(A_p)]^*)\cap \mathrm{L}^{q}((0,T); \mathbb{W}_{\Gamma_D}^{1,p'}(\Omega)):\ y(t)=0\}
	$.
\end{definition}

As in \cite[Remark 2.13]{muench} note the following:
\begin{remark}[Properties of $A_p$]\label{Rem:maximal parabolic regularity}
	\leavevmode
	\begin{enumerate}
		\item
		If Definition~\ref{Def:maximal parabolic regularity} applies for $A_p$ with some $p\in(1,\infty)$ then the property of maximal parabolic regularity is independent of
		$q\in (1,\infty)$ and of the interval $(t_0,T)$, so  we just say that $A_p$ satisfies maximal parabolic regularity on $\mathbb{W}_{\Gamma_D}^{-1,p}(\Omega)$ in this case. 
%		\cite[cf.][Remark 11.3]{squareroot_problem}.
		\item
		If $A_p$ satisfies maximal parabolic regularity on $\mathbb{W}_{\Gamma_D}^{-1,p}(\Omega)$ for some $p\in(1,\infty)$ then
		$(\frac{d}{dt}+A_p)^{-1}$
		is bounded from $\mathrm{L}^q((0,T);\mathbb{W}_{\Gamma_D}^{-1,p}(\Omega))$ to $Y_{q,0}$ for any $q\in (1,\infty)$.
%		 \cite[cf.][Proof of Proposition 2.8]{meyeroptimal}.
		\item In the setting of Assumption~\ref{Ass:domain} there is an open interval $\mathrm{J}$ containing $2$ such that for $p\in \mathrm{J}$ the operator $\mathcal{A}_p + I_p$ is a topological isomorphism and such that $-A_p$ generates an analytic semigroup of operators on $\mathbb{W}_{\Gamma_D}^{-1,p}(\Omega)$
		 \cite[Theorem 2.10]{muench} or \cite[Theorem 5.6 and Theorem 5.12]{rehbergsystems}.
		\item
		If $p\in \mathrm{J}$ and if $\theta \geq 0$ is given then for $A_p + 1:=A_p + \mathrm{Id}$ the fractional power spaces $X^\theta:= \mathrm{dom}([A_p + 1]^\theta)\subset \mathbb{W}_{\Gamma_D}^{-1,p}(\Omega)$ and the unbounded operators $[A_p + 1]^\theta$ in the sense of \cite[Chapter 1]{henry} are well-defined with $X^0=\mathbb{W}_{\Gamma_D}^{-1,p}(\Omega)$.
		$X^\theta$ is equipped with the norm
		$\|x\|_{X^\theta}= \|(A_p+1)^\theta x\|_{\mathbb{W}_{\Gamma_D}^{-1,p}(\Omega)}$ \cite[cf.][Remark 2.11]{muench}.
		Note that we can identify $X^1$ with the space $\mathrm{dom}(A_p)$ endowed with the graph norm. 
%		then $A_p +1$ has bounded imaginary powers and satisfies maximal parabolic Sobolev regularity on $\mathbb{W}_{\Gamma_D}^{-1,p}(\Omega)$ \cite[Theorem 11.5]{squareroot_problem}.
		\item If $p\in \mathrm{J}\cap[2,\infty)$, then $A_p$ satisfies maximal parabolic Sobolev regularity on $\mathbb{W}_{\Gamma_D}^{-1,p}(\Omega)$ and we have the topological equivalences 	
		$
		[\mathbb{W}_{\Gamma_D}^{-1,p}(\Omega), \mathbb{W}_{\Gamma_D}^{1,p}(\Omega)]_{\theta} \simeq [\mathbb{W}_{\Gamma_D}^{-1,p}(\Omega), \mathrm{dom}(A_p)]_{\theta} \simeq X^\theta
		$
		for all $\theta \in (0,1)$ \cite[Theorem 11.6.1]{carracedo}.
		By $[\cdot, \cdot]_{\theta}$ we mean complex interpolation.
%		For the notion of bounded imaginary powers see \cite[Remark 2.10]{muench} for example. We will not need the exact definition for the rest of the work, but only apply the conclusion which allows us to identify fractional power spaces with complex interpolation spaces. 
	\end{enumerate}

\end{remark}

We will make use of the following embeddings:
\begin{remark}[Embeddings]\label{Rem:embeddings}
	\cite[Remark 2.14]{muench}
	With $q\in (1,\infty)$ one has
	\begin{align*}
	Y_{q} &\hhookrightarrow \mathrm{C}^\beta((0,T); (\mathbb{W}_{\mathrm{\Gamma_D}}^{-1,p}(\Omega), \mathrm{dom}(A_p))_{\eta,1})
	\hookrightarrow \mathrm{C}^{\beta}((0,T); [\mathbb{W}_{\mathrm{\Gamma_D}}^{-1,p}(\Omega), \mathrm{dom}(A_p)]_{\theta})
	\text{ and }\\
	Y_q &\hhookrightarrow \mathrm{C}([0,T]; (\mathbb{W}_{\mathrm{\Gamma_D}}^{-1,p}(\Omega), \mathrm{dom}(A_p))_{\eta,q})
	\hookrightarrow \mathrm{C}([0,T]; [\mathbb{W}_{\mathrm{\Gamma_D}}^{-1,p}(\Omega), \mathrm{dom}(A_p)]_{\theta})
	\end{align*}
	for every $0<\theta < \eta < 1-1/q$ and $0\leq\beta < 1-1/q -\eta$. $(\cdot,\cdot)_{\eta,1}$ or $(\cdot,\cdot)_{\eta,q}$ respectively means real interpolation. The first embeddings are compact because $\mathrm{dom}(A_p)$ is compactly embedded into $\mathbb{W}_{\mathrm{\Gamma_D}}^{-1,p}(\Omega)$.
	
	%	Moreover it holds
	%	\begin{align*}
	%	Y_q &\hookrightarrow \mathrm{L}^{q/(1-qs)}((0,T); (\mathbb{W}_{\mathrm{\Gamma_D}}^{-1,p}(\Omega), \mathrm{dom}(A_p))_{\eta,1})\\
	%	&\hookrightarrow \mathrm{L}^{q/(1-qs)}((0,T); [\mathbb{W}_{\mathrm{\Gamma_D}}^{-1,p}(\Omega), \mathrm{dom}(A_p)]_{\theta})
	%	\end{align*}
	%	for all $0<\theta < \eta <1-s$ and $0<s<1/p$ and also here the first embedding is compact.
\end{remark}
With $p\in \mathrm{J}$, the following estimate for the fractional powers of $A_p+1$ and the analytic semigroup $\exp(-A_p t)$ is crucial:
%We collect several estimates for the operator $(A_p+1)^\alpha$ and the semigroup $\exp(-A_p t)$.
\begin{remark}
	\cite[Remark 2.15]{muench}
	Let $p\in \mathrm{J}$ with $\mathrm{J}$ from Remark~\ref{Rem:maximal parabolic regularity}. For $t>0$ 
%	it is shown in \cite[Theorem 1.3.4]{henry} that for some $C>0$ one can estimate 
%	\begin{align*}
%	&\Vert \exp(-A_p t) \Vert_{\mathcal{L}(\mathbb{W}_{\mathrm{\Gamma_D}}^{-1,p}(\Omega))}\leq C \exp((1-\gamma) t)
%	\text{ and }\\
%	&\Vert (A_p+1)\exp(-A_p t) \Vert_{\mathcal{L}(\mathbb{W}_{\mathrm{\Gamma_D}}^{-1,p}(\Omega))}\leq \frac{C}{t} \exp((1-\gamma) t).
%	\end{align*}
%	
%	Moreover for 
	and arbitrary $\gamma\in (0,1)$ and $\theta\geq 0$
%	, according to \cite[Theorem 1.4.3.]{henry},
	there exists some 
	$C_\theta\in (0,\infty)$ such that
	\begin{equation}\label{frac_pow_estimate1}
	\Vert (A_p+1)^{\theta}\exp(-A_p t) \Vert_{\mathcal{L}\bigl(\mathbb{W}_{\Gamma_D}^{-1,p}(\Omega)\bigr)}\leq C_\theta t^{-\theta}\exp((1-\gamma) t).
	\end{equation}
	%	and for $0<\theta\leq 1$ and $x\in X^\theta$
	%	\begin{align}\label{frac_pow_estimate2}
	%	\Vert (\exp(-(A_p+1) t)-\mathrm{Id})x \Vert_{\mathcal{L}(\mathbb{W}_{\Gamma_D}^{-1,p}(\Omega))}\leq \frac{1}{\theta} C_{1-\theta} t^{\theta}\Vert (A_p+1)^{\theta}x \Vert_{\mathbb{W}_{\Gamma_D}^{-1,p}(\Omega)}.
	%	\end{align}
%	The constants $C_\theta$ are bounded if $\theta$ is contained in any compact subinterval of $]0,\infty[$ and also for $\theta\downarrow 0$.
\end{remark}

The stop operator has the following regularity properties.
%according to \cite[Subsection 2.4 and Subsection 4.2]{muench}, see also \cite[Part 1, Chapter III]{visintin2013differential} and \cite{Brokate_weak_diff}:
\begin{lemma}[Stop operator]\label{Lem:hyst_props}
	With $T>0$ the stop operator $\mathcal{W}$, which is represented by \eqref{stop1}-\eqref{stop2}, is Lipschitz continuous as a mapping on $\mathrm{C}[0,T]$ and
	\begin{equation}
	\begin{aligned}
	|\mathcal{W}[v_1](t)- \mathcal{W}[v_2](t)| \leq 2\sup\limits_{0\leq \tau \leq t} \vert v_1(\tau) -v_2(\tau) \vert
	&& \text{and}
	&&|\mathcal{W}[v](t)| \leq 2\sup\limits_{0\leq \tau \leq t} \vert v(\tau) \vert + |z_0|
	\end{aligned}\label{hyst_growth}
	\end{equation}
	for all $v,v_1,v_2 \in \mathrm{C}[0,T]$ and $t\in [0,T]$. Note that we have to add $|z_0|$ in \eqref{hyst_growth} because, by \eqref{stop2}, $\mathcal{W}[v](0)=z_0$ for any $v \in \mathrm{C}[0,T]$.
	For $q\in [1,\infty)$ it is also bounded and weakly continuous on $\mathrm{W}^{1,q}(0,T)$.	
	$\mathcal{W}:\mathrm{C}[0,T]\rightarrow \mathrm{L}^q(0,T)$ is Hadamard directionally differentiable, see Definition~\ref{Def:Hadam_differentiability} below. The same regularity properties hold for the operator $\mathcal{P} = \mathrm{Id} - \mathcal{W}$. $\mathcal{P}$ is a scalar play operator. More precisely, for $r=\frac{b-a}{2}$ let $\mathcal{P}_r:\mathrm{C}[0,T]\times\mathbb{R} \rightarrow \mathrm{C}[0,T]$ denote a symmetrical scalar play operator (as in \cite{Brokate_weak_diff}). 
	Consider the affine linear transformation
	$
	\mathcal{T}:[-r,r]\rightarrow [a,b],\ \mathcal{T}:x\mapsto x - \frac{b+a}{2}.
	$
	Then for $v\in \mathrm{C}[0,T]$ there holds
	\begin{equation*}
	\mathcal{P}[v] = \mathcal{P}_r[\mathcal{T}(v), v(0)-z_0] \in \mathrm{C}[0,T].
	\end{equation*}
\end{lemma}
\begin{proof}
	Follows from \cite[Subsection 2.4 and Subsection 4.2]{muench}, see also \cite[Part 1, Chapter III]{visintin2013differential} and \cite{Brokate_weak_diff}.
\end{proof}

\subsection{Assumptions and notation}
Our main assumption is the following:
\begin{assumption}[Main assumption]\label{Ass:general_ass_and_short_notation_1}	
	\cite[Assumptions 2.16, 4.6 and 5.1]{muench}
	We always suppose that Assumption~\ref{Ass:domain} holds.
	Moreover we assume:
	\begin{itemize}
		\item[(A1)] Dimension and Sobolev exponent: $d\geq 2$ and with $\mathrm{J}$ from Remark~\ref{Rem:maximal parabolic regularity} there holds $p\in \mathrm{J}\cap [2,\infty)$ and $2\geq p\left(1-\frac{1}{d}\right)$.	
		
		\item[(A2)] Nonlinearity locally Lipschitz + Hadamard: We will need a fractional power space $X^\alpha=\mathrm{dom}([A_p+1]^\alpha)$ with exponent strictly smaller than one half. This fact is highlighted by a new parameter $\alpha$ which we use instead of $\theta\in [0,\infty)$. For some $\alpha\in (0,\frac{1}{2})$ suppose that the function 
		$f:X^\alpha\times \mathbb{R}\rightarrow \mathbb{W}_{\Gamma_D}^{-1,p}(\Omega)$ is locally Lipschitz continuous with respect to the $X^\alpha$-norm.
		This means that given any $y_0\in X^\alpha$ there is a constant $L(y_0)$ and a neighbourhood 
		$
		V(y_0)=\left\lbrace y\in X^\alpha: \Vert y-y_0 \Vert_{X^\alpha}\leq \delta \in (0,\infty) \right\rbrace
		$ of $y_0$ such that
		\begin{equation*}
		\Vert f(y_1,x_1) - f(y_2,x_2) \Vert_{\mathbb{W}_{\Gamma_D}^{-1,p}(\Omega)} \leq L(y_0) \left( \Vert y_1-y_2 \Vert_{\alpha} + \vert x_1-x_2 \vert \right)
		\end{equation*}
		for every $y_1,y_2 \in V(y_0)$ and all $x_1,x_2\in \mathbb{R}$.
		$f$ is assumed to be directionally differentiable and therefore Hadamard directionally differentiable, see Definition~\ref{Def:Hadam_differentiability} below. Furthermore, the linear growth condition
		\begin{equation*}
		\Vert f(y,x) \Vert_{\mathbb{W}_{\Gamma_D}^{-1,p}(\Omega)} \leq M \left( 1+ \Vert y \Vert_{\alpha} + \vert x \vert \right)
		\end{equation*}
		holds for some $M>0$.
		
		\item[(A3)] Scalar projection: For some $w\in\mathbb{W}_{\Gamma_D}^{1,p'}(\Omega)\backslash\{0\}$ the operator $S\in [\mathbb{W}_{\Gamma_D}^{-1,p}(\Omega)]^*$ in equation~\eqref{state_equ_z} is given by
		$
		S y = \langle y,w \rangle_{\mathbb{W}_{\Gamma_D}^{1,p'}(\Omega)}\ \forall y\in\mathbb{W}_{\Gamma_D}^{-1,p}(\Omega).
		$
		%		\item $1>\alpha > \frac{1}{2} + \frac{1}{p}$.
		We assume that $w$ is even contained in the space $\mathrm{dom}([(1+A_p)^{1-\alpha}]^*)$. Note that $S$ belongs to $[X^\theta]^*$ for all $\theta\geq 0$ because of the embedding 
		$X^\theta\hookrightarrow \mathbb{W}_{\Gamma_D}^{-1,p}(\Omega)$.
		
		\item[(A4)] Desired state: The desired state $y_d$ in \eqref{opt_control_ control_problem} is in $\mathrm{L}^2 \left((0,T);[\mathrm{L}^2(\Omega)]^m\right)$.
	\end{itemize}
\end{assumption}
%We need one more assumption:
%\begin{assumption}\label{Ass:general_ass_and_short_notation}
%	In addition to Assumption~\ref{Ass:general_ass_and_short_notation_1} we suppose that
%%	\begin{itemize}
%%		\item 
%%		The space $Y^*_{q,t}$ is defined as in Definition~\ref{Def:maximal parabolic regularity}.
%		$w$ is even contained in $\mathrm{dom}([A_p^{1-\alpha}]^*)$.
%\end{assumption}

We introduce some more notation for the rest of the work:
\begin{itemize}
	\item[(N1)]
	For the particular $p$ from (A1) in Assumption~\ref{Ass:general_ass_and_short_notation_1} we set
	$X:=\mathbb{W}_{\Gamma_D}^{-1,p}(\Omega)$
	with $\mathbb{W}_{\Gamma_D}^{-1,p}(\Omega)$ from Definition~\ref{Def:Sobolev_space}.
	We sometimes identify elements $v\in X^*$ with their Riesz representation in $\mathbb{W}_{\Gamma_D}^{1,p'}(\Omega)$, i.e.
	$
	\langle v,y \rangle_{X}= \langle y,v \rangle_{\mathbb{W}_{\Gamma_D}^{1,p'}(\Omega)},\ \forall y\in X.
	$
	\item[(N2)] The operators $A_p$ and the spaces $X^\theta = \mathrm{dom}([A_p+1]^\theta)$ are defined as in Definition~\ref{Def:vector_Sobolev_space} and Remark~\ref{Rem:maximal parabolic regularity}.
	\item[(N3)] 
			The spaces $Y_q$, $Y_{q,t}$ and $Y^*_{q,t}$ are defined as in Definition~\ref{Def:maximal parabolic regularity}.
	\item[(N4)] $\mathcal{W}$ is the scalar stop operator from Lemma~\ref{Lem:hyst_props}.
	\item[(N5)] $B_1$ is defined by
	$
	B_1:[\mathrm{L}^2(\Omega)]^m\rightarrow X,\ \langle B_1 u,v \rangle_{\mathbb{W}_{\Gamma_D}^{1,p'}(\Omega)}:=\int_\Omega u\cdot v\, dx\ v\in\mathbb{W}_{\Gamma_D}^{1,p'}(\Omega).
	$
	
	Since $2\geq p\left(1-\frac{1}{d}\right)$, the embeddings 
	$\mathrm{L}^2(\Gamma_{N_j},\mathcal{H}_{d-1})\hookrightarrow \mathrm{W}_{\mathrm{\Gamma_D}_j}^{-1,p}(\Omega)$ are continuous for $j\in \{1,\ldots,m\}$ \cite[Remark 5.11]{rehbergsystems}.
	Therefore also the operator
	
	$B_2:\prod_{j=1}^m \mathrm{L}^2(\Gamma_{N_j},\mathcal{H}_{d-1})\rightarrow X,\ \langle B_2y,v \rangle_{\mathbb{W}^{1,p'}(\Omega)}:=\sum_{j=1}^m \int_{\Gamma_{N_j}}y_jv_j\, d\mathcal{H}_{d-1}\ \forall v\in\mathbb{W}_{\Gamma_D}^{1,p'}(\Omega)$
	is continuous.
	
	\item[(N6)] We write $J_T=(0,T)$,	$
	U_1= \mathrm{L}^2 \left(J_T;[\mathrm{L}^2(\Omega)]^m\right)$
	and
	$
	U_2= \mathrm{L}^2 \left(J_T; \prod_{j=1}^m \mathrm{L}^2(\Gamma_{N_j},\mathcal{H}_{d-1})\right).
	$
	
\end{itemize}

%\begin{assumption}\label{Ass:hadamard_hyst}
%	For arbitrary $T>0$, the stop operator $\mathcal{W}$ defined on the interval $[a,b]$ and for initial value $0$ is Lipschitz continuous on $\mathrm{C}(\overline{J_T})$.
%	
%	For all $q\in [1,\infty)$, the stop operator $\mathcal{W}:\mathrm{C}([t_0,T])\rightarrow \mathrm{L}^q(t_0,T)$ is Lipschitz continuous with some modulus $L_\mathcal{W}>0$ and directionally differentiable, and therefore Hadamard directionally differentiable  \cite[cf.][Proposition 5.5]{Brokate_weak_diff}.
%\end{assumption}

\subsection{Solution operator and optimal control}\label{Subsec:Sol_op_optimal_control}
As in \cite[Equation (1)]{muench} we denote $F[y](t):=f(y(t),\mathcal{W}[S y](t))$ and introduce the more general abstract evolution equation 
%corresponding to the system \eqref{state_equ_y}-\eqref{state_equ_z} 
\begin{equation}
\begin{aligned}
\dot{y}(t) + (A_p y)(t) &= (F[y])(t) + u(t)&&  \text{in }X && \text{for } t>0,\\
y(0)&=0\in X.
\end{aligned}\label{state_eqation_abstract}\
\end{equation}

%\section{Hadamard directional differentiability}
%\subsection{Definition and properties}

%We want to show differentiability of the solution mapping $G$ for \eqref{state_eqation_abstract}. Because of the hysteresis operator we can not expect a Fr{\'e}chet derivative.
%Therefore we consider a weaker form of differentiability, the Hadamard directional derivative in the sense of \cite{bonnans}.
%
%To start with we define what we mean by directional differentiability of a mapping
%
%$g:U\subset X\rightarrow Y$ from an open set $U\subset X$ of a normed vector space $X$ into a normed vector space $Y$ \cite[cf.][Definition 2.44]{bonnans}.
%\begin{definition}
%	We call $g$ directionally differentiable at $x\in U$ in the direction $h\in X$ if
%	\begin{align*}
%	g'[x;h]:= \lim\limits_{\lambda \downarrow 0} \frac{g(x+\lambda h) - g(x)}{\lambda}
%	\end{align*}
%	exits. If $g$ is directionally differentiable at $x$ in every direction $h$ we call $g$ directionally differentiable at $x$.
%\end{definition}
Note that $F[y]$ is non-local in time.
In order to obtain some kind of differentiability of the reduced cost function, the solution operator of the state equation has to be differentiable in a sense which allows for the chain rule.
We can not expect a Fr{\'e}chet derivative because of the non-smooth hysteresis operator, see \cite{Brokate_weak_diff}.
But the chain rule can also be applied within the weaker concept of Hadamard directional differentiability.
\begin{definition}\label{Def:Hadam_differentiability}[Hadamard directional differentiability]
	Let $X,Y$ be normed vector spaces and let $U\subset X$ be open. If $g:U\rightarrow Y$ is directionally differentiable at $x\in U$ and if in addition for all functions $r:[0,\lambda_0)\rightarrow X$ with $\lim\limits_{\lambda \rightarrow 0} \frac{r(\lambda)}{\lambda}=0$ it holds
	$
	g'[x;h]= \lim\limits_{\lambda \downarrow 0} \frac{g(x+\lambda h + r(\lambda)) - g(x)}{\lambda}
	$
	for all directions $h\in X$, we call $g'[x;h]$ the Hadamard directional derivative of $g$ at $x$ in the direction $h$. 
	Note that $g(x+\lambda h + r(\lambda))$ is only well defined if $\lambda$ is already small enough so that $x+\lambda h + r(\lambda) \in U$. The chain rule applies for Hadamard directionally differentiable functions \cite[Lemma 4.3]{muench}.
\end{definition}
Hadamard directional differentiability of the solution operator $G$ is shown in \cite{muench}.
%The following properties hold true:
%\begin{lemma}\label{Lem:hadam_chain_rule}
%	Suppose that $g:U\subset X \rightarrow Y$ is Hadamard directionally differentiable at $x\in U$ and that $f:V\subset g(U) \rightarrow Z$ is Hadamard directionally differentiable at $g(x)\in V$.
%	Then $f \circ g : U \rightarrow Z$ is Hadamard directionally differentiable at $x$ and
%	\begin{align*}
%	(f \circ g)'[x;h] = f'\left[g(x);g'[x;h]\right]
%	\end{align*}
%	\cite[cf.][Proposition 2.47]{bonnans}.
%\end{lemma}
%
%\begin{lemma}\label{Lem:hadam_for_lip}
%	Suppose that $g:U\subset X \rightarrow Y$ is directionally differentiable at $x\in U$ and in addition Lipschitz continuous with modulus $c(x)$ in a neighbourhood of $x$. Then $g$  is Hadamard directionally differentiable at $x$ and $g'[x;.]$ is Lipschitz continuous on $X$ with modulus $c(x)$ \cite[cf.][Proposition 2.49]{bonnans}.
%\end{lemma}
%\subsection{Hadamard differentiability of the Solution operator to the state equation}
By \cite[Theorem 3.1 and Theorem 4.7]{muench} we have:
\begin{theorem}[Solution operator for the state equation]\label{Thm:state_equ_sol_op}
	Let Assumption \ref{Ass:general_ass_and_short_notation_1} hold.	
	Then for the fixed value $\alpha\in \left(0,\frac{1}{2}\right)$ and for all $u\in \mathrm{L}^q(J_T;X)$ with $q\in (\frac{1}{1-\alpha},\infty]$ problem \eqref{state_eqation_abstract} has a unique mild solution
	$y=y(u)=:y^u$ in $\mathrm{C}(\overline{J_T};X^\alpha)$.
	In particular, this means that $(F[y]) + u$ is contained in $\mathrm{L}^1(J_T;X)$ and that $y$ solves the integral equation
	\begin{equation*}
	y(t) = \int_{0}^{t} \exp(-A_p(t-s))[(F[y])(s) + u(s)] \, ds,\ t\in J_T.
	\end{equation*}
	The solution mapping 
	$G:u \mapsto y(u),\ \mathrm{L}^q(J_T;X) \rightarrow \mathrm{C}(\overline{J_T};X^\alpha)$
	is locally Lipschitz continuous.
	$G$ is linearly bounded with values in $\mathrm{C}(\overline{J_T};X^\alpha)$.
%	, i.e. for some $C=C(T)>0$ it is
%	\begin{align}
%	\Vert G(u) \Vert_{\mathrm{C}(\overline{J_T};X^\alpha)} &\leq C(T) (1+\Vert u\Vert_{\mathrm{L}^q(J_T;X)}).\label{solution_op_bdd}
%	\end{align} 
	All statements remain valid if $\mathrm{C}(\overline{J_T};X^\alpha)$ is replaced by 
	$Y_{s,0}$ 
%	(see Definition~\ref{Def:maximal parabolic regularity})
	where $s=q$ if $q < \infty$ and with $s\in (1,\infty)$ arbitrary if $q = \infty$.
	$G$ is Hadamard directionally differentiable as a mapping into $\mathrm{C}(\overline{J_T};X^\alpha)$ as well as into $Y_{q,0}$ for any $q\in (\frac{1}{1-\alpha},\infty)$.
	Its derivative $y^{u,h}:=G'[u;h]$ at $u\in \mathrm{L}^q(J_T;X)$ in direction $h\in \mathrm{L}^q(J_T;X)$ is given by the unique solution $\zeta$ of
	\begin{equation}
	\dot{\zeta}(t) + (A_p \zeta)(t)= F'[y;\zeta](t) + h(t)\ \text{for } t\in J_T,\
	\zeta(0)=0\label{eq:Thm:state_equ_sol_op}
	\end{equation}
	where $F'[y;\zeta](t)=f'[(y(t),\mathcal{W}[Sy](t));(y(t),\mathcal{W}'[Sy;S\zeta](t))]$ and $y=G(u)$.
	The mapping $h\mapsto G'[u;h]$ is Lipschitz continuous from $\mathrm{L}^q(J_T;X)$ to $\mathrm{C}(\overline{J_T};X^\alpha)$ and to $Y_{q,0}$ with a modulus $C=C(G(u),T)>0$.
	
\end{theorem}
\begin{proof}
	See \cite[Theorem 3.1 and Theorem 4.7]{muench}.
\end{proof}

%In \cite[Theorem 5.5]{muench} it is shown that:
%\begin{theorem}\label{Thm:state_equ_sol_op}
%	Let Assumption \ref{Ass:general_ass_and_short_notation} hold.
%	For any $q\in (\frac{1}{1-\alpha},\infty)$ the solution operator $G:\mathrm{L}^q(J_T;X)\rightarrow \mathrm{C}(\overline{J_T};X^\alpha)$ of problem \eqref{state_eqation_abstract} is Hadamard directionally differentiable.
%	Its derivative $y^{u,h}:=G'[u;h]$ at $u\in \mathrm{L}^q(J_T;X)$ in direction $h\in \mathrm{L}^q(J_T;X)$ is given by the unique mild solution $\zeta$ of
%	\begin{alignat*}{2}
%	\dot{\zeta}(t) + (A_p \zeta)(t)&= F'[y;\zeta](t) + h(t)&&\ \text{for } t\in J_T,\
%	\zeta(0)=0
%	\end{alignat*}
%	where $F'[y;\zeta](t)=f'[(y(t),\mathcal{W}[Sy](t);(y(t),\mathcal{W}'[Sy;S\zeta](t))]$.
%	
%	Moreover $G'[u;h]\in Y_{s,0}$ is the Hadamard directional derivative of $G:\mathrm{L}^q(J_T;X)\rightarrow Y_{q,0}$.
%	
%	The mapping $h\rightarrow G'[u;h]$ is Lipschitz continuous from $\mathrm{L}^q(J_T;X)$ to $\mathrm{C}(\overline{J_T};X^\alpha)$ and to $Y_{q,0}$ with a modulus $C=C(G(u),T)$.
%\end{theorem}

%\section{Existence of an optimal control for \lowercase{\eqref{opt_control_ control_problem}}}

Existence of an optimal control for problem \eqref{opt_control_ control_problem}-\eqref{state_equ_z} is shown in \cite[Theorem 5.4]{muench}:
%\begin{lemma}\label{Lem:convergence_nonreg}
%	Suppose that $u_n \rightharpoonup u$ in $U_i$ with $i\in \{1,2\}$.
%	
%	Then $y_n=G(B_iu_n)\rightarrow G(B_iu)$ weakly in $Y_{2,0}$ and strongly in $\mathrm{C}(\overline{J_T};X^\alpha)$ and
%	
%	$z_n=\mathcal{W}[Sy_n]\rightarrow \mathcal{W}[SG(B_iu)]$ weakly in $\mathrm{H}^1(J_T)$ and strongly in $\mathrm{C}(\overline{J_T})$ \cite[cf.][Lemma 2.3]{brokate2013optimal}. 
%	
%	If the convergence of $u_n$ is strong then $y_n\rightarrow G(B_iu)$ in $Y_{2,0}$ strongly.
%\end{lemma}

\begin{theorem}[Existence of optimal control]\label{Thm:opt_control_existence}
	Let Assumption~\ref{Ass:general_ass_and_short_notation_1} hold.
	Then for $i\in \{1,2\}$, there exists an optimal control $\overline{u}\in U_i$ for the optimal control problem \eqref{opt_control_ control_problem}-\eqref{state_equ_z}. This means that $\overline{u}$, together with the optimal state $\overline{y}=G(\overline{u})$, which solves \eqref{state_equ_y}, are a solution of the minimization problem \eqref{opt_control_ control_problem}. The solution of \eqref{state_equ_z} is given by $\overline{z}=\mathcal{W}[S\overline{y}]$.
\end{theorem}
\begin{proof}
	See \cite[Theorem 5.4]{muench}.
\end{proof}

\section{Regularized control problem}
	In order to derive an adjoint system for problem \eqref{opt_control_ control_problem}-\eqref{state_equ_z} we introduce a sequence of control problems with regularized $\varepsilon$-dependent state equations, for which we can derive adjoint systems. 
	To this aim we regularize the variational inequality which defines $\mathcal{W}$ and the non-linearity $f$, which yields a regularization of the solution operator of \eqref{state_eqation_abstract}.
	The regularization of $\mathcal{W}$ follows the techniques in \cite[Section 3]{brokate2013optimal} and the approach for the regularization of semilinear parabolic equations relies on \cite[Section 4]{meyeroptimal}.
	
	In the end of Subsection~\ref{Subsec:Reg_state_equ_and_estimates}, we estimate the norms of the solutions of the regularized state equations against the forcing term $u$, independently of $\varepsilon$. 
	
	The dynamics of the regularized state equations in dependence of $\varepsilon$ is analyzed in Subsection~\ref{Subsec:dynamics_reg_prob}: The estimates from Subsection~\ref{Subsec:Reg_state_equ_and_estimates} together with a weak compactness argument imply weak compactness of the regularized solution operators for fixed $\varepsilon > 0$. This yields weakly converging subsequences $y_{\varepsilon_k}$ and $z_{\varepsilon_k}$ for any weakly converging sequence $u_{\varepsilon}$, $\varepsilon \rightarrow 0$.
	
	In Subsection~\ref{Subsec:reg_contr_problem}, we apply the convergence result from Subsection~\ref{Subsec:dynamics_reg_prob} to deduce convergence of the solutions of regularized control problems, which are introduced in Subsection~\ref{Subsec:reg_contr_problem}, to an optimal solution of problem \eqref{opt_control_ control_problem}-\eqref{state_equ_z} as $\varepsilon\rightarrow 0$, see Theorem~\ref{Thm:conv_minimizers_reg_to_min_original}. 
	
	The adjoint equations for the solutions of the regularized control problems with $\varepsilon>0$ fixed are derived in Subsection~\ref{Subsec:Adj_reg_problem}, see Theorem~\ref{Thm:opt_syst_reg} below.
	
	In Subsection~\ref{Subsec:estimates_adjoints}, we derive uniform-in-$\varepsilon$ bounds for the norms of the adjoint variables $p_\varepsilon,q_\varepsilon$ from Theorem~\ref{Thm:opt_syst_reg}.
	
	The norm bounds on $p_\varepsilon,q_\varepsilon$ from Subsection~\ref{Subsec:estimates_adjoints} give rise to weakly converging subsequences $p_{\varepsilon_k}$ and $q_{\varepsilon_k}$.
	Taking the limit $k\rightarrow \infty$ then yields an adjoint system for \eqref{opt_control_ control_problem}-\eqref{state_equ_z}. This step is carried out in Section~\ref{Sec:Limit}.
	
	We begin with several assumptions on the functions which will enter the regularized problems.
	\begin{assumption}[Regularization]\label{Ass:regularized_problem}
		For $\varepsilon_*>0$ and $\varepsilon \in (0,\varepsilon_*]$ we assume that:
		\begin{itemize}
			\item[$(A1)_\varepsilon$] $f_\varepsilon: X^\alpha\times \mathbb{R} \rightarrow $ is Gâteaux differentiable.
			\item[$(A2)_\varepsilon$] $\sup_{(y,z)\in X^\alpha\times \mathbb{R}} \|f_\varepsilon(y,z)-f(y,z)\|_X \rightarrow 0$ as $\varepsilon \rightarrow 0$.
			\item[$(A3)_\varepsilon$] $f_\varepsilon$ is locally Lipschitz continuous with respect to the $X^\alpha$-norm and all the neighbourhoods and Lipschitz constants are equal to the ones of $f$ in (A2) in Assumption~\ref{Ass:general_ass_and_short_notation_1}, independently of $\varepsilon$.
			The growth condition
			$
			\Vert f_\varepsilon(y,x) \Vert_{X} \leq M \left( 1+ \Vert y \Vert_{X^\alpha} + \vert x \vert \right)
			$
			holds for all $y\in X^\alpha$ and $x\in \mathbb{R}$, with $M$ from (A2) in Assumption~\ref{Ass:general_ass_and_short_notation_1}.
			\item[$(A4)_\varepsilon$] Following the ideas of \cite{brokate2013optimal}, we introduce a convex function $\Psi: \mathbb{R} \rightarrow \mathbb{R}$ with $\Psi(x)\equiv 0$ for $x\in [a,b]$ and $\Psi(x)>0$ for $x\in \mathbb{R}\backslash [a,b]$. We assume that $\Psi$ is twice continuously differentiable and $\Psi'(x) \leq m_1 |x-a|$ for some $m_1>0$ and all $x\in\mathbb{R}$. Moreover, $\Psi''(x) \leq m_2$ for some $m_2>0$ and all $x\in \mathbb{R}$ and $\Psi''$ is assumed to be locally Lipschitz continuous.
		\end{itemize} 
	\end{assumption}
	
	\begin{remark}\label{Rem:construction_Psi}
		A function $\Psi$ as in Assumption \ref{Ass:regularized_problem} can be contructed as a piecewise defined mapping
		$
		\Psi = \chi_{(-\infty, a_1]} \Psi_{-2} + \chi_{(a_1, a]} \Psi_{-1} + \chi_{(b, b_1]} \Psi_{1} + \chi_{(b_1, \infty)} \Psi_{2},
		$
		where $a_1<a<b<b_1$. $\chi$ denotes the characteristic function.
		$\Psi_{-2}$ and $\Psi_{2}$ are affine linear and $\Psi_{-1}$ and $\Psi_{1}$ are polynomials of order four with roots in $a$ and $b$ which are at the same time saddle points and with turning points in $a_1$ and $b_1$.\\
		For example we can choose $b_1:=b+2$, 
		$
			\Psi_1(x):=(x-b)^3(4+b-x)
		$
		and
		$
		\Psi_2(x):=16(x -1-b)
		$
		and define $\Psi_{-1},\Psi_{-2}$ in a similar way, cf. Figure~\ref{fig:graph_psi}.
%		Then it is
%		\begin{align*}
%		\Psi_1'(x) = 4[3(x-b)^2 - (x-b)^3] \text{ and } \Psi_1''(x) = 12[2(x-b) -(x-b)^2].
%		\end{align*}
%		Clearly
%		\[
%		\Psi_1(b)=\Psi_1'(b)=\Psi_1''(b)=0
%		\text{ and also }
%		\Psi_1''(2+b) = 6(2) -3(2)^2 = 0.
%		\]
%		Furthermore
%		\[
%		\Psi_1'(2+b) = 4[3*2^2 - 2^3] = 16
%		\text{ and }
%		\Psi_1(2+b)=2^3*2 = 16.
%		\]
%		We set
%		\[
%		\Psi_2(x):=16(x -1-b)
%		\]
%		in order to have
%		\[
%		\Psi_1''(2+b) = \Psi_2''(2+b) = 0,
%		\	
%		\Psi_2'(2+b) = 16= \Psi_1'(2+b)
%		\]
%		and
%		\[
%		\Psi_2(2+b) = 16  = \Psi_1(2+b).
%		\]
%		$\Psi_{-1},\Psi_{-2}$ can be constructed in an analogous way cf. Figure~\ref{fig:graph_psi}.
		Local Lipschitz continuity of $\Psi''$ also in the points where the functions $\Psi_{-2},\ldots,\Psi_{2}$ are glued together is not hard to see. It follows that $\Psi''$ is Lipschitz continuous.
		
	\end{remark}
	
	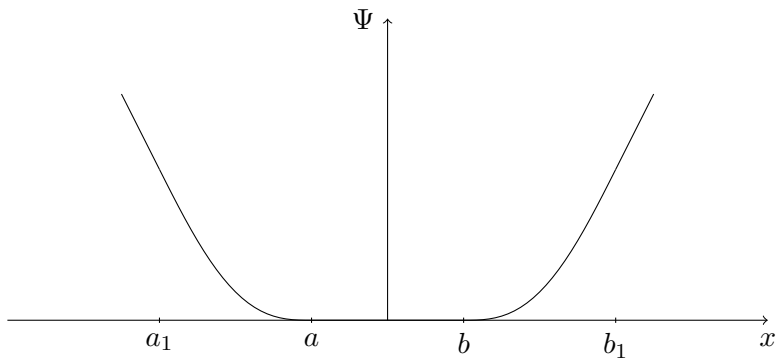
\begin{figure}
		\centering
		\begin{tikzpicture}[] 
		%\draw[step=.5cm, gray, very thin] (0,0) grid (10,5); 
		\draw[->] (-5,0) -- (5,0) coordinate (x axis);
		\draw (-1pt,4 ) node[anchor=east,fill=white] {$\Psi$};
		\draw (5 cm,-1pt) node[anchor=north] {$x$};
		\draw[->] (0,0) -- (0,4) coordinate (y axis);
		\draw (-3 cm,1pt) -- (-3 cm,-1pt) node[anchor=north] {$a_1$};
		\draw (-1 cm,1pt) -- (-1 cm,-1pt) node[anchor=north] {$a$};
		\draw (1 cm,1pt) -- (1 cm,-1pt) node[anchor=north] {$b$};
		\draw (3 cm,1pt) -- (3 cm,-1pt) node[anchor=north] {$b_1$};
		\draw [black, domain=1:3] plot (\x, {1/8*(\x-1)^3*(4+1-\x)});
		\draw [black, domain=3:3.5] plot (\x, {1/8*16*(\x-1-1)});
		\draw [black] (-1,0) -- (1,0);
		\begin{scope}[yscale=1,xscale=-1]
		\draw [black, domain=1:3] plot (\x, {1/8*(\x-1)^3*(4+1-\x)});
		\draw [black, domain=3:3.5] plot (\x, {1/8*16*(\x-1-1)});
		\end{scope}
		\end{tikzpicture}
		\caption{Graph of $\Psi$}
		\label{fig:graph_psi}
	\end{figure}
	
	We introduce the following regularized state equations for $i\in\{1,2\}$ and $\varepsilon>0$:
%	 \cite[cf.][Section 3.1]{brokate2013optimal}:
	\begin{alignat}{3}
	\dot{y}(t) +(A_p y)(t) &= f_\varepsilon(y(t),z(t))  + (B_i u)(t)\  \text{in } X&&\ \text{for }t\in J_T, \label{state_equ_regular_y}&&\
	y(0)= 0 \text{ in } X,\\
	\dot{z}(t)-S\dot{y}(t) &= -\frac{1}{\varepsilon}\Psi'(z(t)) &&\ \text{for } t\in J_T\label{state_equ_regular_z},&&\
	z(0)=z_0.
	\end{alignat}

	\subsection{Regularization of \lowercase{\eqref{state_eqation_abstract} and uniform-in-$\varepsilon$ estimates}}\label{Subsec:Reg_state_equ_and_estimates}
	In this subsection, we introduce a regularization of \eqref{state_eqation_abstract}, similar to the regularized state equations \eqref{state_equ_regular_y}-\eqref{state_equ_regular_z} but for source terms $u\in \mathrm{L}^q(J_T;X)$. We show well-posedness and estimate the norms of the solutions in $u$, independently of $\varepsilon$. The ideas for many of the steps in this subsection go back to \cite[Subsection 3.1]{brokate2013optimal}.
	\begin{definition}[Regularized stop]\label{Def:Solution_op_regularized_hyst}
	For $\varepsilon\in (0,\varepsilon_*]$ we denote by $Z_\varepsilon:v\mapsto Z_\varepsilon(v)$ the solution operator of
	\begin{equation*}
	\dot{z}(t)-\dot{v}(t) = -\frac{1}{\varepsilon}\Psi'(z(t)) \ \text{for } t\in J_T,\
	z(0)=z_0,
	\end{equation*}
	or of the corresponding integral equation. The input $v$ is a function defined on $J_T$.
	\end{definition}
	
	\begin{remark}\label{Rem:Solution_op_regularized_hyst}
		By standard techniques it follows that $Z_\varepsilon$ is continuously differentiable on $\mathrm{C}(\overline{J_T})$.	
	Its derivative at $v$ in direction $h$ is given by the unique solution $Z'_\varepsilon[v;h]=z$ of the integral equation
	$
	z(t) = h(t) - \int_{0}^{t} \frac{1}{\varepsilon} \Psi''(Z_\varepsilon(v)(s))z(s) ds.
	$
	$Z_\varepsilon$ is bounded on $\mathrm{W}^{1,q}(J_T)$ for all $q\in (1,\infty)$.
	\end{remark}	
	Similar to the definition of $F$ in Subsection~\ref{Subsec:Sol_op_optimal_control} we denote $(F_\varepsilon(y))(t):=f_\varepsilon(y(t),Z_\varepsilon(S y)(t))$.
	Consider the abstract evolution equation 
%	\begin{alignat*}{2}
%	\dot{y}(t) +A_p y(t) &= f_\varepsilon(y(t),z(t))  + u(t)&&\  \text{in } X \text{ for }t>0, \label{state_equ_regular_abstract}\\
%	y(0)&= 0 &&\ \text{in } X,\\
%	\dot{z}(t)-S\dot{y}(t) &= -\frac{1}{\varepsilon}\Psi'(z(t)) &&\ \text{for } t>0,\\
%	z(0)&=0.
%	\end{alignat*}  
	\begin{equation}
	\begin{aligned}
	\dot{y}(t) + (A_p y)(t) &= (F_\varepsilon(y))(t) + u(t)&& \text{in } X && \text{for } t>0,\label{state_equ_regular_abstract}\\
	y(0)&=0\in X.
	\end{aligned}
	\end{equation}
	
	\begin{corollary}[Existence of regularized problem]\label{Cor:state_equ_reg_wellposed}
		Let Assumption~\ref{Ass:general_ass_and_short_notation_1} and Assumption~\ref{Ass:regularized_problem} hold and let $\varepsilon\in (0,\varepsilon_*]$ be arbitrary.		
		Furthermore, assume $q\in (\frac{1}{1-\alpha},\infty]$ and set $s=q$ if $q<\infty$ or $s\in (1,\infty)$ arbitrary if $q=\infty$.
%		We define the mapping $Z_\varepsilon:v\mapsto Z_\varepsilon(y)$ by the solution operator for
%		\begin{alignat*}{2}
%		\dot{z}(t)-\dot{v}(t) &= -\frac{1}{\varepsilon}\Psi'(z(t)) &&\ \text{for } t\in J_T,\\
%		z(0)&=0.
%		\end{alignat*}
%		Then $Z_\varepsilon$ is continuously differentiable on $\mathrm{C}(\overline{J_T})$.	
%		
%		Its derivative at $v$ in direction $h$ is given by the unique solution $Z_\varepsilon^v(h)=z_\varepsilon^{v,h}$ of the integral equation
%		\begin{align*}
%		z_\varepsilon^{v,h}(t) &= h(t) - \int_{0}^{t} \frac{1}{\varepsilon} \Psi''(Z_\varepsilon(v)(s))z_\varepsilon^{v,h}(s) ds.
%		\end{align*}	
%		$Z_\varepsilon$ is bounded on $\mathrm{W}^{1,s}(J_T)$.	
		Then for all $u\in \mathrm{L}^q(J_T;X)$ problem \eqref{state_equ_regular_abstract} has a unique solution $y_\varepsilon(u)$ in $Y_{s,0}$.
%		(see Definition~\ref{Def:maximal parabolic regularity}).
%		where $z=z_\varepsilon(u)=Z_\varepsilon(Sy_\varepsilon(u))$.		
		The solution mapping 
		$
		G_\varepsilon:u \mapsto y_\varepsilon(u)=:y_\varepsilon^u
		$
		is locally Lipschitz continuous from $\mathrm{L}^q(J_T;X)$ to $\mathrm{C}(\overline{J_T};X^\alpha)$ and to $Y_{s,0}$. We denote $z_\varepsilon^u:=z_\varepsilon(u):=Z_\varepsilon(Sy_\varepsilon^{u})$.
	\end{corollary}

\begin{proof}

Unique solvability of \eqref{state_equ_regular_abstract} and local Lipschitz continuity of the solution mapping follow because $Z_\varepsilon$ satisfies the properties of $\mathcal{W}$ in Theorem~\ref{Thm:state_equ_sol_op}.
\end{proof}

In the next step we estimate the norms of the solutions of \eqref{state_equ_regular_abstract} independently of $\varepsilon$ by the norm of the source function $u\in \mathrm{L}^q(J_T;X)$ . This yields analogous estimates also for the solutions of \eqref{state_equ_regular_y}-\eqref{state_equ_regular_z} if $u$ is replaced by $B_iu$.

\begin{lemma}[Uniform bounds]\label{Lem:estimates_regularized_solution}
	Adopt the assumptions and the notation from Lemma~\ref{Cor:state_equ_reg_wellposed}.
	There exists a constant $c>0$ which is independent of $\varepsilon$ and $u$ such that the following holds true. For all $q\in (\frac{1}{1-\alpha}, \infty]$ and $\varepsilon\in (0,\varepsilon_*]$ we have
	\begin{equation}
	\begin{aligned}
	\| y_\varepsilon^{u} \|_{Y_{s,0}} \leq c(1+\|u \|_{\mathrm{L}^q(J_T;X)})
	&&
	\text{and}
	&&
	\|z_\varepsilon^u\|_{\mathrm{C}(\overline{J_T})} \leq c(1+ \|u \|_{\mathrm{L}^q(J_T;X)})\label{estim_reg_problem_y_z}
	\end{aligned}
	\end{equation}
    with $s=q$ if $q<\infty$ and for all $s\in (1,\infty)$ if $q=\infty$.
	Moreover, there holds
	\begin{equation}
	0\leq\int_{0}^{T} \left|\dot{z}_\varepsilon^u(s)\right|^2 ds + \sup_{t\in \overline{J_T}} \frac{1}{\varepsilon} \Psi(z_\varepsilon^u(t)) \leq c(1+ \|u \|_{\mathrm{L}^2(J_T;X)})^2.\label{estim_reg_problem_dot(z)_1/epsilonPsi}
	\end{equation}
\end{lemma}

\begin{proof}
%	We consider $q$ and $s$ as in Lemma~\ref{Cor:state_equ_reg_wellposed}.
%\cite[cf.][Section 3.1]{brokate2013optimal}.
Note first that for $v\in \mathrm{W}^{1,s}(J_T)$ and for $t\in J_T$ we have 
%similar as in \cite[Section 3.1]{brokate2013optimal} (we do not assume $z_0=0$) that
\begin{equation*}
	|Z_\varepsilon(v)(t)-z_0| - |Z_\varepsilon(v)(0)-z_0| = \int_{0}^{t} \frac{d}{ds}|Z_\varepsilon(v)-z_0| ds = \int_{0}^{t} \frac{\frac{d}{ds}(Z_\varepsilon(v)) (Z_\varepsilon(v)-z_0)}{|Z_\varepsilon(v)-z_0|} ds.
\end{equation*}
Note that $\Psi'(x)(x-z_0) \geq 0$ for all $x\in\mathbb{R}$ because $\Psi$ is convex and since $\Psi'(z_0)=0$.
We insert $Z_\varepsilon(v)(0)=z_0$ and $\frac{d}{ds}(Z_\varepsilon(v))=\dot{v} -\frac{1}{\varepsilon}\Psi'(Z_\varepsilon(v))$ according to Definition~\ref{Def:Solution_op_regularized_hyst}. 
%If the initial value $v(0)=v_0$ is fixed, then b
The triangle inequality and rearranging yields
\begin{equation*}
	0\leq |Z_\varepsilon(v)(t)| + \frac{1}{\varepsilon}\int_{0}^{t} \frac{\Psi'(Z_\varepsilon(v)) (Z_\varepsilon(v)-z_0)}{|Z_\varepsilon(v)-z_0|} ds \leq |z_0| + \int_{0}^{t} | \dot{v}(s)| ds.
\end{equation*}
%and this estimate is independent of $\varepsilon$ and $v$.
Hence, with $z_\varepsilon^u=Z_\varepsilon(Sy_\varepsilon^{u})$ and $v=Sy_\varepsilon^{u}$ there follows
\begin{equation}
	0\leq|z_\varepsilon^u(t)| \leq |z_0| + \int_{0}^{t} |S\dot{y}_\varepsilon^{u}(s)| ds.\label{Est:z_eps_in_Sy_eps}
\end{equation}
%and this estimate is independent of $\varepsilon$ and $y_\varepsilon^{u}$.
Because the Riesz representation $w$ of $S$ is contained in $ \mathrm{dom}([(A_p+1)^{1-\alpha}]^*)$ by (A3) in Assumption~\ref{Ass:general_ass_and_short_notation_1}, we can estimate for all $y\in \mathrm{dom}(A_p)$:
\begin{equation*}
	\begin{split}
	| S A_p y | =& | S (A_p+1) y - Sy| 
	\leq | \langle w, (A_p +1)y \rangle_{X} | + \|S\|_{[X^\alpha]^*}\|y\|_{X^\alpha}\\
	=& |\langle w , (A_p+1)^{1-\alpha}(A_p+1)^\alpha y\rangle_{X}| + \|S\|_{[X^\alpha]^*}\|y\|_{X^\alpha}\\
	=& |\langle[(A_p+1)^{1-\alpha}]^*w, (A_p+1)^\alpha y \rangle_{X} + \|S\|_{[X^\alpha]^*}\|y\|_{X^\alpha}\\
	\leq& (\|[(A_p+1)^{1-\alpha}]^*w\|_{X^*} + \|S\|_{[X^\alpha]^*}) \|y\|_{X^\alpha} =: c_1 \|y\|_{X^\alpha}.
	\end{split}
\end{equation*}
For $y=Sy_\varepsilon^{u}(t)$, this together with \eqref{state_equ_regular_abstract} and the triangle inequality implies that for a.e. $t\in J_T$
\begin{equation*}
	|S\dot{y}_\varepsilon^{u}(t)| \leq c_1 \|y_\varepsilon^{u}(t)\|_{X^\alpha} + |S f_\varepsilon(y_\varepsilon^{u}(t),z_\varepsilon^u(t))| + |Su(t)|.
\end{equation*}
Consequently, by the linear growth condition on $f_\varepsilon$ in $(A3)_\varepsilon$ of Assumption~\ref{Ass:regularized_problem} we further estimate \eqref{Est:z_eps_in_Sy_eps} by
\begin{equation*}
	\begin{split}
	|z_\varepsilon^u(t)| &\leq |z_0| + \int_{0}^{t} c_1 \|y_\varepsilon^{u}(s)\|_{X^\alpha} + |S f_\varepsilon(y_\varepsilon^{u}(s),z_\varepsilon^u(s))| + |Su(s)| ds\\
	&\leq |z_0| + \int_{0}^{t} (M\|S\|_{X^*}+c_1)\left[\| y_\varepsilon^{u}(s)\|_{X^\alpha} + |z_\varepsilon^u(s)| + 1\right] + \|S\|_{X^*}\|u(s)\|_X  ds.
	\end{split}
\end{equation*}
Remember that $y_\varepsilon^{u}(0)=0$ for any $\varepsilon\in(0,\varepsilon_*]$.
Since $y_\varepsilon^{u}$ is the mild solution of \eqref{state_equ_regular_abstract}, we can use \eqref{frac_pow_estimate1} for arbitrary $\gamma\in(0,1)$ and again the linear growth condition on $f_\varepsilon$ to obtain 
\begin{equation*}
	\begin{split}
	\| y_\varepsilon^{u}(t) \|_{X^\alpha} &= \left\| \int_{0}^{t} e^{-A_p(t-s)}[f_\varepsilon(y_\varepsilon^{u}(s),z_\varepsilon^u(s)) + u(s) ] ds \right\|_{X^\alpha}\\
	&\leq C_\alpha e^{(1-\gamma) T}\int_{0}^{t} (t-s)^{-\alpha}[M(\| y_\varepsilon^{u}(s)\|_{X^\alpha} + |z_\varepsilon^u(s)| + 1) + \|u(s)\|_X ] ds.
	\end{split}
\end{equation*}
Note that 
$
\left(\int_{0}^{t} (t-s)^{-\alpha q'} \,ds\right)^{1/q'} = \left(\frac{t^{1-\alpha q'}}{1-\alpha q'}\right)^{1/q'} = \frac{t^{1/q'-\alpha}}{(1-\alpha q')^{1/q'-\alpha}}
$ since $q<\frac{1}{1-\alpha}\Leftrightarrow \frac{1}{q'}-\alpha >0$.
We sum up the estimates for $|z_\varepsilon^u(t)|$ and $\| y_\varepsilon^{u}(t) \|_{X^\alpha}$ and apply Gronwall's Lemma to arrive at
\begin{equation*}
	\begin{aligned}
	\| y_\varepsilon^{u} \|_{\mathrm{C}(\overline{J_T};X^\alpha)} \leq c_3(1+\|u \|_{\mathrm{L}^q(J_T;X)})
	&&\text{and}&&
	\|z_\varepsilon^u\|_{\mathrm{C}(\overline{J_T})} \leq c_3(1+ \|u \|_{\mathrm{L}^q(J_T;X)})
	\end{aligned}
\end{equation*}
for all  $q\in (\frac{1}{1-\alpha}, \infty]$ and a constant $c_3>0$ which depends on $T$, $q'$ and $\alpha$ but not on $\varepsilon$ and $u$.
% \cite[cf.][equation (19)-(20)]{brokate2013optimal}.
By maximal parabolic regularity of $A_p$, see Remark~\ref{Rem:maximal parabolic regularity}, one obtains 
\begin{equation*}
	\| y_\varepsilon^{u} \|_{Y_{s,0}}  \leq c_4(1+ \|u \|_{\mathrm{L}^q(J_T;X)})
\end{equation*}
for $s=q$ if $q\in (\frac{1}{1-\alpha},\infty)$ and for all $s\in (1,\infty)$ if $q=\infty$, again for some $c_4>0$ which is independent of $\varepsilon$ and $u$. This shows \eqref{estim_reg_problem_y_z}.
We are left to prove \eqref{estim_reg_problem_dot(z)_1/epsilonPsi}.
Note that $2>\frac{1}{1-\alpha}$ by (A2) in Assumption~\ref{Ass:general_ass_and_short_notation_1}.
Because $S\in X^*$, \eqref{estim_reg_problem_y_z} yields 
$
\| S\dot{y}_\varepsilon^{u} \|_{\mathrm{L}^2(J_T)}\leq c_5(1+ \|u \|_{\mathrm{L}^2(J_T;X)})
$
for $c_5=c_4\|S\|_{X^*}$.
%Moreover, $\Psi'(x)(x - \xi) \geq 0$ because $\Psi$ is convex and since $\Psi'(\xi)=0$ for $\xi\in [a,b]$. This fact, together with the usage of \eqref{estim_reg_problem_y_z} and H{\"o}lder's inequality in \eqref{estim:_1/eps_Psi} we conclude
%\begin{align}
%	0 \leq \sup_{t\in \overline{J_T}} \frac{1}{\varepsilon} \int_{0}^{t} \Psi'(z_\varepsilon^u(s))(z_\varepsilon^u(s)-\xi) ds \leq c_6(1+ \|u \|_{\mathrm{L}^2(J_T;X)})^2\label{est:dtPsibounded}
%\end{align}
%for some $c_6>0$.
We test $\dot{z}_\varepsilon^u$ in  Definition~\ref{Def:Solution_op_regularized_hyst} by $\dot{z}_\varepsilon^u$, integrate from zero to $t$ and use Young's inequality to compute for $t\in \overline{J_T}$:
% as in \cite{brokate2013optimal} and use 
\begin{equation*}
	\begin{split}
	&\int_{0}^{t} \left|\dot{z}_\varepsilon^u(s)\right|^2 ds =  \int_{0}^{t} S\dot{y}_\varepsilon^{u}(s)\dot{z}_\varepsilon^u(s)ds -\frac{1}{\varepsilon} \int_{0}^{t} \Psi'(z_\varepsilon^u(s))\dot{z}_\varepsilon^u(s) ds\\
	&\leq \frac{1}{2}\int_{0}^{t} \left|\dot{z}_\varepsilon^u(s)\right|^2 ds + \frac{c_5^2}{2}(1+ \|u \|_{\mathrm{L}^2(J_T;X)})^2 
	- \frac{1}{\varepsilon} [\Psi(z_\varepsilon^u(t)) - \Psi(z_\varepsilon^u(0))].
	\end{split}
\end{equation*} 
Since $\Psi(z_\varepsilon^u(0))=0$ and because $\Psi\geq 0$ it follows
% from \eqref{est:dtPsibounded} that
\begin{equation*}
	0\leq\int_{0}^{T} \left|\dot{z}_\varepsilon^u(s)\right|^2 ds + 2\sup_{t\in \overline{J_T}} \frac{1}{\varepsilon} \Psi(z_\varepsilon^u(t)) \leq c_5^2(1+ \|u \|_{\mathrm{L}^2(J_T;X)})^2.
\end{equation*}
% \cite[cf.][(23)]{brokate2013optimal}.
\end{proof}
The estimates which we derived in this subsection are crucial for Subsection~\ref{Subsec:dynamics_reg_prob}.

\subsection{Dynamics of the regularized states}\label{Subsec:dynamics_reg_prob}
This subsection contains ideas from \cite[Section 4]{meyeroptimal} and \cite[Section 3.1]{brokate2013optimal}.
We prove weak continuity of the solution operator of \eqref{state_equ_regular_abstract} for fixed $\varepsilon\in(0,\varepsilon_*]$.
This yields weakly converging subsequences $y_{\varepsilon_k}$ and $z_{\varepsilon_k}$ for any weakly converging sequence $u_{\varepsilon}$, $\varepsilon \rightarrow 0$.
All results then also hold for the regularized state equations \eqref{state_equ_regular_y}-\eqref{state_equ_regular_z}.

Using this, we are able to prove convergence of the solutions of the regularized control problems, as defined in Subsection~\ref{Subsec:reg_contr_problem}, to an optimal solution of problem \eqref{opt_control_ control_problem}-\eqref{state_equ_z} with $\varepsilon\rightarrow 0$.

The following lemma is proved as \cite[Lemma 5.3]{muench}.
%, \cite[cf.][Lemma 3.1]{brokate2013optimal}.
\begin{lemma}\label{Lem:continuity_G_epsilon}
Let Assumption~\ref{Ass:general_ass_and_short_notation_1} and Assumption~\ref{Ass:regularized_problem} hold and consider the notation from Lemma~\ref{Cor:state_equ_reg_wellposed}.		
Suppose that $u_n \rightharpoonup u$ in $\mathrm{L}^2(J_T;X)$ with $n\rightarrow \infty$ for some sequence $\{u_n\}\subset \mathrm{L}^2(J_T;X)$.
For $\varepsilon\in (0,\varepsilon_*]$ fixed consider the solutions $y_\varepsilon^{u_n}$ and $y_\varepsilon^u$ of \eqref{state_equ_regular_abstract}, together with $z_\varepsilon^{u_n}$ and $z_\varepsilon^{u}$. Then $y_\varepsilon^{u_n}\rightarrow y_\varepsilon^u$ with $n\rightarrow \infty$ weakly in $Y_{2,0}$ and strongly in $\mathrm{C}(\overline{J_T};X^\alpha)$ and $z_\varepsilon^{u_n}\rightarrow z_\varepsilon^u$ with $n\rightarrow \infty$ weakly in $\mathrm{H}^1(J_T)$ and strongly in $\mathrm{C}(\overline{J_T})$. 
If the convergence of $\{u_n\}$ is strong then the convergence of $\{y_\varepsilon^{u_n}\}$ in $Y_{2,0}$
% and of $z_\varepsilon^n$ in $\mathrm{H}^1(J_T)$
is also strong.
The same holds if $\mathrm{L}^2(J_T;X)$ is replaced by $U_i$ for $i\in\{1,2\}$ and if $u_n$ and $u$ are replaced by $B_iu_n$ and $B_iu$. In this case, $(y_\varepsilon^{B_iu_n},z_\varepsilon^{B_iu_n})$ and $(y_\varepsilon^{B_iu},z_\varepsilon^{B_iu})$ are the solutions of \eqref{state_equ_regular_y}-\eqref{state_equ_regular_z}.
\end{lemma}	
Furthermore, we have the following convergence result:
% \cite[cf.][Lemma 3.2]{brokate2013optimal}:
\begin{lemma}\label{Lem:convergence_reg_to_nonreg}
	Let Assumption~\ref{Ass:general_ass_and_short_notation_1} and Assumption~\ref{Ass:regularized_problem} hold and consider the notation from Lemma~\ref{Cor:state_equ_reg_wellposed}.		
	Suppose that $u_\varepsilon \rightharpoonup u$ in $\mathrm{L}^2(J_T;X)$ as $\varepsilon\rightarrow 0$.
	Consider the solutions $y_\varepsilon^{u_\varepsilon}$ and $y_\varepsilon^u$ of \eqref{state_equ_regular_abstract}, together with $z_\varepsilon^{u_\varepsilon}$ and $z_\varepsilon^{u}$.
	Then $y_\varepsilon^{u_\varepsilon}\rightarrow y^u$ with $\varepsilon\rightarrow 0$ weakly in $Y_{2,0}$ and strongly in $\mathrm{C}(\overline{J_T};X^\alpha)$ and $z_\varepsilon^{u_\varepsilon}\rightarrow \mathcal{W}[Sy^u]$ with $\varepsilon\rightarrow 0$ weakly in $\mathrm{H}^1(J_T)$ and strongly in $\mathrm{C}(\overline{J_T})$. 
	If the convergence of $\{u_\varepsilon\}$ is strong then also the convergence of $\{y_\varepsilon^{u_\varepsilon}\}$ in $Y_{2,0}$ is strong.
	The same holds if $\mathrm{L}^2(J_T;X)$ is replaced by $U_i$ for $i\in\{1,2\}$ and if $u_\varepsilon$ and $u$ are replaced by $B_iu_\varepsilon$ and $B_iu$.
	In this case, $(y_\varepsilon^{B_iu_\varepsilon},z_\varepsilon^{B_iu_\varepsilon})$ and $(y_\varepsilon^{B_iu},z_\varepsilon^{B_iu})$ are the solutions of \eqref{state_equ_regular_y}-\eqref{state_equ_regular_z}.
\end{lemma}
\begin{proof}
	The proof combines the proofs of \cite[Lemma 3.2]{brokate2013optimal} and \cite[Lemma 5.3]{muench}.
	By Lemma~\ref{Lem:estimates_regularized_solution} we obtain a bound for $y_\varepsilon^{u_\varepsilon}$ in $Y_{2,0}$ and for $z_\varepsilon^{u_\varepsilon}$ in $\mathrm{H}^1(J_T)$ which is independent of $\varepsilon\in (0,\varepsilon_*]$. Hence, there exists a subsequence $\{\varepsilon_k\}$ of the sequence $\{\varepsilon\}$ and functions $\tilde{y}\in Y_{2,0}$ and $\tilde{z}\in\mathrm{H}^1(J_T)$ to which $y_{\varepsilon_k}(u_{\varepsilon_k})$ and $z_{\varepsilon_k}(u_{\varepsilon_k})$ converge weakly in $Y_{2,0}$ and $\mathrm{H}^1(J_T)$ and strongly in $\mathrm{C}(\overline{J_T};X^\alpha)$ and $\mathrm{C}(\overline{J_T})$ with $k\rightarrow \infty$.
	We abbreviate $y_{\varepsilon_k}:=y_{\varepsilon_k}(u_{\varepsilon_k})$ and $z_{\varepsilon_k}:=z_{\varepsilon_k}(u_{\varepsilon_k})$.
	\eqref{estim_reg_problem_dot(z)_1/epsilonPsi} implies that $\Psi(z_{\varepsilon_k}(t)) \rightarrow 0$ with $k\rightarrow \infty$ for $t\in \overline{J_T}$. By $(A4)_\varepsilon$ in Assumption~\ref{Ass:regularized_problem} this yields $\tilde{z}(t)\in [a,b]$ for $t\in \overline{J_T}$. For any $x\in \mathbb{R}$ and $\xi\in [a,b]$ there holds $\Psi'(x)(x - \xi) \geq 0$ because $\Psi$ is convex and since $\Psi'(\xi)=0$ for $\xi\in [a,b]$. For any $\xi \in [a,b]$ we therefore have
	\begin{equation*}
		\int_{0}^{T} (\dot{z}_{\varepsilon_k}(t) - S\dot{y}_{\varepsilon_k}(t))(z_{\varepsilon_k}(t) - \xi )  \, dt
		=\int_{0}^{T} -\frac{1}{\varepsilon} \Psi'(z_{\varepsilon_k}(u_{\varepsilon_k}(t))(z_{\varepsilon_k}(t) - \xi)\, dt \leq 0.
	\end{equation*}
%	\cite[cf.][Proof of Lemma 3.2]{brokate2013optimal}. 
	Taking the limit $k\rightarrow \infty$ yields $\tilde{z}=\mathcal{W}[S\tilde{y}]$ since $\tilde{z}$ solves \eqref{stop1}-\eqref{stop2} with $v=S\tilde{y}$.
%	As in \cite[Lemma 5.2]{muench}, 
	Weak continuity of $\frac{d}{dt}$ and $A_p$ implies
	$
	\frac{d}{dt}y_{\varepsilon_k} + A_p y_{\varepsilon_k} \rightharpoonup \frac{d}{dt}\tilde{y} + A_p \tilde{y}
	$ in $\mathrm{L}^2(J_T;X)$ with $k\rightarrow \infty$.
	For $\varepsilon_k$ small enough we estimate with $(A3)_\varepsilon$ in Assumption~\ref{Ass:regularized_problem}:
	\begin{equation*}
	\begin{split}
	&\|F_{\varepsilon_k}[y_{\varepsilon_k}] - F[\tilde{y}] \|_{\mathrm{C}(\overline{J_T};X)}
	=\|f_{\varepsilon_k}(y_{\varepsilon_k}(\cdot),z_{\varepsilon_k}(\cdot)) - f(\tilde{y}(\cdot),\tilde{z}(\cdot))\|_{\mathrm{C}(\overline{J_T};X)}\\
	&\leq \|f_{\varepsilon_k}(y_{\varepsilon_k}(\cdot),z_{\varepsilon_k}(\cdot)) - f_{\varepsilon_k}(\tilde{y}(\cdot),\tilde{z}(\cdot))\|_{\mathrm{C}(\overline{J_T};X)} + \|f_{\varepsilon_k}(\tilde{y}(\cdot),\tilde{z}(\cdot))-f(\tilde{y}(\cdot),\tilde{z}(\cdot))\|_{\mathrm{C}(\overline{J_T};X)}\\
	&\leq L(\tilde{y})(\|y_{\varepsilon_k}-\tilde{y}\|_{\mathrm{C}(\overline{J_T};X^\alpha)}+ \|z_{\varepsilon_k}-\tilde{z}\|_{\mathrm{C}(\overline{J_T})}) + \|f_{\varepsilon_k}(\tilde{y}(\cdot),\tilde{z}(\cdot))-f(\tilde{y}(\cdot),\tilde{z}(\cdot))\|_{\mathrm{C}(\overline{J_T};X)}.
	\end{split}
	\end{equation*}
	Because the right side converges to zero, we conclude that $F_{\varepsilon_k}[y_{\varepsilon_k}]$ converges to $F[\tilde{y}]$ in $\mathrm{C}(\overline{J_T};X)$ with $k\rightarrow \infty$.
	This together with $\tilde{z}=\mathcal{W}[S\tilde{y}]$ yields $\tilde{y}=G(u)$. Uniqueness of the limit implies convergence of the whole sequence.
	The statement about strong convergence follows essentially the same way as in \cite[Lemma 5.3]{muench}.
\end{proof}

\subsection{The regularized optimal control problem}\label{Subsec:reg_contr_problem}
In this subsection, we introduce regularized optimal control problems. It still requires work to get adjoint systems for those problems. 
Nevertheless, we can exploit linearity of the derivatives of the solution operators of \eqref{state_equ_regular_y}-\eqref{state_equ_regular_z} to derive adjoint systems by a direct approach. This will be done in Subsection~\ref{Subsec:Adj_reg_problem} below. 

We follow the ideas in \cite[Section 3.2]{brokate2013optimal} and \cite[Section 4]{meyeroptimal} in this subsection.
For $i\in\{1,2\}$ consider an optimal control $\overline{u}\in U_i$ of problem \eqref{opt_control_ control_problem}-\eqref{state_equ_z} together with the state $\overline{y}=G(B_i\overline{u})$ and $\overline{z}=\mathcal{W}[S\overline{y}]$. Existence of $\overline{u}$ follows from Theorem~\ref{Thm:opt_control_existence}.
%We define a regularized cost functional by
%\[
%J_{\mathrm{reg}}(y,u;\overline{u}):= J(y,u) + \frac{1}{2}\|u - \overline{u}\|_{U_i}^2.
%\]
For $\varepsilon \in (0,\varepsilon_*]$ we introduce the regularized optimal control problem
\begin{equation}
\min_{u \in U_i}J_{\mathrm{reg}}(y,u;\overline{u}) = \min_{u \in U_i}J(y,u) + \frac{1}{2}\|u - \overline{u}\|_{U_i}^2 \label{opt_control_problem_regular}
\end{equation} 
subject to \eqref{state_equ_regular_y}-\eqref{state_equ_regular_z}.
% as in \cite[Theorem 3.3]{brokate2013optimal}:
\begin{theorem}[Convergence of optimal solutions]\label{Thm:conv_minimizers_reg_to_min_original}
	Let Assumption~\ref{Ass:general_ass_and_short_notation_1} and Assumption~\ref{Ass:regularized_problem} hold. For $i\in\{1,2\}$ suppose that $\overline{u}\in U_i$ is an optimal control for problem \eqref{opt_control_ control_problem}-\eqref{state_equ_z}.
	Then for all $\varepsilon \in (0,\varepsilon_*]$ problem \eqref{state_equ_regular_y},\eqref{state_equ_regular_z},\eqref{opt_control_problem_regular} has an optimal control $\overline{u}_\varepsilon\in U_i$.
	This means that $\overline{u}_\varepsilon$, together with $\overline{y}_\varepsilon=G_\varepsilon(B_i\overline{u}_\varepsilon)$ and $\overline{z}_\varepsilon=Z_\varepsilon(S\overline{y}_\varepsilon)$ (see Definition~\ref{Def:Solution_op_regularized_hyst}), are a solution of the minimization problem \eqref{opt_control_problem_regular}. Furthermore, $\overline{u}_\varepsilon \rightarrow \overline{u}$ in $U_i$, $\overline{y}_\varepsilon \rightarrow \overline{y}=G(B_i\overline{u})$ in $Y_{2,0}$ and in $\mathrm{C}(\overline{J_T};X^\alpha)$ and $\overline{z}_\varepsilon \rightarrow \overline{z}=\mathcal{W}[S\overline{y}]$ weakly in $\mathrm{H}^1(J_T)$ and strongly in $\mathrm{C}(\overline{J_T})$ with $\varepsilon\rightarrow 0$.
\end{theorem}	

\begin{proof}
	First of all note that the embedding $Y_{0,2}\hookrightarrow U_1$ is continuous, because $\mathrm{dom}(A_p)\simeq \mathbb{W}_{\Gamma_D}^{1,p}(\Omega)\hookrightarrow \mathrm{L}^2(\Omega)$.
	Note also that $\overline{u}$ exists by Theorem~\ref{Thm:opt_control_existence}. Existence of optimal controls $\overline{u}_\varepsilon$ for \eqref{state_equ_regular_y},\eqref{state_equ_regular_z},\eqref{opt_control_problem_regular} follows essentially the same way as for problem \eqref{opt_control_ control_problem}-\eqref{state_equ_z} by using Lemma~\ref{Lem:continuity_G_epsilon}, see also Theorem~\ref{Thm:opt_control_existence}. 
%	As in the proof of \cite[Theorem 3.3]{brokate2013optimal} 
	For all $\varepsilon \in (0,\varepsilon_*]$, we deduce from optimality of $(\overline{y}_\varepsilon,\overline{z}_\varepsilon,\overline{u}_\varepsilon)$ for problem \eqref{state_equ_regular_y},\eqref{state_equ_regular_z},\eqref{opt_control_problem_regular} and of $(\overline{y},\overline{z},\overline{u})$ for problem \eqref{opt_control_ control_problem}-\eqref{state_equ_z} that
	\begin{equation}
	J(G_\varepsilon(B_i\overline{u}),\overline{u})= J_{\mathrm{reg}}(G_\varepsilon(B_i\overline{u}),\overline{u};\overline{u})\geq J_{\mathrm{reg}}(\overline{y}_\varepsilon,\overline{u}_\varepsilon;\overline{u}) = J(\overline{y}_\varepsilon,\overline{u}_\varepsilon) + \frac{1}{2} \| \overline{u}_\varepsilon - \overline{u} \|^2_{U_i} \geq J(\overline{y},\overline{u}).\label{eq:proof:conv_controls}
	\end{equation}
	Moreover, by \eqref{estim_reg_problem_y_z} in Lemma~\ref{Lem:estimates_regularized_solution}, $G_\varepsilon(B_i\overline{u}) \in Y_{2,0}$ is uniformly bounded for $\varepsilon \in (0,\varepsilon_*]$ so that $J(G_\varepsilon(B_i\overline{u}),\overline{u}) \leq c$ for some constant $c>0$.
%	Hence, again by optimality of $(\overline{y}_\varepsilon,\overline{z}_\varepsilon,\overline{u}_\varepsilon)$ there holds $0\leq J_{\mathrm{reg}}(\overline{y}_\varepsilon,\overline{u}_\varepsilon;\overline{u})\leq J_{\mathrm{reg}}(0,0;\overline{u})$ for all $\varepsilon \in (0,\varepsilon_*]$. 
%	Now
	Hence,
	$c>J_{\mathrm{reg}}(\overline{y}_\varepsilon,\overline{u}_\varepsilon;\overline{u}) = \frac{1}{2}  \|\overline{y}_\varepsilon-y_d\|_{U_1}^2+ \frac{\kappa}{2}\| \overline{u}_\varepsilon \|_{U_i}^2 + \frac{1}{2}\|\overline{u}_\varepsilon - \overline{u}\|_{U_i}^2$ and the norms of $\overline{u}_\varepsilon$ in $U_i$ are bounded from above 
%	by $\frac{2}{\kappa}J_{\mathrm{reg}}(0,0;\overline{u})$
	independently of $\varepsilon \in (0,\varepsilon_*]$.
%	By \eqref{estim_reg_problem_y_z} in Lemma~\ref{Lem:estimates_regularized_solution} the same holds for all $\overline{y}_\varepsilon$ with $\varepsilon \in (0,\varepsilon_*]$. $J(\overline{y}_\varepsilon,\overline{u})$ is bounded independently of $\varepsilon$.
	Consequently, we can extract a subsequence $\{u_{\varepsilon_k}\}$ which converges weakly in $U_i$ to some $\tilde{u}$ with $k\rightarrow\infty$.
	By Lemma~\ref{Lem:convergence_reg_to_nonreg}, $y_{\varepsilon_k}\rightarrow G(\tilde{u})$ with $k\rightarrow\infty$ weakly in $Y_{0,2}$ and then also in $U_1$.
%	Weak continuity of $J$ and Lemma~\ref{Lem:convergence_reg_to_nonreg} yield that the very left side in $\eqref{eq:proof:conv_controls}$ converges to the very right side as $k\rightarrow \infty$.
	Also by Lemma~\ref{Lem:convergence_reg_to_nonreg}, $G_\varepsilon(B_i\overline{u})\rightarrow \overline{y}$ with $\varepsilon\rightarrow 0$ strongly in $Y_{2,0}$ and then in $U_1$. $J_{\mathrm{reg}}$ is weakly lower semi-continuous.	
	Hence, with \eqref{eq:proof:conv_controls} we obtain
	\begin{equation*}
	\begin{split}
	J(\overline{y},\overline{u}) &=\lim\limits_{k\rightarrow\infty}J(G_{\varepsilon_k}(B_i\overline{u}),\overline{u}) 
	\geq \liminf_{k\rightarrow\infty} J_{\mathrm{reg}}(\overline{y}_{\varepsilon_k},\overline{u}_{\varepsilon_k};\overline{u})
	\geq J(\tilde{y},\tilde{u}) + \frac{1}{2} \| \tilde{u} - \overline{u} \|^2_{U_i} \geq J(\overline{y},\overline{u}).
	\end{split}
	\end{equation*}
	But this implies $\tilde{u}=\overline{u}$ and that the convergence of $\{u_{\varepsilon_k}\}$ in $U_i$ is strong.
	Since the limit is uniquely determined by $\overline{u}$, the whole sequence $\{u_{\varepsilon}\}$ converges to $\overline{u}$ in $U_i$ with $\varepsilon\rightarrow 0$.
	All results then follow by applying the statement about strong convergence in Lemma~\ref{Lem:convergence_reg_to_nonreg}. 
\end{proof}

	\subsection{Gâteaux differentiability of the solution operator of the regularized state equation}
	In this subsection, we show that $G_\varepsilon$ is Gâteaux differentiable for all $\varepsilon\in (0,\varepsilon_*]$.
%	We make use the fact that the derivative of $G_\varepsilon$ is linear in order to derive an adjoint system for problem \eqref{state_equ_regular_y},\eqref{state_equ_regular_z},\eqref{opt_control_problem_regular}, see Theorem~\ref{Thm:opt_syst_reg} below.
	\begin{lemma}[Gâteaux differentiability of $G_\varepsilon$]\label{Lem:Gateau_reg_state_equ}
		Let Assumption~\ref{Ass:general_ass_and_short_notation_1} and Assumption~\ref{Ass:regularized_problem} hold and take the notation from Lemma~\ref{Cor:state_equ_reg_wellposed}.
		Then for any $\varepsilon\in (0,\varepsilon_*]$ and $q\in (\frac{1}{1-\alpha},\infty)$ the solution operator $G_\varepsilon:\mathrm{L}^q(J_T;X)\rightarrow Y_{q,0}$ of problem \eqref{state_equ_regular_abstract} is Gâteaux differentiable.
%		With $y_\varepsilon^u=G_\varepsilon(u)$ and $z_\varepsilon^u=Z_\varepsilon(Sy_\varepsilon^u)$ (see Definition~\ref{Def:Solution_op_regularized_hyst}) 
		The derivative $G_\varepsilon'[u;h]$ at $u\in \mathrm{L}^q(J_T;X)$ in direction $h\in \mathrm{L}^q(J_T;X)$ is given by $y^{u,h}_\varepsilon$, where $y^{u,h}_\varepsilon$ together with $z=z^{u,h}_\varepsilon=Z'_\varepsilon[Sy_\varepsilon^u;Sy^{u,h}_\varepsilon]\in \mathrm{W}^{1,q}(J_T)$ are the unique solution of
%		\begin{equation}
%		\begin{aligned}
%		\dot{y}(t) + (A_p y)(t)&= \frac{\partial}{\partial y}f_{\varepsilon}(y_\varepsilon^u(t),z_\varepsilon^u(t))y(t) + \frac{\partial}{\partial z}f_{\varepsilon}(y_\varepsilon^u(t),z_\varepsilon^u(t))z(t) + h(t)&& \text{in }X,\ t\in J_T,\\
%		y(0)&=0 && \text{in }X,
%		\end{aligned}\label{state_equ_regular_derivative_y}
%		\end{equation}
%		where $z=z^{u,h}_\varepsilon=Z'_\varepsilon[Sy_\varepsilon^u;Sy^{u,h}_\varepsilon]\in \mathrm{W}^{1,q}(J_T)$ is the unique solution of 
		\begin{alignat}{3}
				\dot{y}(t) + (A_p y)(t)&= \frac{\partial}{\partial y}f_{\varepsilon}(y_\varepsilon^u(t),z_\varepsilon^u(t))y(t) + \frac{\partial}{\partial z}f_{\varepsilon}(y_\varepsilon^u(t),z_\varepsilon^u(t))z(t) + h(t)&&\ \text{in }X&&\ \text{for }t\in J_T, \label{state_equ_regular_derivative_y}  \\
				y(0)&=0 &&\ \text{in }X,\notag\\
				\dot{z}(t)-S\dot{y}(t) &= -\frac{1}{\varepsilon}\Psi''(z_\varepsilon^u(t))z(t)&&\ &&\ \text{for }t\in J_T,\label{state_equ_regular_derivative_z}\\
				z(0)&=0.\notag
		\end{alignat}
%		
%		
%		\begin{equation}
%		\begin{aligned}
%		\dot{z}(t)-S\dot{y}(t) &= -\frac{1}{\varepsilon}\Psi''(z_\varepsilon^u(t))z(t)&&\ &&\ \text{for }t\in J_T,\label{state_equ_regular_derivative_z}\\
%		z(0)&=0.
%		\end{aligned}
%		\end{equation}
%		\eqref{state_equ_regular_derivative_z}.
%		Consider the problem
%		\begin{alignat}{3}
%		\dot{y}(t) + (A_p y)(t)&= \frac{\partial}{\partial y}f_{\varepsilon}(y_\varepsilon^u(t),z_\varepsilon^u(t))y(t) + \frac{\partial}{\partial z}f_{\varepsilon}(y_\varepsilon^u(t),z_\varepsilon^u(t))z(t) + h(t)&&\ \text{in }X&&\ \text{for }t\in J_T, \label{state_equ_regular_derivative_y}  \\
%		y(0)&=0 &&\ \text{in }X,\notag\\
%		\dot{z}(t)-S\dot{y}(t) &= -\frac{1}{\varepsilon}\Psi''(z_\varepsilon^u(t))z(t)&&\ &&\ \text{for }t\in J_T,\label{state_equ_regular_derivative_z}\\
%		z(0)&=0.\notag
%		\end{alignat}
%		We denote by \[Z^{u}_\varepsilon:Y_{s,0}\rightarrow\mathrm{W}^{1,s}(J_T),\ v\mapsto Z^{u}_\varepsilon(v)\] the linear and continuous solution mapping for
%		\begin{alignat*}{2}
%		\dot{z}(t)-\dot{v}(t) &= -\frac{1}{\varepsilon}\Psi''(z_\varepsilon^u(t))z(t)&&\ \text{for } t\in J_T\\
%		z(0)&=0.
%		\end{alignat*}
		For $i\in\{1,2\}$ and $u,h\in U_i$ the derivative of the solution mapping $u\mapsto G_\varepsilon(B_iu)$ at $u$ in direction $h$ is given by $y^{B_iu,B_ih}_\varepsilon$, i.e. by the unique solution  of \eqref{state_equ_regular_derivative_y} with $h$ replaced by $B_ih$ and $z=z^{B_iu,B_ih}_\varepsilon=Z'_\varepsilon[Sy_\varepsilon^{B_iu};Sy^{B_iu,B_ih}_\varepsilon]$.
	\end{lemma}
	
	\begin{proof}
		$G_\varepsilon$ is Hadamard directionally differentiable because $Z_\varepsilon$ satisfies the properties of $\mathcal{W}$ in Theorem~\ref{Thm:state_equ_sol_op}.
		Gâteaux differentiability then follows from linearity of all the derivatives. To see that $z^{u,h}_\varepsilon\in \mathrm{W}^{1,q}(J_T)$, insert $Sy^{u,h}$ for $h$ in Remark~\ref{Rem:Solution_op_regularized_hyst} and note that the right side is contained in $\mathrm{W}^{1,q}(J_T)$.
	\end{proof}

\subsection{Adjoint system for the regularized problem}\label{Subsec:Adj_reg_problem}

In this section, we derive adjoint systems for the regularized problems \eqref{state_equ_regular_y},\eqref{state_equ_regular_z},\eqref{opt_control_problem_regular} with $\varepsilon\in (0,\varepsilon_*]$, see Theorem~\ref{Thm:opt_syst_reg} below. We proceed in a similar way as in \cite[Sections 3.3 and 3.5]{brokate2013optimal} and \cite[Section 4]{meyeroptimal}.
The following estimates are needed.
\begin{lemma}\label{Lem:derivative_reg_function}
	Let Assumption~\ref{Ass:general_ass_and_short_notation_1} and Assumption~\ref{Ass:regularized_problem} hold.
	With a little abuse of notation we use the same symbol for the Nemitskii operator of $f_\varepsilon$, i.e. we write
	$
	f_\varepsilon: (y,z) \mapsto f_\varepsilon(y(\cdot),z(\cdot)).
	$
	Then
	$f_\varepsilon$ is locally Lipschitz continuous and Gâteaux differentiable from $\mathrm{C}(\overline{J_T};X^\alpha)\times \mathrm{L}^q(J_T)$ to $\mathrm{L}^q(J_T;X)$ for all $\varepsilon\in (0,\varepsilon_*]$ and $q\in (\frac{1}{1-\alpha},\infty)$.
	
	Moreover, the derivative $f_\varepsilon'[(y,z);(\cdot,\cdot)]$ at $(y,z)\in\mathrm{C}(\overline{J_T};X^\alpha)\times \mathrm{L}^q(J_T)$ is Lipschitz continuous with a modulus of the form $K(y)=L(y) (1+T^{1/q})$, where $L(y)>0$ only depends on $y \in \mathrm{C}(\overline{J_T};X^\alpha)$.
	$K(y)$ and $L(y)$ are independent of $\varepsilon$ and remain the same in a sufficiently small neighbourhood of $y$.
	For $(v,h)\in \mathrm{C}(\overline{J_T};X^\alpha)\times \mathrm{L}^q(J_T)$ we can estimate
	\begin{equation}
	\left\| \frac{\partial}{\partial y} f_\varepsilon(y,z)v\right\|_{\mathrm{L}^q(J_T;X)} + \left\| \frac{\partial}{\partial z} f_\varepsilon(y,z) h\right\|_{\mathrm{L}^q(J_T;X)} \leq K(y) (\|v\|_{\mathrm{C}(\overline{J_T};X^\alpha)} + \|h\|_{\mathrm{L}^q(J_T)})\label{Est:deriv_f_eps1}
	\end{equation}
	For a.e. $t\in J_T$, there also holds the pointwise estimate
	\begin{equation}
	\left\| \frac{\partial}{\partial y} f_\varepsilon(y(t),z(t))v(t)\right\|_{X} + \left\| \frac{\partial}{\partial z} f_\varepsilon(y(t),z(t)) \right\|_X h(t) \leq K(y) (\|v(t)\|_{X^\alpha} + |h(t)|).\label{Est:deriv_f_eps2}
	\end{equation}
	
	Furthermore, $\frac{\partial}{\partial y} f_\varepsilon(y,z)=\frac{\partial}{\partial y} f_\varepsilon(y(\cdot),z(\cdot))$ is bounded by $K(y)$ in $\mathrm{L}^\infty(J_T;\mathcal{L}(X^\alpha,X))$.
	Moreover, $\frac{\partial}{\partial z} f_\varepsilon(y,z) = \frac{\partial}{\partial z} f_\varepsilon(y(\cdot),z(\cdot))$ is bounded by $K(y)$ in $\mathrm{L}^\infty(J_T;X)$.

\end{lemma}

\begin{proof}
	First of all, $f_\varepsilon$ is locally Lipschitz continuous and Gâteaux differentiable from the space $\mathrm{C}(\overline{J_T};X^\alpha)\times \mathrm{L}^q(J_T)$ to $\mathrm{L}^q(J_T;X)$ for all $\varepsilon\in (0,\varepsilon_*]$ and $q\in (\frac{1}{1-\alpha},\infty)$.
	This follows from Step 3 in the proof of \cite[Theorem 3.1]{muench} and Step 1 in the proof of \cite[Theorem 4.7]{muench}. We give a sketch of the proof:
	
	One first makes use of $(A3)_\varepsilon$ in Assumption~\ref{Ass:regularized_problem} to show that $(y(\cdot),v)\mapsto f_\varepsilon(y(\cdot),v)$ is locally Lipschitz continuous from $\mathrm{C}(\overline{J_T};X^\alpha)\times \mathbb{R}$ to $\mathrm{C}(\overline{J_T};X)$ with respect to the $\mathrm{C}(\overline{J_T};X^\alpha)$-norm.
	The proof contains a pointwise estimate of the following form:
	For $y\in \mathrm{C}(\overline{J_T};X^\alpha)$ and some neighbourhood $\overline{B_{\mathrm{C}(\overline{J_T};X^\alpha)}(y,\delta)}$ of $y$ there holds
	\begin{equation*}
	\|f_\varepsilon(y_1(t),z_1) - f(y_2(t),z_2)\|_X\leq L(y) (\|y_1(t) - y_2(t)\|_{X^\alpha} + |z_1-z_2|)
	\end{equation*} 
	for all $y_1,y_2\in \overline{B_{\mathrm{C}(\overline{J_T};X^\alpha)}(y,\delta)}$, $z_1,z_2\in \mathbb{R}$ and $t\in \overline{J_T}$ and for some $L(y)>0$.
	This local estimate leads to a pointwise estimate of the form
	\begin{equation*}
	\begin{split}
	&\Vert f_\varepsilon(y_1,z_1)(s) - f_\varepsilon(y_2,z_2)(s) \Vert_{X} 
	\leq  L(y)\left[  \Vert y_1(s) - y_2(s) \Vert_{X^\alpha} + \vert z_1(s) - z_2(s) \vert \right]
	\end{split}
	\end{equation*}
	for a.e. $s\in J_T$, for any $y_1,y_2\in \overline{B_{\mathrm{C}(\overline{J_T};X^\alpha)}(y,\delta)}$ and $z_1,z_2\in \mathrm{L}^q(J_T)$.
	By Minkowski's inequality, $f_\varepsilon$ is locally Lipschitz continuous from $\mathrm{C}(\overline{J_T};X^\alpha)\times \mathrm{L}^q(J_T)$ to $\mathrm{L}^q(J_T;X)$ with Lipschitz constants of the form $K(y)=L(y) (1+T^{1/q})$.
	
	In a second step one shows that $f_\varepsilon$ is directionally differentiable. 
	Convergence of the difference quotients
	\begin{equation*}
	\lim_{\lambda \rightarrow 0}\frac{f_\varepsilon(y(s)+\lambda v(s),z(s) + \lambda h(s)) }{\lambda} = f_\varepsilon'[(y(s),z(s)); (v(s),h(s)) ]\in X
	\end{equation*}
	for a.e. $s\in J_T$ and $(y,z),(v,h)\in \mathrm{C}(\overline{J_T};X^\alpha)\times \mathrm{L}^q(J_T)$ follows from $(A3)_\varepsilon$ in Assumption~\ref{Ass:regularized_problem}.
	Lebesgue's dominated convergence theorem yields directional differentiability of $f_\varepsilon$ from the space $\mathrm{C}(\overline{J_T};X^\alpha)\times \mathrm{L}^q(J_T)$ to $\mathrm{L}^q(J_T;X)$ and the bounds \eqref{Est:deriv_f_eps1} and \eqref{Est:deriv_f_eps2} for $f_\varepsilon'[(y,z);(\cdot,\cdot)]$.
	Linearity of the derivative and local Lipschitz continuity then already imply Gâteaux differentiability of $f_\varepsilon$.
	
	Now for arbitrary $y\in X^\alpha$ with $\|y\|_{X^\alpha}=1$, we choose the constant function $v\in \mathrm{C}(\overline{J}_T;X^\alpha)$, $v(t)= y$ for $t\in \overline{J_T}$  and set $h= 0\in \mathrm{L}^q(J_T)$ in \eqref{Est:deriv_f_eps2}.
	This implies that $\frac{\partial}{\partial y} f_\varepsilon(y,z)=\frac{\partial}{\partial y} f_\varepsilon(y(\cdot),z(\cdot))$ is bounded by $K(y)$ in $\mathrm{L}^\infty(J_T;\mathcal{L}(X^\alpha,X))$.
	Then we choose $v= 0\in \mathrm{C}(\overline{J}_T;X^\alpha)$, $h\in \mathrm{L}^q(J_T)$, $h(t)= c>0$ for $t\in \overline{J_T}$ in \eqref{Est:deriv_f_eps2} and divide by $c$ on both sides to prove that $\frac{\partial}{\partial z} f_\varepsilon(y,z) = \frac{\partial}{\partial z} f_\varepsilon(y(\cdot),z(\cdot))$ is bounded by $K(y)$ in $\mathrm{L}^\infty(J_T;X)$.
\end{proof}

The following lemma provides the main tool to derive adjoint systems for the regularized problems \eqref{state_equ_regular_y},\eqref{state_equ_regular_z},\eqref{opt_control_problem_regular}.
The hardest part in the proof is to find an explicit expression of the adjoint operator $[G_\varepsilon'[u;\cdot]]^*:Y_{q,0}^*\rightarrow \mathrm{L}^{q'}(J_T;X^*)$ of $G_\varepsilon'[u;\cdot]$ from Lemma~\ref{Lem:Gateau_reg_state_equ}. This comes from the fact that $G_\varepsilon'[u;\cdot]$ is defined as the mapping which assigns to each $h\in \mathrm{L}^q(J_T;X)$ the solution $y_\varepsilon^{u,h}\in Y_{q,0}$ of \eqref{state_equ_regular_derivative_y}, which contains the solution $z_\varepsilon^{u,h}$ of \eqref{state_equ_regular_derivative_z} only implicitly.

\begin{lemma}\label{Lem:adjoint_equation_reg}
	Let Assumption~\ref{Ass:general_ass_and_short_notation_1} and Assumption~\ref{Ass:regularized_problem} hold and adopt the notation from Lemma~\ref{Lem:Gateau_reg_state_equ}. For $\varepsilon\in (0,\varepsilon_*]$ and any $q\in (\frac{1}{1-\alpha},\infty)$, $h\in \mathrm{L}^q(J_T;X)$ and $\nu \in \mathrm{L}^{q'}(J_T;[\mathrm{dom}(A_p)]^*)$ there holds
	\begin{equation*}
	\langle \nu, y_\varepsilon^{u,h} \rangle_{\mathrm{L}^{q}(J_T;\mathrm{dom}(A_p))} = \langle p^\nu_\varepsilon + Sq^\nu_\varepsilon, h \rangle_{\mathrm{L}^q(J_T;X)},
	\end{equation*}
	where  
	$p^\nu_\varepsilon \in Y^*_{q',T}$
	and	
	$q^\nu_\varepsilon \in \mathrm{L}^{q'}(J_T)$ are the unique solution of
	\begin{alignat*}{2}
		-\dot{p} + A^{*}_p p &= \left[\frac{\partial}{\partial y}f_\varepsilon(y_\varepsilon^u,z_\varepsilon^u)\right]^{*}p + S\left[- A_p + \frac{\partial}{\partial y}f_{\varepsilon}(y_\varepsilon^u,z_\varepsilon^u)  \right]q + \nu \ && \text{for } t\in J_T,\\
		p(T)&=0,\\
		-\dot{q} &=  \langle p, \frac{\partial}{\partial z}f_\varepsilon(y_\varepsilon^u,z_\varepsilon^u)\rangle_{X} + S\frac{\partial}{\partial z}f_{\varepsilon}(y_\varepsilon^u,z_\varepsilon^u)q  -\frac{1}{\varepsilon}\Psi''(z_\varepsilon^u)q \ && \text{for } t\in J_T,\\
		q(T)&=0,
	\end{alignat*}
	and where $y_\varepsilon^{u,h}\in Y_{q,0}$ and $z_\varepsilon^{u,h}\in \mathrm{W}^{1,q}(J_T)$ are
	the unique solution of \eqref{state_equ_regular_derivative_y}-\eqref{state_equ_regular_derivative_z}.
	%	\cite[cf.][Lemma 4.10]{meyeroptimal} for a similar statement.
	Moreover,
	\begin{equation}\label{estim_reg_derivative}
		\| y^{u,h}_\varepsilon \|_{Y_{q,0}}\leq C(y_\varepsilon^u) \|h\|_{\mathrm{L}^q(J_T;X)} \text{ and } \|z^{u,h}_\varepsilon\|_{\mathrm{C}(\overline{J_T})} \leq C(y_\varepsilon^u)\|h\|_{\mathrm{L}^q(J_T;X)}
	\end{equation}
	for some constant $C(y_\varepsilon^u)>0$. $C(y_\varepsilon^u)$ remains the same in a sufficiently small neighbourhood of $y_\varepsilon^u$. 
\end{lemma}

\begin{proof}
	
Let $q\in (\frac{1}{1-\alpha},\infty)$ be arbitrary. Consider the solution operator of 
\begin{equation*}
\begin{aligned}
\dot{z}(t) = v(t) + \left(S\frac{\partial}{\partial z}f_{\varepsilon}(y_\varepsilon^u(t),z_\varepsilon^u(t)) - \frac{1}{\varepsilon}\Psi''(z_\varepsilon^u(t))\right) z(t)&& \text{for } t\in J_T,\ z(0)=0,
\end{aligned}
\end{equation*}
which maps any $v\in \mathrm{L}^q(J_T)$ to $z\in\mathrm{W}^{1,q}(J_T)$.
We denote by $T_{z,\varepsilon}^u:\mathrm{L}^q(J_T)\rightarrow \mathrm{L}^q(J_T),\ v\mapsto T_{z,\varepsilon}^u v$ the corresponding operator
on $\mathrm{L}^q(J_T)$.

Consider then the operator
$
T_{y,\varepsilon}^u := A_p - \frac{\partial}{\partial y}f_{\varepsilon}(y_\varepsilon^u,z_\varepsilon^u) - \frac{\partial}{\partial z}f_{\varepsilon}(y_\varepsilon^u,z_\varepsilon^u)T_{z,\varepsilon}^u S \left(-A_p + \frac{\partial}{\partial y}f_{\varepsilon}(y_\varepsilon^u,z_\varepsilon^u) \right)
$
from $Y_{q,0}$ to $\mathrm{L}^q(J_T;X)$.
It follows as for the system \eqref{state_equ_regular_derivative_y}-\eqref{state_equ_regular_derivative_z} that for each $h\in \mathrm{L}^q(J_T;X)$ there exists a unique couple of solutions $(\tilde{y}^{u,h}_\varepsilon,\tilde{z}^{u,h}_\varepsilon)$ in $Y_{q,0}\times \mathrm{L}^q(J_T)$ of the system
	\begin{alignat}{2}
	\dot{y}(t) +(T_{y,\varepsilon}^u y)(t)&= h(t)\ \text{for } t\in J_T, \
	y(0)=0,\label{state_equ_regular_derivative_y_without_h}\\
	z&= T_{z,\varepsilon}^u S\left(-A_p + \frac{\partial}{\partial y}f_{\varepsilon}(y_\varepsilon^u,z_\varepsilon^u) \right)y.\label{state_equ_regular_derivative_z_without_h}
	\end{alignat}

This implies that $\left(\frac{d}{dt} + T_{y,\varepsilon}^u\right)^{-1}$ is bijective from $\mathrm{L}^q(J_T;X)$ to $Y_{q,0}$. Note the difference between \eqref{state_equ_regular_derivative_y_without_h}-\eqref{state_equ_regular_derivative_z_without_h} and \eqref{state_equ_regular_derivative_y}-\eqref{state_equ_regular_derivative_z}.
% since no $h$ appears in \eqref{state_equ_regular_derivative_z_without_h}. 	
We identify $\tilde{z}^{u,h}_\varepsilon\in \mathrm{L}^q(J_T)$ with the corresponding function in $\mathrm{W}^{1,q}(J_T)$ and estimate the norms of $(\tilde{y}^{u,h}_\varepsilon,\tilde{z}^{u,h}_\varepsilon)$.
For $t\in \overline{J_T}$ we have
\begin{equation*}
| \tilde{z}^{u,h}_\varepsilon(t)| = \int_{0}^{t} \frac{\dot{\tilde{z}}^{u,h}_\varepsilon(s)\tilde{z}^{u,h}_\varepsilon(s)}{|\tilde{z}^{u,h}_\varepsilon(s)|} ds = \int_{0}^{t} \frac{-S [(T_{y,\varepsilon}^u \tilde{y}^{u,h}_\varepsilon)(s)]\tilde{z}^{u,h}_\varepsilon(s)}{|\tilde{z}^{u,h}_\varepsilon(s)|} ds -\frac{1}{\varepsilon} \int_{0}^{t} \Psi''(z_\varepsilon^u(s))|\tilde{z}^{u,h}_\varepsilon(s)|\, ds.
\end{equation*}
With \eqref{Est:deriv_f_eps2} in Lemma~\ref{Lem:derivative_reg_function} and (A3) in Assumption~\ref{Ass:general_ass_and_short_notation_1} it follows
\begin{equation*}
\begin{split}
0&\leq |\tilde{z}^{u,h}_\varepsilon(t)| + \frac{1}{\varepsilon} \int_{0}^{t} \Psi''(z_\varepsilon^u(s))|\tilde{z}^{u,h}_\varepsilon(s)| ds \\
&\leq \int_{0}^{t} | S A_p \tilde{y}^{u,h}_\varepsilon(s)| + \left|S\frac{\partial}{\partial y}f_{\varepsilon}(y_\varepsilon^u(s),z_\varepsilon^u(s))\tilde{y}^{u,h}_\varepsilon(s)\right| + \left|S\frac{\partial}{\partial z}f_{\varepsilon}(y_\varepsilon^u(s),z_\varepsilon^u(s))\tilde{z}^{u,h}_\varepsilon(s)\right|ds\\
&\leq (c+\|S\|_{X^*} K(y_\varepsilon^u)) \int_{0}^{t} \| \tilde{y}^{u,h}_\varepsilon(s) \|_{X^\alpha}  + |\tilde{z}^{u,h}_\varepsilon(s)|ds
\end{split}
\end{equation*}
for a constant $c>0$ which is independent of $\varepsilon$.
Note that $\Psi''(z_\varepsilon^u(s))\geq 0$ because $\Psi$ is convex.
Moreover, with \eqref{frac_pow_estimate1} and again \eqref{Est:deriv_f_eps2} in Lemma~\ref{Lem:derivative_reg_function} we obtain
\begin{equation*}
\begin{split}
&\| \tilde{y}^{u,h}_\varepsilon(t) \|_{X^\alpha} \\
&= \left\| \int_{0}^{t} e^{-A_p(t-s)}\left[\frac{\partial}{\partial y}f_{\varepsilon}(y_\varepsilon^u(s),z_\varepsilon^u(s))\tilde{y}^{u,h}_\varepsilon(s) + \frac{\partial}{\partial z}f_{\varepsilon}(y_\varepsilon^u(s),z_\varepsilon^u(t))\tilde{z}^{u,h}_\varepsilon(s) + h(s)\right] ds \right\|_{X^\alpha}\\
&\leq C_\alpha (1 + K(y_\varepsilon^u))e^{(1-\gamma)T} \int_{0}^{t} (t-s)^{-\alpha}[\| \tilde{y}^{u,h}_\varepsilon(s)\|_{X^\alpha} + |\tilde{z}^{u,h}_\varepsilon(s)| + \|h(s)\|_X ] ds.
\end{split}
\end{equation*}

Gronwall's Lemma yields a constant $C_1(y_\varepsilon^u)>0$ which depends only on $y_\varepsilon^u \in \mathrm{C}(\overline{J_T};X^\alpha)$ such that
$
\| \tilde{y}^{u,h}_\varepsilon \|_{\mathrm{C}(\overline{J_T};X^\alpha)}\leq C_1(y_\varepsilon^u) \|h\|_{\mathrm{L}^q(J_T;X)}$ and $\|\tilde{z}^{u,h}_\varepsilon\|_{\mathrm{C}(\overline{J_T})} \leq C_1(y_\varepsilon^u)\|h\|_{\mathrm{L}^q(J_T;X)}
$
for $q\in (\frac{1}{1-\alpha},\infty)$.
Moreover, there holds $C_1(y_\varepsilon^u)=C_1(y)$ for $\varepsilon$ small enough if $\{y_\varepsilon^u\}$ converges to $y$ with $\varepsilon\rightarrow 0$. This is the case for the states $\overline{y}_\varepsilon$ in Theorem~\ref{Thm:conv_minimizers_reg_to_min_original}.
As several times before we use maximal parabolic regularity of $A_p$ to obtain
$
\| \tilde{y}^{u,h}_\varepsilon \|_{Y_{q,0}}\leq C_2(y_\varepsilon^u) \|h\|_{\mathrm{L}^q(J_T;X)}
$
where $C_2(y_\varepsilon^u) >0$ has the same dependence on $y_\varepsilon^u$ as $C_1(y_\varepsilon^u)$.
The inequalities in \eqref{estim_reg_derivative} are shown analogously to the estimates which we derived for $(\tilde{y}^{u,h}_\varepsilon,\tilde{z}^{u,h}_\varepsilon)$.
We also conclude that there exists a constant $C(y_\varepsilon^u)>0$ with
$
\left\|\left( \frac{d}{dt} + T_{y,\varepsilon}^u \right)^{-1}\right\|_{\mathcal{L}(\mathrm{L}^q(J_T;X), Y_{q,0})} \leq C(y_\varepsilon^u).
$
This proves maximal parabolic $\mathrm{L}^q(J_T;X)$-regularity of $T_{y,\varepsilon}^u$ for $q\in (\frac{1}{1-\alpha},\infty)$.
For $\varepsilon$ small enough, also the values $C(y_\varepsilon^u)$ can be chosen independently of $\varepsilon$ if $\{y_\varepsilon^u\}$ converges to some $y$ with $\varepsilon\rightarrow 0$ as it is the case for the sequence $\{\overline{y}_\varepsilon\}$ in Theorem~\ref{Thm:conv_minimizers_reg_to_min_original}.
Maximal parabolic $\mathrm{L}^q(J_T;X)$-regularity of $T_{y,\varepsilon}^u$ for $q\in (\frac{1}{1-\alpha},\infty)$ implies maximal parabolic $\mathrm{L}^{q'}(J_T;[\mathrm{dom}(A_p)]^*)$-regularity of $[T_{y,\varepsilon}^u]^*$
\cite[Lemma 4.10]{meyeroptimal}. To derive a representation of $[T_{y,\varepsilon}^u]^*$, we collect some information about the adjoint mappings of the single components which define $T_{y,\varepsilon}^u$.
Lemma~\ref{Lem:derivative_reg_function} yields that multiplication with $\frac{\partial}{\partial z}f_{\varepsilon}(y_\varepsilon^u,z_\varepsilon^u)$ is well-defined as a mapping from $\mathrm{L}^{q}(J_T)$ into $\mathrm{L}^q(J_T;X)$ and
$
\left[
	\frac{\partial}{\partial z}f_{\varepsilon}(y_\varepsilon^u,z_\varepsilon^u)
\right] ^* = 
\langle
	\cdot, \frac{\partial}{\partial z}f_{\varepsilon}(y_\varepsilon^u,z_\varepsilon^u)
\rangle_{X}.
$
Similarly, $\frac{\partial}{\partial y}f_{\varepsilon}(y_\varepsilon^u,z_\varepsilon^u)$ is a linear continuous mapping from $\mathrm{L}^q(J_T;X^\alpha)$ into $\mathrm{L}^q(J_T;X)$. 
Moreover,
$\left[
S\frac{\partial}{\partial y}f_{\varepsilon}(y_\varepsilon^u,z_\varepsilon^u)
\right] ^*$ is given by multiplication with $S\frac{\partial}{\partial y}f_{\varepsilon}(y_\varepsilon^u,z_\varepsilon^u)$.
The adjoint of $T_{z,\varepsilon}^u$ maps any $v\in \mathrm{L}^{q'}(J_T)$ to the function $q\in \mathrm{L}^{q'}(J_T)$ which may be identified with the unique solution of 
\begin{equation*}
\begin{aligned}
-\dot{q}(t) = v(t) + S\frac{\partial}{\partial z}f_{\varepsilon}(y_\varepsilon^u(t),z_\varepsilon^u(t))q(t) -\frac{1}{\varepsilon}\Psi''(z_\varepsilon^u(t))q(t)&& \text{for } t\in J_T,\
q(T)=0.
\end{aligned}
\end{equation*}

$S^*$ and $[SA_p]^*$ are given
by multiplication with $S$ and $SA_p$.
Furthermore, $SA_p\in [X^\alpha]^*$ by the assumptions on $w$ in (A3) in Assumption~\ref{Ass:general_ass_and_short_notation_1}.
All bounds are independent of $\varepsilon$ if $\overline{y}_\varepsilon$ and $\overline{z}_\varepsilon$ in Theorem~\ref{Thm:conv_minimizers_reg_to_min_original} are considered and if  $\varepsilon$ is small enough.
We obtain
\begin{equation*}
\begin{split}
[T_{y,\varepsilon}^u]^*
&= A_p^* - \left[\frac{\partial}{\partial y}f_{\varepsilon}(y_\varepsilon^u,z_\varepsilon^u)\right]^* - \left[\frac{\partial}{\partial z}f_{\varepsilon}(y_\varepsilon^u,z_\varepsilon^u)T_{z,\varepsilon}^u S \left(-A_p + \frac{\partial}{\partial y}f_{\varepsilon}(y_\varepsilon^u,z_\varepsilon^u) \right) \right]^* \\
&= A_p^* - 
\left[
\frac{\partial}{\partial y}f_{\varepsilon}(y_\varepsilon^u,z_\varepsilon^u)
\right]^* + 
\left[
\left[SA_p\right]^* - 
\left[
S\frac{\partial}{\partial y}f_{\varepsilon}(y_\varepsilon^u,z_\varepsilon^u)
\right] ^*
\right]
[T_{z,\varepsilon}^u]^*
\left[
\frac{\partial}{\partial z}f_{\varepsilon}(y_\varepsilon^u,z_\varepsilon^u)
\right]^*\\
&= A_p^* - 
\left[
\frac{\partial}{\partial y}f_{\varepsilon}(y_\varepsilon^u,z_\varepsilon^u)
\right]^* + 
S\left[
	A_p - \frac{\partial}{\partial y}f_{\varepsilon}(y_\varepsilon^u,z_\varepsilon^u)
\right]
[T_{z,\varepsilon}^u]^*
\langle
., \frac{\partial}{\partial z}f_{\varepsilon}(y_\varepsilon^u,z_\varepsilon^u)
\rangle_{X}.
\end{split}
\end{equation*}

%So we know for given $v\in \mathrm{L}^{s'}(J_T)$ and $q:=[\tilde{Z}^{u}_\varepsilon]^* v $, that 
%\begin{align*}
%	\int_{0}^{T}&\langle [Z^{u}_\varepsilon]^*v , y_u \rangle_{[Y_{s,0}]^*,Y_{s,0}} dt = \int_{0}^{T} v Z^{u}_\varepsilon y_u dt\\
%	&= \int_{0}^{T} v \tilde{Z}^{u}_\varepsilon H y_u dt\\
%	&= \int_{0}^{T} \langle H^*[\tilde{Z}^{u}_\varepsilon]^* v , y_u \rangle_{[X^\alpha]^* \times X^*,X^\alpha \times X} dt\\
%	&= \int_{0}^{T} \langle \left[- [SA_p]^* + \left[\frac{\partial}{\partial y}f_{\varepsilon}(y_\varepsilon^u,\overline{z}_\varepsilon^u) \right]^*S^* \right]q , y_u \rangle_{[X^\alpha]^*,X^\alpha}dt\label{action_of_Z}
%\end{align*}
%So $[Z^{u}_\varepsilon]^*v$ acts on the embedding of $Y_{s,0}$ into $L^s(J_T;\mathrm{D}(A_p))$ as a linear continuous functional.
%The right hand side of \eqref{action_of_Z} defines a linear continuous functional on all of $L^{s}(J_T;\mathrm{D}(A_p))$. Since the norm of $q$ in $\mathrm{L}^{s'}(J_T)$ is bounded by the norm of $v$ in $\mathrm{L}^{s'}(J_T)$, this defines an extension of $[Z^{u}_\varepsilon]^*$ to a linear continuous operator from $\mathrm{L}^{s'}(J_T)$ into $L^{s'}(J_T;[\mathrm{D}(A_p)]^*)$.

Maximal parabolic $\mathrm{L}^{q'}(J_T;[\mathrm{dom}(A_p)]^*)$-regularity of $[T_{y,\varepsilon}^u]^*$
%\begin{align*}
%A_p^* - \left[\frac{\partial}{\partial y}f_{\varepsilon}(y_\varepsilon^u,\overline{z}_\varepsilon^u)\right]^* - \left[\frac{\partial}{\partial z}f_{\varepsilon}(y_\varepsilon^u,\overline{z}_\varepsilon^u)Z^{u}_\varepsilon\right]^* = A_p^* - \left[\frac{\partial}{\partial y}f_{\varepsilon}(y_\varepsilon^u,\overline{z}_\varepsilon^u)\right]^* - [Z^{u}_\varepsilon]^*\left[\frac{\partial}{\partial z}f_{\varepsilon}(y_\varepsilon^u,\overline{z}_\varepsilon^u)\right]^*
%\end{align*}
implies that for each

$\nu \in \mathrm{L}^{q'}(J_T;[\mathrm{dom}(A_p)]^*)$ there exists a unique $p^\nu_\varepsilon\in Y_{q',T}^*$ with
$
	\left(-\frac{d}{dt} + [T_{\varepsilon,y}^u]^*\right)p = \nu.
$

%Because $S^*q\in Y^*_{s',T}$, it also justifies that $T_\varepsilon^*S^*q\in \mathrm{L}^{s'}(J_T;[\mathrm{D}(A_p)]^*)$, so that indeed, its action in \eqref{action_of_Z^{u}_\varepsilon*} is well defined.

For given $\nu \in \mathrm{L}^{q'}(J_T;[\mathrm{dom}(A_p)]^*)$ let $q^\nu_\varepsilon$ be the representative in $\mathrm{L}^{q'}(J_T)$ of the solution of
\begin{alignat*}{2}
-\dot{q}(t) &=  \langle p^\nu_\varepsilon(t), \frac{\partial}{\partial z}f_\varepsilon(y_\varepsilon^u(t),z_\varepsilon^u(t))\rangle_{X} + S\frac{\partial}{\partial z}f_{\varepsilon}(y_\varepsilon^u(t),z_\varepsilon^u(t))q(t)  -\frac{1}{\varepsilon}\Psi''(z_\varepsilon^u(t))q(t)&&\ \text{for } t\in J_T,\\
q(T)&=0.
\end{alignat*}
Let also $(y_\varepsilon^{u,h},z_\varepsilon^{u,h})$ be the solutions of \eqref{state_equ_regular_derivative_y}-\eqref{state_equ_regular_derivative_z} for some given $h\in \mathrm{L}^q(J_T;X)$.
Then we obtain with \eqref{state_equ_regular_derivative_y} and partial integration:
\begin{equation*}
\begin{split}
&\int_{0}^{T} \langle p^\nu_\varepsilon +Sq^\nu_\varepsilon, h \rangle_{X} dt\\
&=\int_{0}^{T} \langle p^\nu_\varepsilon , \dot{y}_\varepsilon^{u,h} + A_p y_\varepsilon^{u,h} - \frac{\partial}{\partial y}f_\varepsilon(y_\varepsilon^u,z_\varepsilon^u)y_\varepsilon^{u,h} - \frac{\partial}{\partial z}f_\varepsilon(y_\varepsilon^u,z_\varepsilon^u)z_\varepsilon^{u,h} \rangle_{X} + \langle Sq^\nu_\varepsilon, h \rangle_{X}dt\\
&=\int_{0}^{T} \langle -\dot{p}^\nu_\varepsilon + A^{*}_p p^\nu_\varepsilon - \left[\frac{\partial}{\partial y}f_\varepsilon(y_\varepsilon^u,z_\varepsilon^u)\right]^{*}p^\nu_\varepsilon , y_\varepsilon^{u,h}  \rangle_{\mathrm{dom}(A_p)} + \langle Sq^\nu_\varepsilon, h \rangle_{X}\\
& - \langle p^\nu_\varepsilon, \frac{\partial}{\partial z}f_\varepsilon(y_\varepsilon^u,z_\varepsilon^u) \rangle_{X}z_\varepsilon^{u,h} dt.
\end{split}
\end{equation*}
By definition of $q^\nu_\varepsilon$ the last term on the right side is equal to
\begin{equation*}
\int_{0}^{T} \left( \dot{q^\nu_\varepsilon} + S\frac{\partial}{\partial z}f_{\varepsilon}(y_\varepsilon^u,z_\varepsilon^u)q^\nu_\varepsilon   -\frac{1}{\varepsilon}\Psi''(z_\varepsilon^u)q^\nu_\varepsilon \right) \overline{z}_\varepsilon^{u,h} dt.
\end{equation*}
Another partial integration together with \eqref{state_equ_regular_derivative_z} and canceling out some terms yields
\begin{equation*}
\begin{split}
\int_{0}^{T} \langle p^\nu_\varepsilon +Sq^\nu_\varepsilon, h \rangle_{X} dt
=&\int_{0}^{T} \langle -\dot{p}^\nu_\varepsilon + A^{*}_p p^\nu_\varepsilon - \left[\frac{\partial}{\partial y}f_\varepsilon(y_\varepsilon^u,z_\varepsilon^u)\right]^{*}p^\nu_\varepsilon , y_\varepsilon^{u,h}  \rangle_{\mathrm{dom}(A_p)} + \langle Sq^\nu_\varepsilon, h \rangle_{X}\\
&-q^\nu_\varepsilon S \left[\left(-A_p + \frac{\partial}{\partial y}f_{\varepsilon}(y_\varepsilon^u,z_\varepsilon^u)\right)y_\varepsilon^{u,h} + h \right]dt.
\end{split}
\end{equation*}
By definition of $p^\nu_\varepsilon$ we finally arrive at
\begin{equation*}
\int_{0}^{T} \langle p^\nu_\varepsilon +Sq^\nu_\varepsilon, h \rangle_{X} dt
=\int_{0}^{T} 
\langle
-\dot{p}^\nu_\varepsilon + [T_{y,\varepsilon}^u]^* p^\nu_\varepsilon , y_\varepsilon^{u,h}
\rangle_{\mathrm{dom}(A_p)} dt 
= \int_{0}^{T} \langle \nu , y_\varepsilon^{u,h} \rangle_{\mathrm{dom}(A_p)}dt.
\end{equation*}

\end{proof}

%We have just proved:
%\begin{lemma}\label{Lem:adjoint_equation_reg}
%	Let Assumption~\ref{Ass:general_ass_and_short_notation_1} and Assumption~\ref{Ass:regularized_problem} hold and let $\varepsilon>0$.
%	For arbitrary  $q\in (\frac{1}{1-\alpha},\infty)$, $h\in \mathrm{L}^q(J_T;X)$ and $\nu \in \mathrm{L}^{q'}(J_T;[\mathrm{dom}(A_p)]^*)$, it holds
%	\[
%	\langle \nu, y_\varepsilon^{u,h} \rangle_{\mathrm{L}^{q}(J_T;\mathrm{dom}(A_p))} = \langle p^\nu_\varepsilon + Sq^\nu_\varepsilon, h \rangle_{\mathrm{L}^q(J_T;X)},
%	\]
%	where  
%	$p^\nu_\varepsilon \in Y^*_{q',T}$
%	and	
%	$q^\nu_\varepsilon \in \mathrm{L}^{q'}(J_T)$ are the unique solution of
%	\begin{alignat*}{2}
%	-\dot{p} + A^{*}_p p &= \left[\frac{\partial}{\partial y}f_\varepsilon(y_\varepsilon^u,z_\varepsilon^u)\right]^{*}p + S\left[- A_p + \frac{\partial}{\partial y}f_{\varepsilon}(y_\varepsilon^u,z_\varepsilon^u)  \right]q + \nu \ && \text{for } t\in J_T,\ p(T)=0,\\
%	-\dot{q} &=  \langle p, \frac{\partial}{\partial z}f_\varepsilon(y_\varepsilon^u,z_\varepsilon^u)\rangle_{X} + S\frac{\partial}{\partial z}f_{\varepsilon}(y_\varepsilon^u,z_\varepsilon^u)q  -\frac{1}{\varepsilon}\Psi''(z_\varepsilon^u)q \ && \text{for } t\in J_T,\ q(T)=0
%	\end{alignat*}
%	and where $y_\varepsilon^{u,h}\in Y_{q,0}$ and $z_\varepsilon^{u,h}\in \mathrm{L}^q(J_T)$ are
%	the unique solution of \eqref{state_equ_regular_derivative_y}-\eqref{state_equ_regular_derivative_z}.
%%	\cite[cf.][Lemma 4.10]{meyeroptimal} for a similar statement.
%\end{lemma}

We can directly write down an adjoint system for a solution $\overline{u}_\varepsilon$ of problem \eqref{state_equ_regular_y},\eqref{state_equ_regular_z},\eqref{opt_control_problem_regular}.
% \cite[cf.][Proposition 4.12]{meyeroptimal}.

\begin{theorem}[Adjoint system regularized problem]\label{Thm:opt_syst_reg}
	Adopt the assumptions of Theorem~\ref{Thm:conv_minimizers_reg_to_min_original} and the notation from Lemma~\ref{Lem:adjoint_equation_reg}. For $i\in \{1,2\}$ and $\varepsilon \in (0,\varepsilon_*]$ let $\overline{u}_\varepsilon \in U_i$ be an optimal control for problem \eqref{state_equ_regular_y},\eqref{state_equ_regular_z},\eqref{opt_control_problem_regular}.	
	Then the adjoint variables for $\overline{y}_\varepsilon\in Y_{2,0}$ and $\overline{z}_\varepsilon\in \mathrm{H}^{1}(J_T)$ are given by 
	$
	p_\varepsilon:= p_\varepsilon^{\overline{y}_\varepsilon - y_d}\in Y^*_{2,T}
	$ and 
%	$
%	q_\varepsilon:= 
%	[T_{z,\varepsilon}^{B_i\overline{u}_\varepsilon}]^*
%	\left(
%	\langle
%	p_\varepsilon, \frac{\partial}{\partial z}f_{\varepsilon}(\overline{y}_\varepsilon,\overline{z}_\varepsilon)
%	\rangle_{X}
%	\right)
%	\in
%	\mathrm{H}^{1}(J_T).
%	$
	$
	q_\varepsilon:= q_\varepsilon^{\overline{y}_\varepsilon - y_d}
	\in
	\mathrm{H}^{1}(J_T).
	$
	There holds
	$
		B_i^*(p_\varepsilon + Sq_\varepsilon)  = -(\kappa+1)\overline{u}_\varepsilon + \overline{u}
	$ in $\mathrm{L}^2(J_T;U_i)$
	and the following system of evolution equations is satisfied by $p_\varepsilon$ and $q_\varepsilon$:
	\begin{alignat}{2}
%	\dot{\overline{y}}_\varepsilon + A_p \overline{y}_\varepsilon &= f_\varepsilon(\overline{y}_\varepsilon,\overline{z}_\varepsilon)  + B_i \overline{u}_\varepsilon,\ &&	\overline{y}_\varepsilon(0)= 0,\notag\\
%	\dot{\overline{z}}_\varepsilon-S\dot{\overline{y}}_\varepsilon &= -\frac{1}{\varepsilon}\Psi'(\overline{z}_\varepsilon),\	&& \overline{z}_\varepsilon(0)=0,\notag\\
	-\dot{p}_\varepsilon + A^{*}_p p_\varepsilon &= \left[\frac{\partial}{\partial y}f_\varepsilon(\overline{y}_\varepsilon,\overline{y}_\varepsilon)\right]^{*}p_\varepsilon + S\left[- A_p + \frac{\partial}{\partial y}f_{\varepsilon}(\overline{y}_\varepsilon,\overline{z}_\varepsilon) 
	 \right]
	q_\varepsilon + \overline{y}_\varepsilon - y_d&&\ \text{for } t\in J_T,\label{equ_p_epsilon}\\
	p_\varepsilon(T)&=0,	\notag\\
	-\dot{q}_\varepsilon &=  \langle p_\varepsilon, \frac{\partial}{\partial z}f_\varepsilon(\overline{y}_\varepsilon,\overline{z}_\varepsilon)\rangle_{X} + S\frac{\partial}{\partial z}f_{\varepsilon}(\overline{y}_\varepsilon,\overline{z}_\varepsilon)q_\varepsilon -\frac{1}{\varepsilon}\Psi''(\overline{z}_\varepsilon)q_\varepsilon&&\ \text{for } t\in J_T,	\label{equ_q_epsilon}\\
	q_\varepsilon(T)&=0\notag.	
	\end{alignat}
\end{theorem}

\begin{proof}
	Note first that we can choose $q=q'=2$ in Lemma~\ref{Lem:adjoint_equation_reg} since $2>\frac{1}{1-\alpha} \Leftrightarrow \alpha <\frac{1}{2}$ which is the case by (A2) in Assumption~\ref{Ass:general_ass_and_short_notation_1}.
	Moreover, the expression $
	\langle \overline{y}_\varepsilon - y_d , y_\varepsilon^{B_i\overline{u}_\varepsilon,B_ih}\rangle_{\mathrm{L}^{2}(J_T;\mathrm{dom}(A_p))}
		= \int_{0}^{T}\int_{\Omega} (\overline{y}_\varepsilon - y_d) \cdot y_\varepsilon^{B_i\overline{u}_\varepsilon,B_ih} dxdt
	$
	is well-defined:
	With $I_p$ as in Definition~\ref{Def:vector_Sobolev_space},
%	$\overline{y}_\varepsilon$ is identified with $I_p^{-1} y_\varepsilon \in \mathrm{L}^2(J_T;\mathbb{W}_{\Gamma_D}^{1,p}(\Omega))$ and	
	$y_\varepsilon^{B_i\overline{u}_\varepsilon,B_ih}\in\mathrm{dom}(A_p)=\mathrm{ran}(I_p)$ may be identified with the embedding of $I_p^{-1}y_\varepsilon^{B_i\overline{u}_\varepsilon,B_ih}$ from $\mathbb{W}_{\Gamma_D}^{1,p}(\Omega)$ into $\mathbb{W}_{\Gamma_D}^{1,p'}(\Omega)\simeq X^*$. Note that $p'\leq2\leq p$.
	Since $\left(\mathcal{A}_p+I_p\right)^{-1}
	\in \mathcal{L}\bigl(X,\mathbb{W}_{\Gamma_D}^{1,p}(\Omega)\bigr)$, see Remark~\ref{Rem:maximal parabolic regularity}, we can first estimate 
	\begin{equation*}
		\begin{split}
		&\left\| I_p^{-1}y_\varepsilon^{B_i\overline{u}_\varepsilon,B_ih}(t) \right\|_{\mathbb{W}_{\Gamma_D}^{1,p'}(\Omega)}
		 \leq \left\| I_p^{-1}y_\varepsilon^{B_i\overline{u}_\varepsilon,B_ih}(t) \right\|_{\mathbb{W}_{\Gamma_D}^{1,p}(\Omega)}\\
		&\leq 
		\left\|
		\left(\mathcal{A}_p+I_p\right)^{-1}
		\right\|_{\mathcal{L}\bigl( X,\mathbb{W}_{\Gamma_D}^{1,p}(\Omega)\bigr)} 
		\left\| y_\varepsilon^{B_i\overline{u}_\varepsilon,B_ih}(t) \right\|_{\mathrm{dom}(A_p)}
		\end{split}
	\end{equation*}
	for a.e. $t\in J_T$ and then with the identification of $\mathbb{W}_{\Gamma_D}^{1,p'}(\Omega)$ with $X^*$
	\begin{equation*}
	\begin{split}
	&\left| \int_{0}^{T}\int_{\Omega} (\overline{y}_\varepsilon - y_d)\cdot y_\varepsilon^{B_i\overline{u}_\varepsilon,B_ih} dxdt \right|  
	= \left| \int_{0}^{T} \langle I_p^{-1}y_\varepsilon^{B_i\overline{u}_\varepsilon,B_ih} , (\overline{y}_\varepsilon - B_1 y_d) \rangle_{X} dt \right|\\
	&\leq \|\overline{y}_\varepsilon - B_1 y_d\|_{\mathrm{L}^2(J_T;X)} \left\|
	\left(\mathcal{A}_p+I_p\right)^{-1}
	\right\|_{\mathcal{L}\bigl(X,\mathbb{W}_{\Gamma_D}^{1,p}(\Omega)\bigr)}\|y_\varepsilon^{B_i\overline{u}_\varepsilon,B_ih}\|_{\mathrm{L}^2(J_T;\mathrm{dom}(A_p))}.
	\end{split}
	\end{equation*}

	The Gâteaux-derivative of 	
	$
	\mathcal{J}_\mathrm{reg}(u):= J_{\mathrm{reg}}(G_\varepsilon(B_iu),u;\overline{u})= J(G_\varepsilon(B_iu),u) + \frac{1}{2}\|u-\overline{u}\|_{U_i}^2
	$
	with respect to $u$ has to be zero at $\overline{u}_\varepsilon$ by optimality.
	Applying Lemma \ref{Lem:adjoint_equation_reg} we compute for $h\in U_i$:
	\begin{equation*}
	\begin{split}
	0&=\mathcal{J}'_\mathrm{reg}[\overline{u}_\varepsilon;h] = \langle \overline{y}_\varepsilon - y_d , y_\varepsilon^{B_i\overline{u}_\varepsilon,B_ih} \rangle_{\mathrm{L}^{2}(J_T;\mathrm{dom}(A_p))} 
	+ \kappa\langle \overline{u}_\varepsilon , h \rangle_{U_i} + \langle \overline{u}_\varepsilon - \overline{u} , h \rangle_{U_i}\\
	&= \langle (p_\varepsilon + S^*q_\varepsilon) , B_ih\rangle_{\mathrm{L}^{2}(J_T;X)} 
	+ \langle (\kappa+1) u_\varepsilon - \overline{u} , h \rangle_{U_i}
	= \langle B_i^*(p_\varepsilon + Sq_\varepsilon) + (\kappa+1) u_\varepsilon - \overline{u} , h \rangle_{U_i}.
	\end{split}
	\end{equation*}
\end{proof}

\subsection{Estimates for the adjoints of the regularized problem}\label{Subsec:estimates_adjoints}
Similar to \cite[Section 3.5]{brokate2013optimal} and \cite[Lemma 4.14]{meyeroptimal} we estimate the norms of the adjoint states $p_\varepsilon$ and $q_\varepsilon$ from Theorem~\ref{Thm:opt_syst_reg} independently of $\varepsilon$ and of the norms of the optimal controls $\overline{u}_\varepsilon$. In Section~\ref{Sec:Limit}, we take a sequence $\{\varepsilon\}$ with $\varepsilon\rightarrow 0$ and use those bounds to extract (weakly) converging subsequences of $p_\varepsilon$ and $q_\varepsilon$. Those finally yield an adjoint system for problem \eqref{opt_control_ control_problem}-\eqref{state_equ_z}, see Theorem~\ref{Thm:opt_control_boundary_opt_cond_limit_adjoint} below.

\begin{lemma}[Uniform bounds]\label{Lem:estim_adjoints_regularized_problem}
	Adopt the assumptions and the notation of Theorem~\ref{Thm:opt_syst_reg}.
	There exists a constant $c>0$ which is independent of $\varepsilon$ and some $\varepsilon_0\in (0,\varepsilon_*]$ such that the following holds true. If $\varepsilon \in (0,\varepsilon_0)$, then 
	\begin{align}
	0\leq\|q_\varepsilon\|_{\mathrm{C}(\overline{J_T})} + \frac{1}{\varepsilon}\int_{0}^{T} \Psi''(\overline{z}_\varepsilon(s)) |q_\varepsilon(s)| ds &\leq c,\label{estim_q_epsilon}\\
	\int_{0}^{T}|\dot{q}_\varepsilon(s)| ds 
	&\leq c,\label{estim_dot(q)_epsilon}\\
	\|p_\varepsilon\|_{Y^*_{2,T}}
	&\leq c,\label{estim_p_epsilon}\\
	\left\| \left[\frac{\partial}{\partial y}f_\varepsilon(\overline{y}_\varepsilon,\overline{z}_\varepsilon)\right]^{*}p_\varepsilon \right\|_{\mathrm{L}^{2}(J_T; [X^\alpha]^*)}\label{estim_f_y,epsilon_p_epsilon}
	&\leq c, \\
	\left\| S\frac{\partial}{\partial y}f_\varepsilon(\overline{y}_\varepsilon,\overline{z}_\varepsilon) q_\varepsilon \right\|_{\mathrm{L}^{2}(J_T; [X^\alpha]^*)}
	&\leq c,
	\text{ as well as }\label{estim_S_f_y,epsilon_q_epsilon}\\
	\left\| SA_p q_\varepsilon \right\|_{\mathrm{C}(\overline{J_T}; [X^\alpha]^*)}
	&\leq c.\label{estim_SA_p_q_epsilon}
	\end{align}	
\end{lemma}

\begin{proof}
%	We start with $p_\varepsilon$.
Firstly, Theorem \ref{Thm:conv_minimizers_reg_to_min_original} yields $\overline{u}_\varepsilon \rightarrow \overline{u}$ in $U_i$, $\overline{y}_\varepsilon \rightarrow \overline{y}$ in $Y_{2,0}$ and in $\mathrm{C}(\overline{J_T};X^\alpha)$ and $\overline{z}_\varepsilon \rightarrow \overline{z}$ weakly in $\mathrm{H}^1(J_T)$ and strongly in $\mathrm{C}(\overline{J_T})$. 
%The proceeding is analogous to the proof of \cite[Lemma 4.14]{meyeroptimal}.
As in the proof of Theorem~\ref{Thm:opt_syst_reg} we obtain that $\overline{y}_\varepsilon - y_d$ is bounded in $\mathrm{L}^{2}(J_T;[\mathrm{dom}(A_p)]^*)$ by
$\|\overline{y}_\varepsilon - B_1 y_d\|_{\mathrm{L}^2(J_T;X)} 
\left\|
\left(\mathcal{A}_p+I_p\right)^{-1}
\right\|_{\mathcal{L}\bigl(X;\mathbb{W}_{\Gamma_D}^{1,p}(\Omega)\bigr)} =:c_0 $.
This constant can be estimated independently of $\varepsilon$ because $\{\overline{y}_\varepsilon\}$ is uniformly bounded in $\mathrm{C}(\overline{J_T};X)$.
For any $\xi\in \mathrm{L}^{2}(J_T;X)$, Lemma~\ref{Lem:adjoint_equation_reg} yields
\begin{equation*}
\langle p_\varepsilon + S q_\varepsilon, \xi \rangle_{\mathrm{L}^{2}(J_T;X)}= \langle \overline{y}_\varepsilon - y_d , y_\varepsilon^{B_i \overline{u}_\varepsilon, \xi}\rangle_{\mathrm{L}^{2}(J_T;\mathrm{dom}(A_p))}\leq c_0
 C(\overline{y}_\varepsilon)\|\xi\|_{\mathrm{L}^2(J_T;X)}.
% &\leq \|\overline{y}_\varepsilon - B_1 y_d\|_{\mathrm{L}^2(J_T;X)} \left\|
% \left(\mathcal{A}_p+I_p\right)^{-1}
% \right\|_{\mathcal{L}\left(X;\mathbb{W}_{\Gamma_D}^{1,p}(\Omega)\right)}
% C(\overline{y}_\varepsilon)\|\xi\|_{\mathrm{L}^2(J_T;X)}
\end{equation*}
Because $\overline{y}_\varepsilon \rightarrow \overline{y}$ in $\mathrm{C}(\overline{J_T};X^\alpha)$ we can find some $\varepsilon_0>0$ such that $C(\overline{y}_\varepsilon)=C(\overline{y})$ for all $\varepsilon\in (0,\varepsilon_0)$.
%The last estimate follows from the convergence of $\overline{y}_\varepsilon$ to $\overline{y}$ in $\mathrm{C}(\overline{J_T};X^\alpha)$.
From reflexivity of $\mathrm{L}^{2}(J_T;X)$ we conclude
\begin{equation}
\| p_\varepsilon + Sq_\varepsilon \|_{\mathrm{L}^{2}(J_T;X^*)} \leq  c_0C(\overline{y}) =: c_1\label{estim_p_epsilon+S_q_epsilon}
\end{equation}
for all $\varepsilon\in (0,\varepsilon_0)$.
%\cite[cf.][Lemma 4.14]{meyeroptimal}.
We continue with estimates for $q_\varepsilon$.
%Similar to \cite[Section 3.5]{brokate2013optimal} 
We test \eqref{equ_q_epsilon} with $q_\varepsilon/|q_\varepsilon|$, integrate from any $t\in J_T$ to $T$ and apply \eqref{Est:deriv_f_eps1} from Lemma~\ref{Lem:derivative_reg_function} and \eqref{estim_p_epsilon+S_q_epsilon} to get
\begin{equation}\label{proof_derivation_estim_dot(q)_epsilon}
\begin{split}
|q_\varepsilon(t)| + \frac{1}{\varepsilon}\int_{t}^{T} \Psi''(\overline{z}_\varepsilon(s)) |q_\varepsilon(s)| ds &= \int_{t}^{T} \langle p_\varepsilon(s) + Sq_\varepsilon(s) , \frac{\partial}{\partial z}f_\varepsilon(\overline{y}_\varepsilon(s),\overline{z}_\varepsilon(s))\rangle_{X}\frac{q_\varepsilon(s)}{|q_\varepsilon(s)|}\, ds\\
&\leq c_1 \left\|\frac{\partial}{\partial z}f_\varepsilon(y_\varepsilon,\overline{z}_\varepsilon)\right\|_{\mathrm{L}^2(J_T;X)}
\leq c_1K(\overline{y}_\varepsilon).
\end{split}
\end{equation}
W.l.o.g. for the same $\varepsilon_0$ as before there holds $c_1K(\overline{y}_\varepsilon)= c_1K(\overline{y}) =: c_2$ for all $\varepsilon\in (0,\varepsilon_0)$.
Note that $\Psi''(\overline{z}_\varepsilon)\geq 0$ by convexity of $\Psi$.
This yields
\begin{equation}
0\leq\|q_\varepsilon\|_{\mathrm{C}(\overline{J_T})} + \frac{1}{\varepsilon}\int_{0}^{T} \Psi''(\overline{z}_\varepsilon(s)) |q_\varepsilon(s)| ds \leq c_2 \label{proof_estim_dot(q)_epsilon}
\end{equation}
for all $\varepsilon\in (0,\varepsilon_0)$.
We conclude $Sq_\varepsilon \in \mathrm{L}^{2}(J_T;X^*)$ and then by \eqref{estim_p_epsilon+S_q_epsilon} also $p_\varepsilon \in \mathrm{L}^{2}(J_T;X^*)$, both with a norm which is independent of $\varepsilon\in (0,\varepsilon_0)$.
We continue by estimating
\begin{equation*}
	\int_{0}^{T}|\dot{q}_\varepsilon(s)| ds \leq \int_{0}^{T} |\langle p_\varepsilon(s) + Sq_\varepsilon(s), \frac{\partial}{\partial z}f_\varepsilon(\overline{y}_\varepsilon(s),\overline{z}_\varepsilon(s))\rangle_{X}| ds+ \frac{1}{\varepsilon}\int_{0}^{T} \Psi''(\overline{z}_\varepsilon(s)) |q_\varepsilon(s)| ds.
\end{equation*}
Because of \eqref{proof_derivation_estim_dot(q)_epsilon} the right side is bounded by $2c_2$ so that
$
\int_{0}^{T}|\dot{q}_\varepsilon(s)| ds 
\leq 2c_2=:c_3
$
for $\varepsilon\in (0,\varepsilon_0)$.

To proceed, we
% as in the proof of \cite[Lemma 4.14]{meyeroptimal} 
use maximal parabolic $\mathrm{L}^{2}(J_T;[\mathrm{dom}(A_p)]^*)$-regularity of $A_p^*$ and \eqref{equ_p_epsilon} to obtain
% together with Lemma~\ref{Lem:derivative_reg_function}, \eqref{estim_p_epsilon+S_q_epsilon}, \eqref{proof_estim_dot(q)_epsilon} and the boundedness of $\overline{y}_\varepsilon - y_d$ that
\begin{equation*}
\begin{split}
\|p_\varepsilon\|_{Y^*_{2,T}}\notag &\leq 
\left\|\left(-\frac{d}{dt} + A_p^*\right)^{-1}
\right\|_{\mathcal{L}(\mathrm{L}^{2}(J_T;[\mathrm{dom}(A_p)]^*),Y^*_{2,T})}\notag\\
& \left\| \left[\frac{\partial}{\partial y}f_\varepsilon(\overline{y}_\varepsilon,\overline{z}_\varepsilon)\right]^{*}p_\varepsilon + S\left[- A_p + \frac{\partial}{\partial y}f_{\varepsilon}(\overline{y}_\varepsilon,\overline{z}_\varepsilon) \right]q_\varepsilon + \overline{y}_\varepsilon - y_d \right\|_{\mathrm{L}^{2}(J_T;[\mathrm{dom}(A_p)]^*)}\notag\\
&\leq 
\left\|\left(-\frac{d}{dt} + A_p^*\right)^{-1}
\right\|_{\mathcal{L}(\mathrm{L}^{2}(J_T;[\mathrm{dom}(A_p)]^*),Y^*_{2,T})}\notag\\
& \left( \left\| \frac{\partial}{\partial y}f_\varepsilon(\overline{y}_\varepsilon,\overline{z}_\varepsilon) \right\|_{\mathcal{L}(\mathrm{L}^{2}(J_T;X^\alpha),\mathrm{L}^{2}(J_T;X))} \|p_\varepsilon + Sq_\varepsilon\|_{\mathrm{L}^{2}(J_T;X^*)}\right.\notag\\
&\left. +\|SA_p\|_{[X^\alpha]^*}\|q_\varepsilon\|_{\mathrm{C}(\overline{J_T})} + \left\| \overline{y}_\varepsilon - y_d \right\|_{\mathrm{L}^{2}(J_T;[\mathrm{dom}(A_p)]^*)} \right).
\end{split}
\end{equation*}
\eqref{Est:deriv_f_eps1} from Lemma~\ref{Lem:derivative_reg_function}, \eqref{estim_p_epsilon+S_q_epsilon}, \eqref{proof_estim_dot(q)_epsilon} and the bound $\left\| \overline{y}_\varepsilon - y_d \right\|_{\mathrm{L}^{2}(J_T;[\mathrm{dom}(A_p)]^*)}\leq c_0 $ yield 
\begin{equation*}
\|p_\varepsilon\|_{Y^*_{2,T}}\leq
\left\|\left(-\frac{d}{dt} + A_p^*\right)^{-1}
\right\|_{\mathcal{L}(\mathrm{L}^{2}(J_T;[\mathrm{dom}(A_p)]^*),Y^*_{2,T})}\left(
c_1K(\overline{y}) + \|SA_p\|_{[X^\alpha]^*}c_2 + c_0
\right)=:
c_4
\end{equation*}
for $\varepsilon\in (0,\varepsilon_0)$.
%Also as in the proof of \cite[Lemma 4.14]{meyeroptimal} 
In a similar way one obtains \eqref{estim_f_y,epsilon_p_epsilon}-\eqref{estim_SA_p_q_epsilon} from the estimates
\begin{align*}
\left\| \left[\frac{\partial}{\partial y}f_\varepsilon(\overline{y}_\varepsilon,\overline{z}_\varepsilon)\right]^{*}p_\varepsilon \right\|_{\mathrm{L}^{2}(J_T; [X^\alpha]^*)}
&\leq \left\| \frac{\partial}{\partial y}f_\varepsilon(\overline{y}_\varepsilon,\overline{z}_\varepsilon) \right\|_{\mathcal{L}(\mathrm{L}^{2}(J_T;X^\alpha),\mathrm{L}^{2}(J_T;X))} \|p_\varepsilon\|_{\mathrm{L}^{2}(J_T;X^*)},\\
\left\| S\frac{\partial}{\partial y}f_\varepsilon(\overline{y}_\varepsilon,\overline{z}_\varepsilon) q_\varepsilon \right\|_{\mathrm{L}^{2}(J_T; [X^\alpha]^*)}
&\leq \left\| \frac{\partial}{\partial y}f_\varepsilon(\overline{y}_\varepsilon,\overline{z}_\varepsilon) \right\|_{\mathcal{L}(\mathrm{L}^{2}(J_T;X^\alpha),\mathrm{L}^{2}(J_T;X))} \|Sq_\varepsilon\|_{\mathrm{L}^{2}(J_T;X^*)}\intertext{ and }
\left\| SA_p q_\varepsilon \right\|_{\mathrm{C}(\overline{J_T}; [X^\alpha]^*)} 
&\leq \|SA_p\|_{[X^\alpha]^*}\|q_\varepsilon\|_{\mathrm{C}(\overline{J_T})}.
\end{align*}	
\end{proof}

\section{Adjoint system and optimality conditions for the optimal control problem}\label{Sec:Limit}
As in \cite[Section 4]{brokate2013optimal} and \cite[Theorem 4.15]{meyeroptimal} we are interested in taking the limit $\varepsilon \rightarrow 0$ in Theorem~\ref{Thm:opt_syst_reg} to obtain an adjoint system for problem \eqref{opt_control_ control_problem}-\eqref{state_equ_z}.
%We divide our results into several subsections and lemmas.
In Subsection~\ref{Subsec:Adjoint_system_distributed_or_boundary}--Subsection~\ref{Subsec:Summary_Adjoint_and_opt_cond_distributed_or_boundary} we study the general case with spatially distributed or boundary controls, i.e. $i\in\{1,2\}$.
Particularly, in Subsection~\ref{Subsec:Adjoint_system_distributed_or_boundary} we derive an adjoint system $(p,q)$ for problem \eqref{opt_control_ control_problem}-\eqref{state_equ_z} for the optimal control $\overline{u}$ from Theorem~\ref{Thm:conv_minimizers_reg_to_min_original}, see Lemma~\ref{Lem:limits_q_and_p_and_props_of_p}. 
%In Lemma~\ref{Lem:q_in_H1}, Lemma~\ref{Lem:play_times_q_equ_zero}, Lemma~\ref{Lem:overall_evol_q}, Lemma~\ref{Lem:q_jumps_down_in_rev_time} and Lemma~\ref{Lem:delta_0_switching} 
Moreover, we gather information about the continuity properties of $q$.
Subsection~\ref{Subsec:opt_cond_distributed_or_boundary} contains the optimality conditions for problem \eqref{opt_control_ control_problem}-\eqref{state_equ_z} for the optimal control $\overline{u}$ in terms of the pair $p$ and $q$, see Lemma~\ref{Lem:opt_cond_general} below. 
In Subsection~\ref{Subsec:Summary_Adjoint_and_opt_cond_distributed_or_boundary} we summarize the results from Subsection~\ref{Subsec:Adjoint_system_distributed_or_boundary}--Subsection~\ref{Subsec:opt_cond_distributed_or_boundary} in Theorem~\ref{Thm:opt_control_boundary_opt_cond_limit_adjoint}. 
Afterwards, we consider the particular case when $f$ is continuously differentiable. In Corollary~\ref{Cor:opt_cond_general} we improve the optimality condition \eqref{opt_control_opt_cond_general} from Theorem~\ref{Thm:opt_control_boundary_opt_cond_limit_adjoint} for this instance.
Both optimality conditions \eqref{opt_control_opt_cond_general} and \eqref{opt_control_opt_cond_regular_f} are restricted to test functions $Sy^{B_i\overline{u},B_ih}$ with $h\in U_i$, $i\in\{1,2\}$.

In Subsection~\ref{Subsec:Improvement_distributed} we focus on the setting when the controls act inside of $\Omega$, i.e. on $i=1$. 
In Corollary~\ref{Cor:opt_control_distrib_opt_cond} we improve the optimality conditions from
Theorem~\ref{Thm:opt_control_boundary_opt_cond_limit_adjoint} as well as those from Corollary~\ref{Cor:opt_cond_general} by extending inequalities \eqref{opt_control_opt_cond_general} and \eqref{opt_control_opt_cond_regular_f} to any test function of the form $(Sv)\varphi$ with $v\in \mathrm{dom}(A_p),$ $Sv>0$ and $\varphi \in \mathrm{C}_0^\infty(J_T)$. Dividing the corresponding inequality by $Sv$ yields, at least in \eqref{opt_control_opt_cond_regular_f}, an optimality condition with arbitrary test functions  $\varphi \in \mathrm{C}_0^\infty(J_T)$.
For $i=1$ we also prove uniqueness of $p$ and $q$ if $f$ is continuously differentiable, see Corollary~\ref{Cor:Uniqueness_dist_contr}.

\subsection{Adjoint system for distributed or boundary controls}\label{Subsec:Adjoint_system_distributed_or_boundary}
In this subsection, we derive an adjoint system $(p,q)$ for problem \eqref{opt_control_ control_problem}-\eqref{state_equ_z} and collect regularity properties of $p$ and $q$.
The evolution equation of $p$ can be derived pretty much straight forward as the limit equation of \eqref{equ_p_epsilon} for $\varepsilon \rightarrow 0$, see Lemma~\ref{Lem:limits_q_and_p_and_props_of_p} below. This is not possible for $q$. The reason is that in Lemma~\ref{Lem:estim_adjoints_regularized_problem} we could bound the norm of $\dot{q_\varepsilon}$ independently of $\varepsilon$ only in $\mathrm{L}^1(J_T)$. As a remedy we split the interval $J_T$ into the set $I_0$ of times $t$ where the limit $\overline{z}(t)$ is contained in the open interval $(a,b)$ and the rest $I_\partial$ where $\overline{z}(t)\in \{a,b\}$. It turns out, that the evolution of $q$ in $I_0$ can be described in form of an evolution equation, see Lemma~\ref{Lem:q_in_H1} below. 
As for $I_\partial$, we have to pass to weak-$*$ convergence of $q_\varepsilon$ and consider the limit $d\mu$ of $\frac{1}{\varepsilon} \Psi''(\overline{z}_\varepsilon) q_\varepsilon$ in $\mathrm{C}(\overline{J_T})^*$. Driving $\varepsilon$ to zero then yields an equality for $dq$ in the sense of measures on $I_\partial$, see Lemma~\ref{Lem:overall_evol_q}. The abstract measure $d\mu$, having support in $I_\partial$, remains part of this evolution equation. It also appears in the optimality conditions for problem \eqref{opt_control_ control_problem}-\eqref{state_equ_z} in \eqref{opt_control_opt_cond_general}.
In order to complete the description of $q$ by analyzing the measure $d\mu$, we will introduce a regularity Assumption~\ref{Ass:regularity_assumption} on $S\overline{y}(t)$ for $t\in I_\partial$. With this assumption, we can characterize $d\mu$ in a subset of $I_\partial$. This allows us to characterize $q$ in open subintervals of $I_\partial$ and we can prove continuity of $q$ at so-called $(0,\partial)$-switching times, see Lemma~\ref{Lem:delta_0_switching}.
In Remark~\ref{Rem:delta_0_switching} we generalize Lemma~\ref{Lem:delta_0_switching} for when Assumption~\ref{Ass:regularity_assumption} is not satisfied.

%First we take the limit $\varepsilon \rightarrow 0$ in Theorem~\ref{Thm:opt_syst_reg} to obtain:
\begin{lemma}[Adjoint system in the limit]\label{Lem:limits_q_and_p_and_props_of_p}
	Adopt the assumptions and the notation of Theorem~\ref{Thm:opt_syst_reg}. 
	For $i\in\{1,2\}$ let $\overline{u}\in U_i,$ $\overline{y}=G(\overline{u})$ and $\overline{z}=\mathcal{W}[S\overline{y}]$ be defined as in Theorem~\ref{Thm:conv_minimizers_reg_to_min_original}. Then
	every sequence $\{\varepsilon\}$ with $\varepsilon\rightarrow 0$ has a subsequence $\{\varepsilon_k\}$ such that the following holds true.
	There exist functions functions $p\in Y^*_{2,T}$ and $\lambda_1,\lambda_2 \in \mathrm{L}^{2}(J_T; [X^\alpha]^*)$ such that as $k\rightarrow \infty$,
	$p_{\varepsilon_k} \rightharpoonup p \text{ in } Y^*_{2,T}$ and
	\begin{equation*}
	\begin{aligned}
	\left[\frac{\partial}{\partial y}f_{\varepsilon_k}(\overline{y}_{\varepsilon_k},\overline{z}_{\varepsilon_k})\right]^{*}p_{\varepsilon_k} &\rightharpoonup \lambda_1 &&\text{in } \mathrm{L}^{2}(J_T; [X^\alpha]^*),\\
	S\frac{\partial}{\partial y}f_{\varepsilon_k}(\overline{y}_{\varepsilon_k},\overline{z}_{\varepsilon_k})q_{\varepsilon_k} &\rightharpoonup \lambda_2 
	&&\text{in } \mathrm{L}^{2}(J_T; [X^\alpha]^*).
	\end{aligned}
	\end{equation*}
	
	Moreover, there exists a function $q$ which has bounded variation, i.e. $q\in \mathrm{BV}(J_T)$, such that $
	q_{\varepsilon_k}$ converges pointwise to $q$ with $k\rightarrow \infty$. There holds $\mathrm{Var}(q) \leq \liminf_{\varepsilon_k\rightarrow 0} \mathrm{Var}(q_{\varepsilon_k})$.
	Alternatively,
	$
	\dot{q}_{\varepsilon_k} \rightarrow dq$ weak-* in $\mathrm{C}(\overline{J_T})^*$ with $k\rightarrow \infty$
	for some signed regular Borel measure $dq\in \mathrm{C}(\overline{J_T})^*$.
	%As in \cite[Section 4]{brokate2013optimal} 
	The relation between $q$ and $dq$ is given by
	$
	q(t-)-q(s+) = dq((s,t))
	$ and 
	$
	q(t+)-q(s-) = dq([s,t])
	$
	for $[s,t]\subset \overline{J_T}$.
	
	The function $p$ solves the evolution equation
	\begin{equation}
	\begin{aligned}
	-\dot{p} + A^{*}_p = \lambda_1 + \lambda_2 -SA_pq + \overline{y} - y_d&& \text{for } t\in J_T,\ p(T)=0.\label{eq:evolution_p}
	\end{aligned}
	\end{equation}
	If $f$ is continuously differentiable from $X^\alpha\times \mathbb{R}$ into $X$ then
	$
	\lambda_1= \left[\frac{\partial}{\partial y}f(\overline{y},\overline{z})\right]^{*}p \text{ and } \lambda_2= S\frac{\partial}{\partial y}f(\overline{y},\overline{z})q.
	$
	Furthermore,
	\begin{equation}
	\begin{aligned}
	B_i^*(p+Sq) = -\kappa \overline{u}&& \text{in } U_i.\label{opt_control_adj_eq_control}
	\end{aligned}
	\end{equation}
\end{lemma}

\begin{proof}
Theorem~\ref{Thm:conv_minimizers_reg_to_min_original} implies $u_\varepsilon\rightarrow \overline{u}$ in $U_i$, $\overline{y}_\varepsilon\rightarrow \overline{y}$ in $Y_{2,0}$ and in $\mathrm{C}(\overline{J_T};X^\alpha)$ and $\overline{z}_\varepsilon\rightarrow \overline{z}$ uniformly and weakly in $\mathrm{H}^1(J_T)$ with $\varepsilon\rightarrow 0$.
By \eqref{estim_p_epsilon}, \eqref{estim_f_y,epsilon_p_epsilon} and \eqref{estim_S_f_y,epsilon_q_epsilon} in Lemma~\ref{Lem:estim_adjoints_regularized_problem},
%, which was the boundedness of $p_\varepsilon$ in $Y^*_{2,T}$ and of $\left[\frac{\partial}{\partial y}f_\varepsilon(\overline{y}_\varepsilon,\overline{z}_\varepsilon)\right]^{*}p_\varepsilon$ and $S\frac{\partial}{\partial y}f_\varepsilon(\overline{y}_\varepsilon,\overline{z}_\varepsilon) q_\varepsilon$ in $\mathrm{L}^{2}(J_T; [X^\alpha]^*)$, 
reflexivity of all spaces yields 
%as in the proof of \cite[Theorem 4.15]{meyeroptimal} 
a subsequence $\{\varepsilon_k\}$ and some functions $p$, $\lambda_1$ and $\lambda_2$ such that
$
p_{\varepsilon_k} \rightharpoonup p \text{ in } Y^*_{2,T}$,
$
\left[\frac{\partial}{\partial y}f_{\varepsilon_k}(\overline{y}_{\varepsilon_k},\overline{z}_{\varepsilon_k})\right]^{*}p_{\varepsilon_k} \rightharpoonup \lambda_1 \text{ in } \mathrm{L}^{2}(J_T; [X^\alpha]^*)$ and
$S\frac{\partial}{\partial y}f_{\varepsilon_k}(\overline{y}_{\varepsilon_k},\overline{z}_{\varepsilon_k})q_{\varepsilon_k} \rightharpoonup \lambda_2 \text{ in } \mathrm{L}^{2}(J_T; [X^\alpha]^*)
$
with $k\rightarrow\infty$.
The condition $p(T)=0$ is included in the definition of the space $Y^*_{2,T}$.
From \eqref{estim_dot(q)_epsilon} we conclude that $q_\varepsilon$ has bounded variation, i.e. $q_\varepsilon \in \mathrm{BV}(J_T)$, with a norm which is bounded independently of $\varepsilon$. This implies that (w.l.o.g. the same) subsequence
$
q_{\varepsilon_k}$ converges pointwise to some $q\in \mathrm{BV}(J_T)$ with $k\rightarrow\infty$ and $\mathrm{Var}(q) \leq \liminf_{\varepsilon_k\rightarrow 0} \mathrm{Var}(q_{\varepsilon_k})$.
Alternatively, by Alaoglu's compactness theorem,
$
\dot{q}_{\varepsilon_k} \rightarrow dq$ weak-* in $\mathrm{C}(\overline{J_T})^*$ with $k\rightarrow\infty$
for some signed regular Borel measure $dq\in \mathrm{C}(\overline{J_T})^*$ and the relation between $q$ and $dq$ is given by
$
q(t-)-q(s+) = dq((s,t))
$ and 
$
q(t+)-q(s-) = dq([s,t])
$
for $[s,t]\subset \overline{J_T}$ \cite[Section 4]{brokate2013optimal}.
We exploit weak continuity of $-\frac{d}{dt} + A_p^*$ from $Y^*_{2,T}$ to $\mathrm{L}^{2}(J_T; [\mathrm{dom}(A_p)]^*)$
%\cite[cf.][Theorem 4.15]{meyeroptimal} 
to see that 
\begin{equation*}
\begin{split}
0&= -\dot{p}_{\varepsilon_k} + A^{*}_p p_{\varepsilon_k} - \left[\frac{\partial}{\partial y}f_{\varepsilon_k}(\overline{y}_{\varepsilon_k},\overline{z}_{\varepsilon_k})\right]^{*}p_{\varepsilon_k} + SA_p q_{\varepsilon_k} - S\frac{\partial}{\partial y}f_{\varepsilon}(\overline{y}_{\varepsilon_k},\overline{z}_{\varepsilon_k}) q_{\varepsilon_k} - (\overline{y}_{\varepsilon_k} - y_d)\\
& \rightharpoonup -\dot{p} + A^{*}_p - \lambda_1 - \lambda_2 +SA_p q - (\overline{y} - y_d)
\end{split}
\end{equation*}
in $\mathrm{L}^{2}(J_T; [\mathrm{dom}(A_p)]^*)$ with $k\rightarrow\infty$.
Consequently, $p\in Y^*_{2,T}$ solves equation \eqref{eq:evolution_p}.
%\begin{align*}
%-\dot{p} + A^{*}_p = \lambda_1 + \lambda_2 -SA_pq + \overline{y} - y_d \text{for } t\in J_T,\ p(T)=0.
%\end{align*}
Note that we can set $f_\varepsilon \equiv f$ if $f$ is continuously differentiable from $X^\alpha\times \mathbb{R}$ into $X$ and in this case
$
\lambda_1= \left[\frac{\partial}{\partial y}f(\overline{y},\overline{z})\right]^{*}p \text{ and } \lambda_2= S\frac{\partial}{\partial y}f(\overline{y},\overline{z})q.
$
Moreover,
\begin{equation*}
0= B_i^*(p_{\varepsilon_k} +Sq_{\varepsilon_k}) +(\kappa+1)u_{\varepsilon_k} - \overline{u} \rightharpoonup B_i^*(p + Sq) + \kappa \overline{u}
\end{equation*}
in $U_i$ with $k\rightarrow\infty$ since $B_i^*$ is weakly continuous .
%\cite[cf.][Theorem 4.15]{meyeroptimal}.
This implies \eqref{opt_control_adj_eq_control}.
%\begin{align*}
%B_i^*(p+Sq) = -\kappa \overline{u} \text{ in } U_i.
%\end{align*}
\end{proof}

To gather information about $q$ from Lemma~\ref{Lem:limits_q_and_p_and_props_of_p} we continue similar as in \cite[Section 4]{brokate2013optimal}.
\begin{definition}[Partition of $J_T$]\label{Def:partition_[0,T]}
	Let $\overline{z}$ be as in Theorem~\ref{Thm:conv_minimizers_reg_to_min_original}.
	We split $\overline{J_T}$ into
	$
	I_0:=\{t\in \overline{J_T}: \overline{z}(t)\in (a,b)\}
	$
	and
	$
	I_\partial := \overline{J_T}\backslash I_0 = \{t \in \overline{J_T}: \overline{z}(t) \in\{a,b\}\}.
	$ 
	We further introduce
	$
	I_\partial^a := \{t \in \overline{J_T}: \overline{z}(t)=a \}
	$
	and
	$
	I_\partial^b := \{t \in \overline{J_T}: \overline{z}(t)=b\}.
	$
\end{definition}
Note that $I_0$ is open because $\overline{z}$ is continuous.

\begin{lemma}[$q$ in $I_0$]\label{Lem:q_in_H1}
Adopt the assumptions and the notation of Lemma~\ref{Lem:limits_q_and_p_and_props_of_p} and consider the subdivision of $\overline{J_T}$ from Definition~\ref{Def:partition_[0,T]}.
For any interval $(c,d)\subset I_0$ the limit $q$ in Lemma~\ref{Lem:limits_q_and_p_and_props_of_p} belongs to $\mathrm{H}^{1}(c,d)$ and there exist $\nu_1,\nu_2\in \mathrm{L}^2(J_T)$ such that
$
- \dot{q}= \nu_1 + \nu_2
$
in $\mathrm{L}^2(c,d)$.
If $f$ is continuously differentiable from $X^\alpha\times \mathbb{R}$ into $X$ then
$
\nu_1= \langle p, \frac{\partial}{\partial z}f(\overline{y},\overline{z})\rangle_{X}
$
and
$
\nu_2= \langle Sq, \frac{\partial}{\partial z}f(\overline{y},\overline{z})\rangle_{X}.
$
%\cite[cf.][Lemma 4.1]{brokate2013optimal}.
\end{lemma}

\begin{proof}
%	\begin{align*}
%	\dot{q}_\varepsilon &= - \langle p_\varepsilon, \frac{\partial}{\partial z}f_\varepsilon(\overline{y}_\varepsilon,\overline{z}_\varepsilon)\rangle_{X} - S\frac{\partial}{\partial z}f_{\varepsilon}(\overline{y}_\varepsilon,\overline{z}_\varepsilon)q_\varepsilon  +\frac{1}{\varepsilon}\Psi''(\overline{z}_\varepsilon)q_\varepsilon,\	q_\varepsilon(T)=0.
%	\end{align*}
	By Theorem~\ref{Thm:conv_minimizers_reg_to_min_original}$, \overline{z}_\varepsilon \rightarrow \overline{z}$ uniformly in $\overline{J_T}$. Let $(c,d)\subset I_0$ and $[s,t]\subset (c,d)$ be arbitrary. $(A4)_\varepsilon$ in Assumption~\ref{Ass:regularized_problem} implies that (w.l.o.g for $\varepsilon_0>0$ from Lemma~\ref{Lem:estim_adjoints_regularized_problem}) $\Psi''(\overline{z}_\varepsilon)\equiv 0$ on $[s,t]$ for all $\varepsilon\in (0,\varepsilon_0)$. 
	For $\varepsilon\in (0,\varepsilon_0)$ we integrate from $s$ to $t$ in \eqref{equ_q_epsilon} in Theorem~\ref{Thm:opt_syst_reg} and obtain
	\begin{equation*}
	q_\varepsilon(t) - q_\varepsilon(s) = \int_s^t - \langle p_\varepsilon(s) + Sq_\varepsilon(s), \frac{\partial}{\partial z}f_\varepsilon(\overline{y}_\varepsilon(s),\overline{z}_\varepsilon(s))\rangle_{X} ds.
	\end{equation*}
	Consider $\{\varepsilon_k\}$ from Lemma~\ref{Lem:limits_q_and_p_and_props_of_p}.
	Lemma~\ref{Lem:estim_adjoints_regularized_problem} together with Lemma~\ref{Lem:derivative_reg_function} implies uniform boundedness of $\langle p_\varepsilon, \frac{\partial}{\partial z}f_\varepsilon(\overline{y}_\varepsilon,\overline{z}_\varepsilon)\rangle_{X}$ and $\langle Sq_\varepsilon, \frac{\partial}{\partial z}f_\varepsilon(\overline{y}_\varepsilon,\overline{z}_\varepsilon)\rangle_{X}$ in $\mathrm{L}^2(J_T)$ if $\varepsilon\in(0,\varepsilon_0)$. Hence, we obtain a subsequence of $\{\varepsilon_k\}$ (still denoted by $\{\varepsilon_k\}$) and functions $\nu_1,\nu_2\in\mathrm{L}^2(J_T)$, such that
	$
	\langle p_{\varepsilon_k}, \frac{\partial}{\partial z}f_{\varepsilon_k}(\overline{y}_{\varepsilon_k},\overline{z}_{\varepsilon_k})\rangle_{X} \rightharpoonup \nu_1 
	$
	and
	$
	\langle Sq_{\varepsilon_k}, \frac{\partial}{\partial z}f_{\varepsilon_k}(\overline{y}_{\varepsilon_k},\overline{z}_{\varepsilon_k})\rangle_{X} \rightharpoonup \nu_2
	$
	in $\mathrm{L}^2(J_T)$ with $k\rightarrow\infty$. 
%	W.l.o.g. we can take the same subsequence as in Lemma~\ref{Lem:limits_q_and_p_and_props_of_p}.
	If $f$ is continuously differentiable from $X^\alpha\times \mathbb{R}$ into $X$ we can set $f_\varepsilon \equiv f$ and get
	$
	\nu_1= \langle p, \frac{\partial}{\partial z}f(\overline{y},\overline{z})\rangle_{X}
	$
	and
	$
	\nu_2= \langle Sq, \frac{\partial}{\partial z}f(\overline{y},\overline{z})\rangle_{X}.
	$
	In the general case we obtain
	\begin{equation*}
	q_{\varepsilon_k}(t) - q_{\varepsilon_k}(s) = \int_0^T -\langle p_{\varepsilon_k} + Sq_{\varepsilon_k}, \frac{\partial}{\partial z}f_{\varepsilon_k}(\overline{y}_{\varepsilon_k},\overline{z}_{\varepsilon_k})\rangle_{X} \chi_{[s,t]} ds \rightarrow \int_s^t -\nu_1-\nu_2 ds
	\end{equation*}
	with $k\rightarrow\infty$.
	So the weak derivative of $q$ exists in $\mathrm{L}^2(c,d)$ and is given by $-\nu_1-\nu_2$.
\end{proof}

Our next goal is to understand the behaviour of $q$ in $I_\partial$.

\begin{lemma}[$q$ in $I_\partial$: Relation to $\mathcal{P}(S\overline{y})$]\label{Lem:play_times_q_equ_zero}
	Adopt the assumptions and the notation of Lemma~\ref{Lem:limits_q_and_p_and_props_of_p} and consider the subdivision of $\overline{J_T}$ from Definition~\ref{Def:partition_[0,T]}. With $\mathcal{P}=\mathrm{Id} - \mathcal{W}$, cf. Lemma~\ref{Lem:hyst_props}, there holds
	$
	\left[\frac{d}{dt}\mathcal{P}[S\overline{y}](t)\right]q(t)=0 \text{ for a.e. } t\in I_\partial.
	$
	% \cite[cf.][Lemma 4.3]{brokate2013optimal} for a similar statement.
\end{lemma}

\begin{proof}
	%Similar to \cite[Lemma 4.2.]{brokate2013optimal} 
%We use the convergence results in Theorem~\ref{Thm:conv_minimizers_reg_to_min_original}, pass to the limit in \eqref{state_equ_regular_z} and apply Lemma~\ref{Lem:hyst_props} to obtain that
%\begin{align}
%	\frac{1}{\varepsilon}\Psi'(\overline{z}_\varepsilon) &= S\dot{\overline{y}}_\varepsilon - \dot{\overline{z}}_\varepsilon \rightharpoonup S\dot{\overline{y}} - \dot{\overline{z}} =
%	\frac{d}{dt}(S\overline{y} - \mathcal{W}[S\overline{y}])
%	= \frac{d}{dt}\mathcal{P}[S\overline{y}]\label{conv:dot(y_epsilon)}
%\end{align}
%in $\mathrm{L}^2(J_T)$ with $\varepsilon\rightarrow 0$.
Consider the concrete choice for $\Psi$ from Remark~\ref{Rem:construction_Psi} and $c$ and $\varepsilon_0$ from Lemma~\ref{Lem:estim_adjoints_regularized_problem}.
By Theorem~\ref{Thm:conv_minimizers_reg_to_min_original},  $\overline{z}_\varepsilon\rightarrow z$ uniformly so that 
$
\overline{z}_\varepsilon(t) \rightarrow b
\text{ for } t\in I_\partial^b \text{ and }
\overline{z}_\varepsilon(t) \rightarrow a \text{ for } t\in I_\partial^a
$
with $\varepsilon\rightarrow 0$.
Hence,
there exists some $\varepsilon_1\in (0,\varepsilon_0]$ such that
\begin{equation}
\begin{aligned}
a < \overline{z}_\varepsilon(t) < b+1&& \text{for }t\in I_\partial^b \label{bound_of_z_on_I_b}
\end{aligned}
\end{equation}
for all $\varepsilon\in (0,\varepsilon_1)$.
Remember that $\Psi_1(x)=(x-b)^3(4+b-x)$ and $\Psi\equiv 0$ on $[a,b]$. For $\varepsilon\in (0,\varepsilon_1)$ and $t\in I_\partial^b$ we obtain 
%$\Psi(\overline{z}_\varepsilon(t))=\Psi_1(\overline{z}_\varepsilon(t))\chi_{\{ b < \overline{z}_\varepsilon\leq b+2\}}(t) = (\overline{z}_\varepsilon(t)-b)^3(4+b-\overline{z}_\varepsilon(t))\chi_{\{ b < \overline{z}_\varepsilon\leq b+2\}}(t)$ and similarly
\begin{align}
\Psi'(\overline{z}_\varepsilon(t)) &=  \Psi_1'(\overline{z}_\varepsilon(t)) \chi_{\{ b < \overline{z}_\varepsilon\leq b+2\}}(t)
%= 4[3(\overline{z}_\varepsilon(t)-b)^2-(\overline{z}_\varepsilon(t)-b)^3] \chi_{\{ b < \overline{z}_\varepsilon\leq b+2\}}\notag\\
= 4(3-(\overline{z}_\varepsilon(t)-b))(\overline{z}_\varepsilon(t)-b)^2 \chi_{\{ b < \overline{z}_\varepsilon\leq b+2\}}(t) ,\label{form_of_Psi'} \\
\Psi''(\overline{z}_\varepsilon(t)) 
&= 12(\overline{z}_\varepsilon(t)-b)[2 - (\overline{z}_\varepsilon(t)-b)] \chi_{\{ b < \overline{z}_\varepsilon\leq b+2\}}(t).\label{form_of_Psi''}
\end{align}
%and 
%\[
%8(\overline{z}_\varepsilon(t)-b)^2 < \Psi_1'(\overline{z}_\varepsilon(t)) < 16(\overline{z}_\varepsilon(t)-b)^2.
%\]
We apply estimate \eqref{estim_q_epsilon} from Lemma~\ref{Lem:estim_adjoints_regularized_problem} together with \eqref{bound_of_z_on_I_b} and \eqref{form_of_Psi''} to see that
\begin{equation}
	\begin{split}
	c&\geq \frac{1}{\varepsilon}\int_{0}^{T} \Psi''(\overline{z}_\varepsilon(s)) |q_\varepsilon(s)| ds
	\geq \frac{1}{\varepsilon}\int_{I_\partial^b} \Psi''(\overline{z}_\varepsilon(s)) |q_\varepsilon(s)| ds\\
	&=\frac{1}{\varepsilon}\int_{I_\partial^b}  12(\overline{z}_\varepsilon(t)-b)[2 - (\overline{z}_\varepsilon(t)-b)] \chi_{\{ b < \overline{z}_\varepsilon\leq b+2\}} |q_\varepsilon(s)| ds\\
	&\geq \frac{1}{\varepsilon}\int_{I_\partial^b}  12(\overline{z}_\varepsilon(t)-b) \chi_{\{ b < \overline{z}_\varepsilon\leq b+2\}} |q_\varepsilon(s)| ds
	\end{split}\label{est_proof_Lem:play_times_q_equ_zero}
\end{equation}
for all $\varepsilon\in(0,\varepsilon_1)$.
We apply the convergence results from Theorem~\ref{Thm:conv_minimizers_reg_to_min_original} in \eqref{state_equ_regular_z} and use the representation $\mathcal{W} + \mathcal{P} = \mathrm{Id}$ from Lemma~\ref{Lem:hyst_props} to obtain the weak convergence
\begin{equation*}
\frac{1}{\varepsilon}\Psi'(\overline{z}_\varepsilon) = S\dot{\overline{y}}_\varepsilon - \dot{\overline{z}}_\varepsilon \rightharpoonup S\dot{\overline{y}} - \dot{\overline{z}} =
\frac{d}{dt}(S\overline{y} - \mathcal{W}[S\overline{y}])
= \frac{d}{dt}\mathcal{P}[S\overline{y}]
\end{equation*}
in $\mathrm{L}^2(J_T)$ with $\varepsilon\rightarrow 0$.
Furthermore, by Lemma~\ref{Lem:limits_q_and_p_and_props_of_p}, $|q_{\varepsilon_k}|\rightarrow |q|$ strongly in $\mathrm{L}^2(J_T)$ with $k\rightarrow \infty$ and
$\frac{d}{dt}\mathcal{P}[S\overline{y}]=\left| \frac{d}{dt}\mathcal{P}[S\overline{y}] \right|$ a.e. in $I_\partial^b$ by definition of $I_\partial^b$.
%, again estimate \eqref{estim_q_epsilon} together with 
This together with 
%implies that  \eqref{conv:dot(y_epsilon)}, 
\eqref{form_of_Psi'} and \eqref{est_proof_Lem:play_times_q_equ_zero} yields
\begin{equation*}
\begin{split}
0&\leq \int_{I_\partial^b} \left| \frac{d}{dt}\mathcal{P}[S\overline{y}] \right| |q(s)| ds
= \lim\limits_{k\rightarrow\infty} \frac{1}{\varepsilon_k}\int_{I_\partial^b} \Psi'(\overline{z}_{\varepsilon_k}(s)) |q_{\varepsilon_k}(s)| ds\\
&=\lim\limits_{k\rightarrow\infty} \frac{1}{\varepsilon_k}\int_{I_\partial^b}  4(3-(\overline{z}_\varepsilon(t)-b))(\overline{z}_\varepsilon(t)-b)^2 \chi_{\{ b < \overline{z}_{\varepsilon_k}\leq b+2\}} |q_{\varepsilon_k}(s)| ds\\
&\leq \lim\limits_{k\rightarrow\infty}\frac{12}{\varepsilon_k}\int_{I_\partial^b}  (\overline{z}_{\varepsilon_k}(s)-b)^2 \chi_{\{ b < \overline{z}_{\varepsilon_k}\leq b+2\}} |q_{\varepsilon_k}(s)| ds
\leq c \lim\limits_{k\rightarrow\infty} \sup_{s\in I_\partial^b}(\overline{z}_{\varepsilon_k}(s)-b) = 0.
\end{split}
\end{equation*}

Similar estimates for $I_\partial^a$ and the fact that $I_\partial=I_\partial^a\cup I_\partial^b$ prove the statement.
\end{proof}

%As in \cite[Lemma 4.6]{brokate2013optimal} 
Next, we pass to the limit in \eqref{equ_q_epsilon} to get the following result:
\begin{lemma}[$q$ in $I_\partial$: Relation to $d\mu$]\label{Lem:overall_evol_q}
	Adopt the assumptions and the notation of Lemma~\ref{Lem:limits_q_and_p_and_props_of_p} and let $\nu_1$ and $\nu_2$ be as in Lemma~\ref{Lem:q_in_H1}. Consider the subdivision of $\overline{J_T}$ from Definition~\ref{Def:partition_[0,T]}.
	We denote $
	d\mu_\varepsilon := \frac{1}{\varepsilon} \Psi''(\overline{z}_\varepsilon) q_\varepsilon
	$.
	There exists a measure $d\mu \in \mathrm{C}(\overline{J_T})^*$, such that a subsequence $\{d\mu_{\varepsilon_k}\}$ (w.l.o.g we may consider $\{\varepsilon_k\}$ from Lemma~\ref{Lem:limits_q_and_p_and_props_of_p}) converges weak-* to $d\mu$ in $\mathrm{C}(\overline{J_T})^*$ with $k\rightarrow\infty$.
	The support of $d\mu$ is contained in $I_\partial$.
	For any $\varphi\in \mathrm{C}(\overline{J_T})$ there holds
	\begin{equation*}
	\int_{0}^{T} -\varphi(t) dq(t) + \int_{I_\partial} \varphi(t) d\mu(t) = \int_{0}^{T} \varphi(t)(\nu_1(t)+\nu_2(t)) dt.
	\end{equation*}
%	Furthermore,
%	\begin{align*}
%	\int_{I_\partial} -\varphi(t) dq(t) + \int_{I_\partial} \varphi(t) d\mu(t) = \int_{I_\partial} \varphi(t)g(t) dt
%	\end{align*}
%	for all $\varphi$ with support in $I_\partial$
%	which implies 
	This implies
	$d\mu = dq + (\nu_1+\nu_2)dt$ as measures on $I_\partial$.
	\end{lemma}
\begin{proof}
	By \eqref{estim_q_epsilon} in Lemma~\ref{Lem:estim_adjoints_regularized_problem} the functions
	$
	d\mu_\varepsilon
	$
	are bounded in $\mathrm{L}^1(J_T)$ independently of $\varepsilon$ for all $\varepsilon\in(0,\varepsilon_0)$.
	Consequently, a subsequence of $\{d\mu_\varepsilon\}$ converges weak-* in $\mathrm{C}(\overline{J_T})^*$ to some measure $d\mu$. 
	By $(A4)_\varepsilon$ in Assumption~\ref{Ass:regularized_problem} and the uniform convergence of $\overline{z}_\varepsilon$ to $\overline{z}$ there holds $\varphi\, \frac{1}{\varepsilon} \Psi''(\overline{z}_\varepsilon) q_\varepsilon \equiv 0$ as soon as $\varepsilon$ is small enough, if $\varphi\in\mathrm{C}(\overline{J_T})$ has compact support in $I_0$. Therefore, the support of $d\mu$ is contained in $I_\partial$ \cite[p.343]{brokate2013optimal}.
	The other statements are shown similar as \cite[Lemma 4.6]{brokate2013optimal} and \cite[Lemma 4.7]{brokate2013optimal}.
\end{proof}

%Just as in \cite[Lemma 4.4]{brokate2013optimal} 
It also follows:
\begin{lemma}[Discontinuity properties of $q$]\label{Lem:q_jumps_down_in_rev_time}
	Adopt the assumptions and notation of Lemma~\ref{Lem:limits_q_and_p_and_props_of_p}.
	The absolute value of $q$ can only jump downwards in reverse time.
	Consequently, for any $t\in \overline{J_T}$ there holds $|q(t-)|\leq |q(t+)|$ and $q(T-)=q(T)=0$.
	Moreover, $q$ is right continuous in $[0,T)$ and left continuous at $T$.
\end{lemma}
\begin{proof}
	From Lemma~\ref{Lem:limits_q_and_p_and_props_of_p} we conclude that
%    in Lemma~\ref{Lem:limits_q_and_p_and_props_of_p}, for which we maintain the same notation,
	$q_{\varepsilon_k}$ converges to $q$ in $\mathrm{L}^1(J_T)$ and that $dq_{\varepsilon_k} = \dot{q}_{\varepsilon_k}dt$ converges to $dq$ weak-* in $\mathrm{C}(\overline{J_T})^*$. 
	From \cite[][ Chapter XII.7]{visintin2013differential} it follows that $q$ has bounded variation and that the limit is right continuous in $[0,T)$ and left continuous at $T$.
	The rest of the statements are shown just as \cite[Lemma 4.4]{brokate2013optimal}.
\end{proof}

The unknown measure $d\mu$ has support in $I_\partial$ so that we only know the behaviour of the sum $-dq + d\mu$ in $\mathrm{C}(\overline{J_T})^*$ but not that of $dq$ alone.
In order to analyze $q$ also in $I_\partial$ we make the following regularity assumption, cf. \cite[p.344]{brokate2013optimal}:
\begin{assumption}[Regularity assumption]\label{Ass:regularity_assumption}
	Let $\overline{y}$ be as in Theorem~\ref{Thm:conv_minimizers_reg_to_min_original} and consider the subdivision of $\overline{J_T}$ from Definition~\ref{Def:partition_[0,T]}. We suppose that the function $\mathcal{P}[S\overline{y}]$ satisfies 
	$
	\frac{d}{dt}\mathcal{P}[S\overline{y}]\neq 0\ \text{ a.e. in }I_\partial.
	$
	Equivalently,
	$
	S\dot{\overline{y}}> 0\ \text{ a.e. in }I_\partial^b$ and
	$S\dot{\overline{y}}< 0\ \text{ a.e. in }I_\partial^a.
	$
\end{assumption}
\begin{remark}
	This assumption is reasonable if $S\overline{y}$ is the size of interest.
	Consider for example the case when $w$ in (A3) in Assumption~\ref{Ass:general_ass_and_short_notation_1} has the form $w= \frac{1}{m|\Omega|}\varphi$ for some $\varphi \in \prod_{j=1}^{m} \mathrm{C}_{\Gamma_{D_j}}^{\infty}(\Omega)$, where the components $\varphi_j$, $j\in\{1,\ldots, m\}$, are constantly equal to $1$ within most of $\Omega$ and vanish only in a small neighbourhood of $\Gamma_{D_j}$.
	If we identify $\mathrm{ran}\left(I_p\right)$ with $\mathbb{W}_{\Gamma_D}^{1,p}(\Omega)$, then $S$ acts on $y\in\mathrm{dom}(A_p)$ as
	$
	Sy = \frac{1}{m|\Omega|}\sum_{j=1}^{m}\int_\Omega y_j \varphi_j dx.
	$
	This means that $Sy$ is approximately the mean value of $y$ in $\Omega$.
	If this is the value of interest then nothing changes in the system if
	$S\dot{\overline{y}}= 0$ in a subset of $I_\partial$ with positive measure.
\end{remark}

%Let $(c,d)\subset I_\partial$ be arbitrary and suppose that Assumption~\ref{Ass:regularity_assumption} holds. 
%Then Lemma~\ref{Lem:play_times_q_equ_zero} implies $q(t)=0$ for a.e. $t\in(c,d)$.
%Because $q$ is right continuous in $[0,T)$ it follows $q\equiv 0$ in $[c,d)$.
%Consequently, for every subinterval $[\beta,\gamma]\subset (c,d)$ we have
%$
%	0= q(\gamma -)-q(\beta +) = dq((\beta,\gamma) )
%$
%so that $\dot{q}=dq=0$ as a measure on $(c,d)$.
%The absolute value of $q$ can only jump downwards in reverse time. This implies also that whenever an interval $(c,d)\subset I_0$ is followed by an interval $[e,f]\subset I_\partial$, then $q$ is absolutely continuous on $[c,f)$.
%To make this observation precise we define as in \cite{brokate2013optimal}:

In order to analyze the behaviour of $q$ and $dq$ in $\overline{I_0}\cap I_\partial$ we introduce the following categories of times as in \cite{brokate2013optimal}:
\begin{definition}[Switching times]\label{Def:switchin_times}
		Consider the subdivision of $\overline{J_T}$ from Definition~\ref{Def:partition_[0,T]}. 
		We call a time $t$ a $(0,\partial)$-switching time if $t\in \overline{I_0}\cap I_\partial$ and if there is some $\varepsilon > 0 $ such that $(t-\varepsilon,t)\subset I_0$ and $[t,t+\varepsilon)\subset I_\partial$.
		We say that $t$ is a $(\partial,0)$-switching time if $t\in \overline{I_0}\cap I_\partial$ and if for some $\varepsilon > 0 $ we have $(t-\varepsilon,t]\subset I_\partial$ and $(t,t+\varepsilon)\subset I_0$.
\end{definition}

\begin{lemma}[$q$ at switching times]\label{Lem:delta_0_switching}
	Adopt the assumptions and the notation of Lemma~\ref{Lem:limits_q_and_p_and_props_of_p}.
	If $t$ is a $(0,\partial)$-switching time in the sense of Definition~\ref{Def:switchin_times} and if Assumption~\ref{Ass:regularity_assumption} holds then there exits some $\varepsilon>0$ such that
	$
	q\equiv 0
	$
	on $[t,t+\varepsilon)$. Moreover, $q$ is continuous at $t$ with $t=0$.
	Furthermore, for every open interval $(c,d)\subset I_\partial$ there holds that $q\equiv 0$ in $[c,d)$.
\end{lemma}
\begin{proof}
	Let $(c,d)\subset I_\partial$ be arbitrary and suppose that Assumption~\ref{Ass:regularity_assumption} holds. 
	Then Lemma~\ref{Lem:play_times_q_equ_zero} implies $q(t)=0$ for a.e. $t\in(c,d)$.
	By Lemma~\ref{Lem:q_jumps_down_in_rev_time}, $q$ is right continuous in $[0,T)$ so that $q\equiv 0$ in $[c,d)$.
	Consequently, for every subinterval $[\beta,\gamma]\subset (c,d)$ we have
	$
	0= q(\gamma -)-q(\beta +) = dq((\beta,\gamma) )
	$
	so that $dq=0$ as a measure on $(c,d)$.
	Again by Lemma~\ref{Lem:q_jumps_down_in_rev_time} the absolute value of $q$ can only jump downwards in reverse time. By Lemma~\ref{Lem:q_in_H1}, $q\in\mathrm{H}^1(e,c)$ for any interval $(e,c)\subset I_0$. Consequently, whenever an interval $(e,c)\subset I_0$ is followed by an interval $[c,d]\subset I_\partial$, then $q$ is absolutely continuous on $[e,d)$.
		
	Now let $t$ be a $(0,\partial)$-switching and consider $\varepsilon>0$ such that
	$(t-\varepsilon,t)\subset I_0$ and $[t,t+\varepsilon)\subset I_\partial$.
	Then setting $e=t-\varepsilon$, $c=t$ and $d=t+\varepsilon$ proves the rest of the lemma.
	
\end{proof}

\begin{remark}\label{Rem:delta_0_switching}
	In the setting of Lemma~\ref{Lem:delta_0_switching} one can prove even more about the continuity properties of $q$ if $f$ is continuously differentiablem, even in absence of Assumption~\ref{Ass:regularity_assumption}:
	\begin{itemize}
	\item	Note first that when $t\in I_\partial$ is a $(\partial,0)$-switching time then $q$ might jump at $t$ no matter if Assumption~\ref{Ass:regularity_assumption} holds or not. If it does not jump then under Assumption~\ref{Ass:regularity_assumption} then necessarily $q(t)=0$.
	It is also possible to prove that $q$ may only jump up at $t$ if $\int_{t}^{t^-} \langle p+Sq, \frac{\partial}{\partial z}f(\overline{y},\overline{z}) \rangle_{X}\,ds > 0$, where either $t^-=t^-(t)\in (t,T]\cap I_\partial^a$ is (essentially) the first time in $(t,T)$ for which there exists some $\varepsilon>0$ such that $S\dot{\overline{y}}<0$ a.e. in $(t^-,t^-+\varepsilon)$, or $t^-=T$.    
	It can further be shown that the height of the jump is bounded by $\int_{t}^{t^-} \langle p+Sq, \frac{\partial}{\partial z}f(\overline{y},\overline{z}) \rangle_{X}\,ds$.
	Analogously, one can prove that $q$ may only jump down at $t$ if $\int_{t}^{t^+} \langle p+Sq, \frac{\partial}{\partial z}f(\overline{y},\overline{z}) \rangle_{X}\,ds < 0$, where either $t^+=t^+(t)\in (t,T]\cap I_\partial^b$ is (essentially) the first time in $(t,T)$ for which there exists some $\varepsilon>0$ such that $S\dot{\overline{y}}>0$ a.e. in $(t^+,t^+\varepsilon)$, or $t^+=T$. In this case the height of the jump is bounded by $-\int_{t}^{t^+} \langle p+Sq, \frac{\partial}{\partial z}f(\overline{y},\overline{z}) \rangle_{X}\,ds$.
	
	\item	Other categories of times can be considered. Those include isolated times in $I_0$ or subintervals of $I_\partial$ in which $S\dot{\overline{y}}=0$ a.e.
	The latter can only occur if Assumption~\ref{Ass:regularity_assumption} does not hold true.
	Also for those categories one can show sign conditions for $dq$ and $d\mu$ and upper bounds for jumps.
	\end{itemize}
	
	The proof of these continuity properties is very technical and exceeds the scope of this work.
	The results will be published in the dissertation of the project in which this paper originated.
	
\end{remark}

%If $t\in I_\partial$ is an isolated point then since $(c,t),(t,d)\subset I_0$ then
%$q\in\mathrm{H}^1(c,t),\mathrm{H}^1(t,d)$ by Lemma~\eqref{Lem:q_in_H1} and $q$ is right continuous on $[t,d)$. 

%The only possible jump points of $q$ are $(\partial,0)$-switching times and isolated points.

\subsection{Optimality conditions for distributed or boundary controls}\label{Subsec:opt_cond_distributed_or_boundary}
We derive optimality conditions for problem \eqref{opt_control_ control_problem}-\eqref{state_equ_z} for the optimal control $\overline{u}$ from Theorem~\ref{Thm:conv_minimizers_reg_to_min_original} in terms of the pair $p$ and $q$ from Lemma~\ref{Lem:limits_q_and_p_and_props_of_p}. 
We can not expect a pointwise condition as in \cite[Section 5]{meyeroptimal} since the hysteresis and its derivative, and then also $F'[\overline{y},\cdot]$ in Theorem~\ref{Thm:state_equ_sol_op}, act non-local in time. This implies that if for some direction $\zeta\in \mathrm{C}(\overline{J_T};X^\alpha)$ and some set $I\subset J_T$ of positive measure the derivative $F'[y;\zeta](\tau)=f'[(y(\tau),\mathcal{W}[Sy](\tau));(y(\tau),\mathcal{W}'[Sy;S\zeta](\tau))]$ is not zero for $\tau\in I$, then the values of the derivative in $I$ might have an influence on its value at any $t$ with $\max\{\tau\in I\}<t\leq T$. That is, we can only expect an optimality condition for problem \eqref{opt_control_ control_problem}-\eqref{state_equ_z} which includes integration at least over a part of the time interval $J_T$.
Nevertheless, we follow the steps in \cite[Section 5]{meyeroptimal} as long as possible.
The optimality condition for $i\in\{1,2\}$ is derived in Lemma~\ref{Lem:opt_cond_general} and improved in Corollary~\ref{Cor:opt_cond_general} for the case when $f$ is continuously differentiable.
%We first treat the general case with controls in $U_i$ and dedicate a whole section to improve our results for problem \eqref{opt_control_ control_problem}-\eqref{state_equ_z} with $i=1$.
%Unfortunately the range of $B_2$ is not dense in $X$.
We can even further improve this condition for the case when the controls act inside of $\Omega$, i.e. for $i=1$. Also in this case we can not expect to obtain an inequality without integration in time. But since the range of $B_1$ is dense in $X$, we are able to derive a condition without variation in space. The results can be found in Corollary~\ref{Cor:opt_control_distrib_opt_cond} in Subsection~\ref{Subsubsec:Improved optimality conditions}. For $i=1$ we are also able to prove uniqueness of $p,q$ and $d\mu$ if $f$ is continuously differentiable, see Corollary~\ref{Cor:Uniqueness_dist_contr} in Subsection~\ref{Subsubsec:uniqueness}.

Because the range of $B_2$ is not dense in $X$, we treat the general case $i\in\{1,2\}$ first.

\begin{lemma}[Optimality condition]\label{Lem:opt_cond_general}
	Adopt the assumptions and the notation of Lemma~\ref{Lem:limits_q_and_p_and_props_of_p} and let $\nu_1$ and $\nu_2$ be as in Lemma~\ref{Lem:q_in_H1}.
	For any $h\in U_i$, $y^{B_i\overline{u},B_ih}=G'[B_i\overline{u};B_ih]$ and
	
	$F'[\overline{y};y^{B_i\overline{u},B_ih}](t)=f'[(\overline{y}(t),\mathcal{W}[S\overline{y}](t));(\overline{y}(t),\mathcal{W}'[S\overline{y};Sy^{B_i\overline{u},B_ih}](t))]$ (see Theorem~\ref{Thm:state_equ_sol_op}), there holds the optimality condition
	\begin{equation}
	\begin{split}
	&\int_{0}^{T} \langle \lambda_1 + \lambda_2 + S(\nu_1 + \nu_2), y^{B_i\overline{u},B_ih} \rangle_{\mathrm{dom}(A_p)}  dt\\
	& \leq \int_{I_\partial} Sy^{B_i\overline{u},B_ih} d\mu + \int_{0}^{T} \langle p+Sq , F'[\overline{y};y^{B_i\overline{u},B_ih}] \rangle_{X}  dt.
	\end{split}\label{opt_control_opt_cond_general}
	\end{equation}
\end{lemma}
\begin{proof}
	Since $\overline{u}$ is an optimal control, the directional derivative of the reduced cost functional $\mathcal{J}$ has to be greater or equal than zero in each direction.
	With $y^{B_i\overline{u},B_ih}=G'[B_i\overline{u};B_ih]$ this means that for any $h\in U_i$ there holds
\begin{equation}
	0\leq \mathcal{J}'[\overline{u};h] = \langle \overline{y} - y_d , y^{B_i\overline{u},B_ih} \rangle_{\mathrm{L}^{2}(J_T;\mathrm{dom}(A_p))} 
	+ \kappa\langle \overline{u} , h \rangle_{U_i}.\label{proof:Lem:opt_cond_general_computation_1}
	\end{equation}

The function $y^{B_i\overline{u},B_ih}$ solves the evolution equation
%proceed just as in \cite[Theorem 5.3]{meyeroptimal} and 
\eqref{eq:Thm:state_equ_sol_op} in Theorem~\ref{Thm:state_equ_sol_op} with $y$ replaced by $\overline{y}$ and $h$ replaced by $B_ih$. We test this equation with $p+Sq$, integrate over time and apply \eqref{opt_control_adj_eq_control} to compute
\begin{equation}
	\begin{split}
	&\int_{0}^{T} \langle p+Sq , \dot{y}^{B_i\overline{u},B_ih} + A_p y^{B_i\overline{u},B_ih} \rangle_{X}  dt - \int_{0}^{T} \langle p+Sq , F'[\overline{y};y^{B_i\overline{u},B_ih}]\rangle_{X}  dt\\
	&= \int_{0}^{T} \langle p+Sq , B_i h \rangle_{X}  dt
	= -\kappa\langle\overline{u}, h \rangle_{U_i}.
	\end{split}\label{proof:Lem:opt_cond_general_computation_2}
\end{equation}

We integrate the first term on the left side of \eqref{proof:Lem:opt_cond_general_computation_2} by parts, insert \eqref{eq:evolution_p} from Lemma~\ref{Lem:limits_q_and_p_and_props_of_p} and use the representation of $dq$ from Lemma~\ref{Lem:overall_evol_q} to observe
\begin{equation}
\begin{split}
&\int_{0}^{T} \langle p+Sq , \dot{y}^{B_i\overline{u},B_ih} + A_p y^{B_i\overline{u},B_ih} \rangle_{X}  dt 
%- \int_{0}^{T} \langle p+S q , F'[\overline{y};y^{B_i\overline{u},B_ih}] \rangle_{X}  dt
\\
&= \int_{0}^{T} \langle \lambda_1 + \lambda_2 - SA_p q + \overline{y} - y_d, y^{B_i\overline{u},B_ih} \rangle_{\mathrm{dom}(A_p)}  dt\\
&-\int_{0}^{T} Sy^{B_i\overline{u},B_ih} dq + \int_{0}^{T} \langle SA_p q, y^{B_i\overline{u},B_ih} \rangle_{\mathrm{dom}(A_p)} dt
\\
&= \int_{0}^{T} \langle \lambda_1 + \lambda_2 + \overline{y} - y_d, y^{B_i\overline{u},B_ih} \rangle_{\mathrm{dom}(A_p)}  dt -\int_{0}^{T} Sy^{B_i\overline{u},B_ih} dq \\
%&- \int_{0}^{T} \langle p+S q , F'[\overline{y};y^{B_i\overline{u},B_ih}] \rangle_{X}  dt\\
&= \int_{0}^{T} \langle \lambda_1 + \lambda_2 + \overline{y} - y_d, y^{B_i\overline{u},B_ih} \rangle_{\mathrm{dom}(A_p)}  dt -\int_{I_\partial} Sy^{B_i\overline{u},B_ih} d\mu + \int_{0}^T (\nu_1+\nu_2) Sy^{B_i\overline{u},B_ih} dt.
\end{split}\label{proof:Lem:opt_cond_general_computation_3}
%- \int_{0}^{T} \langle p+Sq , F'[\overline{y};y^{B_i\overline{u},B_ih}] \rangle_{X} dt.
\end{equation}

We insert \eqref{proof:Lem:opt_cond_general_computation_2} into \eqref{proof:Lem:opt_cond_general_computation_1} and use \eqref{proof:Lem:opt_cond_general_computation_3} to obtain
\begin{equation*}
	\begin{split}
	0 & \leq \int_{0}^{T} \langle \overline{y} - y_d , y^{B_i\overline{u},B_ih} \rangle_{\mathrm{dom}(A_p)}  dt + \kappa\langle\overline{u}, h \rangle_{U_i}\\
	& =  - \int_{0}^{T} \langle \lambda_1 + \lambda_2, y^{B_i\overline{u},B_ih} \rangle_{\mathrm{dom}(A_p)}  dt + \int_{I_\partial} Sy^{B_i\overline{u},B_ih} d\mu - \int_{0}^T (\nu_1+\nu_2) Sy^{B_i\overline{u},B_ih} dt \\
	&+ \int_{0}^{T} \langle p+Sq , F'[\overline{y};y^{B_i\overline{u},B_ih}] \rangle_{X}  dt.
	\end{split}
\end{equation*}
%This proves \eqref{opt_control_opt_cond_general}.
%and then the optimality condition
%\begin{align*}
%&\int_{0}^{T} \langle \lambda_1 + \lambda_2 + S(\nu_1 + \nu_2), y^{B_i\overline{u},B_ih} \rangle_{\mathrm{dom}(A_p)}  dt\notag\\
%& \leq \int_{I_\partial} Sy^{B_i\overline{u},B_ih} d\mu + \int_{0}^{T} \langle p+Sq , F'[\overline{y};y^{B_i\overline{u},B_ih}] \rangle_{X}  dt.
%\end{align*}
\end{proof}

\subsection{Summary: Adjoint system and optimality conditions for distributed or boundary controls}\label{Subsec:Summary_Adjoint_and_opt_cond_distributed_or_boundary}

We summarize our results for the general control problem with $i\in\{1,2\}$.
% \cite[cf.][Theorem 5.2]{brokate2013optimal}:
\begin{theorem}[Adjoint system and optimality condition]\label{Thm:opt_control_boundary_opt_cond_limit_adjoint}
Let Assumption~\ref{Ass:general_ass_and_short_notation_1} and Assumption~\ref{Ass:regularized_problem} hold. For $i\in\{1,2\}$ suppose that $\overline{u}\in U_i$ is an optimal control for problem \eqref{opt_control_ control_problem}-\eqref{state_equ_z} together with the optimal state $\overline{y}\in Y_{2,0}$ and $\overline{z}=\mathcal{W}[S\overline{y}]\in \mathrm{H}^1(J_T)$. Consider the subdivision of $\overline{J_T}$ from Definition~\ref{Def:partition_[0,T]}. 
Then there exist adjoint states $p\in Y_{2,T}^*$ and $q\in \mathrm{BV}(J_T)$ of the following kind:
There holds
$
B_i^*(p+Sq) = -\kappa \overline{u} \text{ in } U_i.
$
%That is, up to redefining $\overline{u}$ on a set of measure zero, $\overline{u}$ is left continuous on $[0,T)$ and right continuous at $T$, and
%\[
%\overline{u} \in \mathrm{L}^\infty\left(J_T;\prod\limits_{i=1}^m \mathrm{L}^2(\Gamma_i)\right).
%\] 
%
%For $\Omega \subset \mathbb{R}^2$, there is some $\beta> 0$, such that 
%\[
%\overline{y}\in \mathrm{C}^\beta([t_0,T], \mathrm{L}^\infty(\Omega,\mathrm{R}^m)).
%\]
%
%If in addition to Assumption \ref{Ass:existence_extension}, there are also bi-Lipschitz mappings for all $x\in \mathrm{D}_i$, $i=1,...,m$, then 
%\[
%\overline{y}\in \mathrm{C}^\beta([t_0,T], [\mathrm{C}^\gamma(\Omega)]^m).
%\]
For some functions $\lambda_1,\lambda_2\in \mathrm{L}^{2}(J_T; [X^\alpha]^*)$ we have
\begin{equation*}
\begin{aligned}
	-\dot{p} + A^{*}_p p= \lambda_1 + \lambda_2 -SA_p q + \overline{y} - y_d&& \text{for } t\in J_T,\ p(T)=0.
\end{aligned}
\end{equation*}
$q$ is left continuous in $J_T$, right continuous at $T$ and absolutely continuous in $I_0$. There exist $\nu_1,\nu_2\in \mathrm{L}^2(J_T)$ such that $q$ solves
$
-\dot{q}= \nu_1 + \nu_2
$
in every open subinterval of $I_0$.
$
\frac{d}{dt}\mathcal{P}[S\overline{y}](t)q(t) = 0\text{ for a.e. }t\in I_\partial
$
and there is a measure $d\mu \in \mathrm{C}(\overline{J_T})^*$ with support in $I_\partial$ such that
%\begin{align*}
%		\int_{I_\partial} -\varphi dq + \int_{I_\partial} \varphi d\mu = \int_{I_\partial} \varphi(t)[\nu_1(t) + \nu_2(t)] dt
%\end{align*}
%for all $\varphi \in \mathrm{C}(\overline{J_T})$ with support in $I_\partial$.
$d\mu = dq + (\nu_1 + \nu_2)dt$ as measures on $I_\partial$.
For all $h\in U_i$ and with $y^{B_i\overline{u},B_ih}=G'[B_i\overline{u};B_ih]$ (see Theorem~\ref{Thm:state_equ_sol_op}) there holds the optimality condition
\begin{equation}
\begin{split}
&\int_{0}^{T} \langle \lambda_1 + \lambda_2 + S(\nu_1 + \nu_2), y^{B_i\overline{u},B_ih} \rangle_{\mathrm{dom}(A_p)}  dt\\
&\leq \int_{I_\partial} Sy^{B_i\overline{u},B_ih} d\mu + \int_{0}^{T} \langle p+Sq , F'[\overline{y};y^{B_i\overline{u},B_ih}] \rangle_{X}  dt,
\end{split}\label{opt_control_opt_cond_general_THM}
\end{equation}
where $F'[\overline{y};y^{B_i\overline{u},B_ih}](t)=f'[(\overline{y}(t),\mathcal{W}[S\overline{y}](t));(\overline{y}(t),\mathcal{W}'[S\overline{y};Sy^{B_i\overline{u},B_ih}](t))]$.
The absolute value of $q$ can only jump downwards in reverse time so that $q(T-)=q(T)=0$ and $|q(t-)|\leq |q(t+)|$ for all $t\in \overline{J_T}$.
%If $f$ is continuously differentiable from $X^\alpha\times \mathbb{R}$ into $X$ then 
%$\lambda_1= \left[\frac{\partial}{\partial y}f(\overline{y},\overline{z})\right]^{*}p$,
%$\lambda_2= S\frac{\partial}{\partial y}f(\overline{y},\overline{z})q,\ 
%\nu_1= \langle p, \frac{\partial}{\partial z}f(\overline{y},\overline{z})\rangle_{X}$,
%$\nu_2= \langle Sq, \frac{\partial}{\partial z}f(\overline{y},\overline{z})\rangle_{X}
%$
%and the optimality condition is given by
%\begin{align}
%&\int_{0}^{T} \langle p+Sq, \frac{\partial}{\partial z}f(\overline{y},\overline{z}) \rangle_{X}\mathcal{P}'[S\overline{y};Sy^{B_i\overline{u},B_ih}] dt \leq \int_{I_\partial} Sy^{B_i\overline{u},B_ih} d\mu\label{opt_control_opt_cond_regular_f}.
%\end{align}
If the regularity Assumption \ref{Ass:regularity_assumption} is valid then
$q$ is continuous at every $(0,\partial)$-switching time $t$ (see Definition~\ref{Def:switchin_times}) with $q(t)=0$.
In this case, for every open interval $(c,d)\subset I_\partial$ it follows $q\equiv 0$ on $[c,d)$.
\end{theorem}

%If $I_\partial^{(out)}=\emptyset$, this case may be prohibited. This holds true, if the following assumption is satisfied:
%\begin{assumption}
%Every point $t\in \overline{I_0}\cap I_\partial$ is either a $(0,\partial)$ switching time, or there exists some $\varepsilon > 0 $, such that 
%\begin{align*}
%\dot{\xi}\neq 0 \text{ a.e. in }(t,t+\varepsilon).
%\end{align*}
%This means that whenever $\overline{z}$ reaches $I_\partial$, then $S\overline{y}$ maintains its growing behaviour at least until time $t+\varepsilon$.
%
%In other words, there are at most $(0,\partial)$ and $(\partial,0)$ switching times.
%\end{assumption}
We can improve the results of Theorem~\ref{Thm:opt_control_boundary_opt_cond_limit_adjoint} if $f$ is continuously differentiable:
\begin{corollary}[Adjoint system and optimality condition for regular $f$]\label{Cor:opt_cond_general}
	Let Assumption~\ref{Ass:general_ass_and_short_notation_1} and Assumption~\ref{Ass:regularized_problem} hold. Moreover, suppose that $f$ is continuously differentiable from $X^\alpha\times \mathbb{R}$ into $X$. For $i\in\{1,2\}$ assume that $\overline{u}\in U_i$ is an optimal control for problem \eqref{opt_control_ control_problem}-\eqref{state_equ_z} together with the optimal state $\overline{y}\in Y_{2,0}$ and $\overline{z}=\mathcal{W}[S\overline{y}]\in \mathrm{H}^1(J_T)$. Consider the subdivision of $\overline{J_T}$ from Definition~\ref{Def:partition_[0,T]}.
	Then there exist adjoint states $p\in Y_{2,T}^*$ and $q\in \mathrm{BV}(J_T)$ of the following kind:
	There holds
	$
	B_i^*(p+Sq) = -\kappa \overline{u} \text{ in } U_i.
	$
	%That is, up to redefining $\overline{u}$ on a set of measure zero, $\overline{u}$ is left continuous on $[0,T)$ and right continuous at $T$, and
	%\[
	%\overline{u} \in \mathrm{L}^\infty\left(J_T;\prod\limits_{i=1}^m \mathrm{L}^2(\Gamma_i)\right).
	%\] 
	%
	%For $\Omega \subset \mathbb{R}^2$, there is some $\beta> 0$, such that 
	%\[
	%\overline{y}\in \mathrm{C}^\beta([t_0,T], \mathrm{L}^\infty(\Omega,\mathrm{R}^m)).
	%\]
	%
	%If in addition to Assumption \ref{Ass:existence_extension}, there are also bi-Lipschitz mappings for all $x\in \mathrm{D}_i$, $i=1,...,m$, then 
	%\[
	%\overline{y}\in \mathrm{C}^\beta([t_0,T], [\mathrm{C}^\gamma(\Omega)]^m).
	%\]
	We have
	\begin{equation}
	\begin{aligned}
	-\dot{p} + A^{*}_p p= \left[\frac{\partial}{\partial y}f(\overline{y},\overline{z})\right]^{*}(p+ Sq) -SA_p q + \overline{y} - y_d &&\text{for } t\in J_T,\ p(T)=0.\label{eq:evolution_p_COR}
	\end{aligned}
	\end{equation}
	$q$ is left continuous in $J_T$, right continuous at $T$ and absolutely continuous in $I_0$. $q$ solves the evolution equation
	$
	-\dot{q}= \langle p+Sq, \frac{\partial}{\partial z}f(\overline{y},\overline{z})\rangle_{X}
	$
	in every open subinterval of $I_0$.
	$
	\frac{d}{dt}\mathcal{P}[S\overline{y}](t)q(t) = 0\text{ for a.e. }t\in I_\partial
	$
	and there is a measure $d\mu \in \mathrm{C}(\overline{J_T})^*$ with support in $I_\partial$ such that
	%\begin{align*}
	%		\int_{I_\partial} -\varphi dq + \int_{I_\partial} \varphi d\mu = \int_{I_\partial} \varphi(t)[\nu_1(t) + \nu_2(t)] dt
	%\end{align*}
	%for all $\varphi \in \mathrm{C}(\overline{J_T})$ with support in $I_\partial$.
	$d\mu = dq + \langle p +Sq, \frac{\partial}{\partial z}f(\overline{y},\overline{z})\rangle_{X}dt$ as measures on $I_\partial$.
	For all $h\in U_i$ and with $y^{B_i\overline{u},B_ih}=G'[B_i\overline{u};B_ih]$ (see Theorem~\ref{Thm:state_equ_sol_op}) and $\mathcal{P}=\mathrm{Id} - \mathcal{W}$ (see Lemma~\ref{Lem:hyst_props}) there holds the optimality condition
	\begin{equation}
	\int_{0}^{T} \langle p+Sq, \frac{\partial}{\partial z}f(\overline{y},\overline{z}) \rangle_{X}\mathcal{P}'[S\overline{y};Sy^{B_i\overline{u},B_ih}] dt \leq \int_{I_\partial} Sy^{B_i\overline{u},B_ih} d\mu\label{opt_control_opt_cond_regular_f}.
	\end{equation}
	The absolute value of $q$ can only jump downwards in reverse time so that $q(T-)=q(T)=0$ and $|q(t-)|\leq |q(t+)|$ for all $t\in \overline{J_T}$.
	%If $f$ is continuously differentiable from $X^\alpha\times \mathbb{R}$ into $X$ then 
	%$\lambda_1= \left[\frac{\partial}{\partial y}f(\overline{y},\overline{z})\right]^{*}p$,
	%$\lambda_2= S\frac{\partial}{\partial y}f(\overline{y},\overline{z})q,\ 
	%\nu_1= \langle p, \frac{\partial}{\partial z}f(\overline{y},\overline{z})\rangle_{X}$,
	%$\nu_2= \langle Sq, \frac{\partial}{\partial z}f(\overline{y},\overline{z})\rangle_{X}
	%$
	%and the optimality condition is given by
	%\begin{align}
	%&\int_{0}^{T} \langle p+Sq, \frac{\partial}{\partial z}f(\overline{y},\overline{z}) \rangle_{X}\mathcal{P}'[S\overline{y};Sy^{B_i\overline{u},B_ih}] dt \leq \int_{I_\partial} Sy^{B_i\overline{u},B_ih} d\mu\label{opt_control_opt_cond_regular_f}.
	%\end{align}
	If the regularity Assumption \ref{Ass:regularity_assumption} is valid then
	$q$ is continuous at every $(0,\partial)$-switching time $t$ (see Definition~\ref{Def:switchin_times}) with $q(t)=0$.
	In this case, for every open interval $(c,d)\subset I_\partial$ it follows $q\equiv 0$ on $[c,d)$.
\end{corollary}
\begin{proof}
	If $f$ is continuously differentiable then by Lemma~\ref{Lem:limits_q_and_p_and_props_of_p} and Lemma~\ref{Lem:q_in_H1} we can replace $\lambda_1= \left[\frac{\partial}{\partial y}f(\overline{y},\overline{z})\right]^{*}p$,
	$\lambda_2= S\frac{\partial}{\partial y}f(\overline{y},\overline{z})q,\ 
	\nu_1= \langle p, \frac{\partial}{\partial z}f(\overline{y},\overline{z})\rangle_{X}$,
	$\nu_2= \langle Sq, \frac{\partial}{\partial z}f(\overline{y},\overline{z})\rangle_{X}
	$ in Theorem~\ref{Thm:opt_control_boundary_opt_cond_limit_adjoint}.
	This yields all statements except for the optimality condition.
	\eqref{opt_control_opt_cond_general} takes the form
\begin{equation*}
\begin{split}
&\int_{0}^{T} \langle \left[\frac{\partial}{\partial y}f(\overline{y},\overline{z})\right]^{*}(p+Sq) , y^{B_i\overline{u},B_ih} \rangle_{\mathrm{dom}(A_p)} + \langle p+Sq, \frac{\partial}{\partial z}f(\overline{y},\overline{z}) \rangle_{X}Sy^{B_i\overline{u},B_ih} dt\\
&\leq \int_{I_\partial} Sy^{B_i\overline{u},B_ih} d\mu + \int_{0}^{T} \langle p+Sq , \frac{\partial}{\partial y}f(\overline{y},\overline{z})y^{B_i\overline{u},B_ih} + \frac{\partial}{\partial z}f(\overline{y},\overline{z}) \mathcal{W}'[S\overline{y};Sy^{B_i\overline{u},B_ih}] \rangle_{X}  dt.
\end{split}
\end{equation*}

Because $\mathcal{P}=\mathrm{Id} - \mathcal{W}$ (see Lemma~\ref{Lem:hyst_props})
we have
$
Sy^{B_i\overline{u},B_ih} - \mathcal{W}'[S\overline{y};Sy^{B_i\overline{u},B_ih}] = \mathcal{P}'[S\overline{y};Sy^{B_i\overline{u},B_ih}].
$

This yields the optimality condition \eqref{opt_control_opt_cond_regular_f}.
%\begin{align*}
%&\int_{0}^{T} \langle p+Sq, \frac{\partial}{\partial z}f(\overline{y},\overline{z}) \rangle_{X}\mathcal{P}'[S\overline{y};Sy^{B_i\overline{u},B_ih}] dt \leq \int_{I_\partial} Sy^{B_i\overline{u},B_ih} d\mu.
%\end{align*}
\end{proof}

\subsection{Improved optimality conditions and uniqueness for distributed controls}\label{Subsec:Improvement_distributed}

We want to replace $y^{B_i\overline{u},B_ih}$ in \eqref{opt_control_opt_cond_general_THM} and \eqref{opt_control_opt_cond_regular_f} by an arbitrary function of an appropriate space. This would certainly improve the optimality conditions in Theorem~\ref{Thm:opt_control_boundary_opt_cond_limit_adjoint} and Corollary~\ref{Cor:opt_cond_general}.
It is not possible in the general case $i\in \{1,2\}$ without density of the ranges of $B_i$. Therefore, we restrict ourselves to problem \eqref{opt_control_ control_problem}-\eqref{state_equ_z} with distributed controls $u\in U_1$ in this subsection.
Suppose that $p$ in (A1) in Assumption~\ref{Ass:general_ass_and_short_notation_1} is chosen close to two such that $\frac{1}{2}<1-\frac{1}{p}-\frac{1}{d}$. Then $2<\frac{dp'}{d-p'}$ and by \cite[Remark 2.7]{muench} we have the compact embedding
$\mathbb{W}_{\Gamma_D}^{1,p'}(\Omega)\hhookrightarrow [\mathrm{L}^2(\Omega)]^m$ which is also one-to-one.
That is, in this case $B_1$ has dense range.
% if $p$ from Assumption~\ref{Ass:general_ass_and_short_notation_1} is chosen close to two, so that $2<\frac{dp'}{d-p'}$. In this case, by \cite[Remark 2.7]{muench} we have the compact and embedding
%$\mathbb{W}_{\Gamma_D}^{1,p'}(\Omega)\hhookrightarrow [\mathrm{L}^2(\Omega)]^m$ which is also one-to-one.
%
% .
% is close to two follows because of the compact embedding $\mathbb{W}_{\Gamma_D}^{1,p'}(\Omega) \hookrightarrow [\mathrm{L}^2(\Omega)]^m$ \cite[cf.][Remark 3.2]{rehbergsystems} which is also one-to-one.
In Corollary~\ref{Cor:opt_control_distrib_opt_cond} in Subsection~\ref{Subsubsec:Improved optimality conditions} we improve the optimality conditions from Theorem~\ref{Thm:opt_control_boundary_opt_cond_limit_adjoint} and Corollary~\ref{Cor:opt_cond_general} for this case. 
For $i=1$ we also prove uniqueness of $p,q$ and $d\mu$ if $f$ is continuously differentiable, see Corollary~\ref{Cor:Uniqueness_dist_contr} in Subsection~\ref{Subsubsec:uniqueness}.

\subsubsection{Improved optimality conditions}\label{Subsubsec:Improved optimality conditions}

We improve the optimality conditions \eqref{opt_control_opt_cond_general_THM} and \eqref{opt_control_opt_cond_regular_f}.
%Since $B_1$ has dense range one proves just as in \cite[Lemma 5.2]{meyeroptimal} that the set
%$
%	\{y^{B_1\overline{u},B_1h}: h\in  U_1\}
%$
%is dense in $Y_{2,0}$.
%Unfortunately, we can not continue as in \cite[Theorem 5.3]{meyeroptimal} to derive a pointwise optimality condition.
%The reason is that $\mathcal{W}'[S\overline{y};S\cdot]$ is non-local in time.
%It is true though that $\mathcal{W}'[S\overline{y};\cdot]$ is positively homogeneous.
%As in \cite[Theorem 5.3]{meyeroptimal}, for arbitrary $\eta \in Y_{2,0}$, we choose a sequence in $
%\{y^{B_1\overline{u},B_1h}: h\in  U_1\}
%$ which converges to $\eta$. We pass to the limit in \eqref{opt_control_opt_cond_general_THM} to obtain
%\begin{align*}
%\int_{0}^{T} \langle \lambda_1 + \lambda_2 + S(\nu_1 + \nu_2),\eta \rangle_{\mathrm{dom}(A_p)}  dt  \leq \int_{I_\partial} S\eta d\mu + \int_{0}^{T} \langle p+Sq , F'[\overline{y};\eta] \rangle_{X}  dt
%\end{align*}
%for all $\eta\in Y_{2,0}$.
%%We continue in a similar manner as in \cite[Theorem 5.3]{meyeroptimal}.
%Let $v\in \mathrm{dom}(A_p)$ with $Sv> 0$ be given.
%Furthermore, let $\varphi \in \mathrm{C}^\infty_0(J_T)$ be arbitrary.
%Then $v\varphi\in Y_{2,0}$ and
%$
%\mathcal{W}'[S\overline{y};S(v\varphi)]= Sv \mathcal{W}'[S\overline{y};\varphi].
%$
%This proves:
\begin{corollary}[Optimality condition for distributed controls]\label{Cor:opt_control_distrib_opt_cond}
	 Let Assumption~\ref{Ass:general_ass_and_short_notation_1} and Assumption~\ref{Ass:regularized_problem} hold and let $\frac{1}{2}<1-\frac{1}{p}-\frac{1}{d}$.	 
	 Assume that $\overline{u}\in U_1$ is a solution of problem \eqref{opt_control_ control_problem}-\eqref{state_equ_z} with $i=1$, together with the state $\overline{y}\in Y_{2,0}$ and $\overline{z}=\mathcal{W}[S\overline{y}]\in \mathrm{H}^1(J_T)$.	 
	 Let $v\in \mathrm{dom}(A_p)$ with $Sv> 0$ and $\varphi \in \mathrm{C}^\infty_0(J_T)$ be arbitrary.
	 Then in addition to \eqref{opt_control_opt_cond_general_THM} in Theorem~\ref{Thm:opt_control_boundary_opt_cond_limit_adjoint} there holds
	\begin{equation*}
	\begin{split}
		&\int_{0}^{T} \langle \lambda_1 + \lambda_2 , \frac{v}{Sv}\varphi \rangle_{\mathrm{dom}(A_p)} + (\nu_1 + \nu_2)\varphi  dt\\
		&\leq \int_{I_\partial} \varphi d\mu + \int_{0}^{T} \langle p+Sq , f'[(\overline{y},\overline{z});((v/Sv)\varphi, \mathcal{W}'[S\overline{y};\varphi])] \rangle_{X}  dt.
	\end{split}
	\end{equation*}
	If $f$ is continuously differentiable then in addition to \eqref{opt_control_opt_cond_regular_f} in Corollary~\ref{Cor:opt_cond_general} there holds
	\begin{equation*}
		\begin{aligned}
		\int_{0}^{T} \langle p+Sq , \frac{\partial}{\partial z}f(\overline{y},\overline{z})  \rangle_{X}  \mathcal{P}'[S\overline{y};\varphi]dt\leq \int_{I_\partial} \varphi \,d\mu &&\text{for all } \varphi \in \mathrm{C}^\infty_0(J_T).
		\end{aligned}
	\end{equation*}
\end{corollary}
\begin{proof}
	Since $B_1$ has dense range one proves just as in \cite[Lemma 5.2]{meyeroptimal} that the set
	$
	\{y^{B_1\overline{u},B_1h}: h\in  U_1\}
	$
	is dense in $Y_{2,0}$.
	Unfortunately, we can not continue as in \cite[Theorem 5.3]{meyeroptimal} to derive a pointwise optimality condition.
	The reason is that for any $\zeta\in \mathrm{C}(\overline{J_T};X^\alpha)$ the function $\mathcal{W}'[S\overline{y};S\zeta]$ is non-local in time.
	Nevertheless, we can still make use of the fact that $\mathcal{W}'[S\overline{y};\cdot]$ and $f'$ are positive homogeneous.
	First of all, as in \cite[Theorem 5.3]{meyeroptimal}, for arbitrary given $\eta \in Y_{2,0}$, we choose a sequence in $
	\{y^{B_1\overline{u},B_1h}: h\in  U_1\}
	$ which converges to $\eta$. We pass to the limit in \eqref{opt_control_opt_cond_general_THM} and obtain
	\begin{equation*}
	\int_{0}^{T} \langle \lambda_1 + \lambda_2 + S(\nu_1 + \nu_2),\eta \rangle_{\mathrm{dom}(A_p)}  dt  \leq \int_{I_\partial} S\eta d\mu + \int_{0}^{T} \langle p+Sq , F'[\overline{y};\eta] \rangle_{X}  dt,
	\end{equation*}
	where $F'[\overline{y};\eta](t)=f'[(\overline{y}(t),\mathcal{W}[S\overline{y}](t));(\overline{y}(t),\mathcal{W}'[S\overline{y};S\eta](t))]$.
	%We continue in a similar manner as in \cite[Theorem 5.3]{meyeroptimal}.
	Let $v\in \mathrm{dom}(A_p)$ with $Sv> 0$ be given.
	Furthermore, let $\varphi \in \mathrm{C}^\infty_0(J_T)$ be arbitrary.
	Then $v\varphi\in Y_{2,0}$ and
	$
	\mathcal{W}'[S\overline{y};S(v\varphi)]= Sv \mathcal{W}'[S\overline{y};\varphi].
	$
	Setting $\eta = \varphi v$ and rearranging yields
	\begin{equation*}
	\begin{split}
	&\int_{0}^{T} \langle \lambda_1 + \lambda_2, \varphi v \rangle_{\mathrm{dom}(A_p)} + \varphi(\nu_1 + \nu_2) Sv  dt  \\
	&\leq \int_{I_\partial} \varphi Sv d\mu + \int_{0}^{T} \langle p+Sq , f'[(\overline{y},\overline{z});(v\varphi,  Sv\mathcal{W}'[S\overline{y};\varphi])] \rangle_{X}  dt.
	\end{split}
	\end{equation*}
	Dividing both sides by $Sv$ proves the first statement. The second inequality is shown analogously.
\end{proof}

\subsubsection{Uniqueness of the adjoint variables}\label{Subsubsec:uniqueness}

If $f$ is continuously differentiable we can also show uniqueness of the adjoint couple.

\begin{corollary}[Unique adjoint system for distributed controls]\label{Cor:Uniqueness_dist_contr}
	Let Assumption~\ref{Ass:general_ass_and_short_notation_1} and Assumption~\ref{Ass:regularized_problem} hold and let $\frac{1}{2}<1-\frac{1}{p}-\frac{1}{d}$. Moreover, suppose that $f$ is continuously differentiable from $X^\alpha\times \mathbb{R}$ into $X$.	 
	Assume that $\overline{u}\in U_1$ is a solution of problem \eqref{opt_control_ control_problem}-\eqref{state_equ_z} with $i=1$, together with the state $\overline{y}\in Y_{2,0}$ and $\overline{z}=\mathcal{W}[S\overline{y}]\in \mathrm{H}^1(J_T)$.	 
	Then in the setting of Corollary~\ref{Cor:opt_cond_general} the adjoint couple $p\in Y_{2,T}^*$ and $q\in \mathrm{BV}(J_T)$ together with the measure $d\mu$ in $\mathrm{C}(\overline{J_T})^*$ is unique.
\end{corollary}
	
	\begin{proof}
		Because $B_1$ has dense range we have 
		$
		\mathrm{ker}(B_1^*) = \overline{\mathrm{ran}(B_1)}^\perp = \{0\}.
		$
		Therefore by Corollary~\ref{Cor:opt_cond_general} we obtain
%		\[
%		B_1^*(p+Sq) = -\kappa\overline{u} \text{ in } U_1
%		\]
%		implies
		\begin{equation}\label{Lem:opt_control_distrib_uniqueness_equ}
		\begin{aligned}
		p+Sq = -\kappa(B_1^*)^{-1}\overline{u}&& \text{in } X^* \text{ a.e. in } J_T,
		\end{aligned}
		\end{equation}
		cf. \cite[Theorem 4.15]{meyeroptimal}.
%		We also know
%		\begin{align*}
%		-\dot{p} + A^{*}_p p= \left[\frac{\partial}{\partial y}f(\overline{y},\overline{z})\right]^{*}(p + Sq) -SA_pq + \overline{y} - y_d \text{for } t\in J_T,\ p(T)=0.
%		\end{align*}
		Suppose there are two adjoint couples $(p_1,q_1), (p_2,q_2)$ which satisfy the conditions of Corollary~\ref{Cor:opt_cond_general}.
		Let $\zeta\in \mathrm{L}^2(J_T;\mathrm{dom}(A_p))$ be arbitrary.
		Then by \eqref{eq:evolution_p_COR} and \eqref{Lem:opt_control_distrib_uniqueness_equ} there holds
		\begin{equation*}
		\begin{split}
		&\langle \dot{p}_2 - \dot{p}_1 , \zeta \rangle_{\mathrm{L}^2(J_T;\mathrm{dom}(A_p))}\\
		&= \langle \left[\frac{\partial}{\partial y}f(\overline{y},\overline{z})\right]^{*}(p_1 + Sq_1 
		- (p_2 + Sq_2)) -A_p^*(p_2-p_1)-SA_p (q_2-q_1), \zeta \rangle_{\mathrm{L}^2(J_T;\mathrm{dom}(A_p))}\\
		&= \langle p_1 + Sq_1 - (p_2 + Sq_2), \frac{\partial}{\partial y}f(\overline{y},\overline{z}) \zeta \rangle_{\mathrm{L}^2(J_T;X)}
		 - \langle p_2 + Sq_2 - (p_1 + Sq_1), A_p \zeta \rangle_{\mathrm{L}^{2}(J_T;X)} = 0.
		\end{split}
		\end{equation*}
		This implies $\dot{p}_2 = \dot{p}_1 $ in $\mathrm{L}^{2}(J_T; [\mathrm{dom}(A_p)]^*)$.
		Together with 
		$p_1(T) = p_2(T)=0 \in [\mathrm{dom}(A_p)]^*$ we obtain $p_1=p_2$ in $\mathrm{L}^{2}(J_T; [\mathrm{dom}(A_p)]^*)$.
		Since the embedding $\mathrm{dom}(A_p)\hookrightarrow X$ is dense, the embedding of $X^*$ into $[\mathrm{dom}(A_p)]^*$ is one-to-one and $p_1=p_2$ also in $\mathrm{L}^{2}(J_T; X^*)$ and then in $Y_{2,T}^*$.
		Let $v\in \mathrm{dom}(A_p)$ be given with $Sv>0$.
		We already know $p_1=p_2$ so that
		$
		S(q_1-q_2) = 0 \text{ in } X^* \text{ a.e. in } J_T
		$
		because of \eqref{Lem:opt_control_distrib_uniqueness_equ}.
		But then
		\begin{equation*}
		\begin{aligned}
		q_1-q_2 = \frac{(q_1-q_2)Sv}{Sv} = \frac{\langle S(q_1-q_2),v\rangle_{X}}{Sv} = 0&& \text{ in } X^* \text{ a.e. in } J_T,
		\end{aligned}
		\end{equation*}
		so that $q_1=q_2$ in $\mathrm{L}^1(J_T)$. 
		This way we obtain
		\begin{equation*}
			\int_{[0,T]} |dq_1-dq_2| = \sup \left\{ \int_{0}^{T}(q_1-q_2) \dot{\varphi} dt : \varphi\in \mathrm{C}_0^1(J_T),\ |\varphi|\leq 1 \right\} =0,
		\end{equation*}
		which implies
		$
			dq_1 - dq_2 = 0
		$
		as measures on $\overline{J_T}$ according to \cite[XII.7]{visintin2013differential}.	
		This yields $q_1=q_2\in \mathrm{BV}(0,T)$.
		From Corollary~\ref{Cor:opt_cond_general} we conclude
		$d\mu_1 = d\mu_2$
		and the proof is complete.
		\end{proof}

\section{Higher regularity of the solutions of the optimal control problem}
In this section we improve the regularity of the optimal control $\overline{u}\in U_i$, $i\in\{1,2\}$, and then also of the optimal state $\overline{y}=G(B_i\overline{u})$ and $\overline{z}=\mathcal{W}[S\overline{y}]$.
We denote 
$\tilde{U}_1:=[\mathrm{L}^2(\Omega)]^m$ and $\tilde{U}_2:=\prod_{j=1}^{m}\mathrm{L}^2(\Gamma_{N_j},\mathcal{H}_{d-1})$.
We want to exploit the equation
$
B_i^*(p+Sq) = -\kappa \overline{u} \text{ in } [\tilde{U}_i]^* \text{ a.e. in } J_T
$
which follows from Theorem~\ref{Thm:opt_control_boundary_opt_cond_limit_adjoint}.
In order to make use of the time-regularity of $p+Sq$ we need to enforce the conditions on $B_i$.

\begin{assumption}\label{Ass:regularity_of_B}
	For $i\in\{1,2\}$, the operator $B_i:\tilde{U}_i\rightarrow X$ in N(5) 
%	Assumption~\ref{Ass:general_ass_and_short_notation_1} 
	is also continuous as a mapping into $X^\gamma$ for some $\gamma\in (0,1]$.
%	Let $\beta:=\min\{\alpha,\gamma\}$.	
%	We define the canonical embeddings $I_{\gamma,\beta}$ from $X^\gamma$ into $X^\beta$ and $I_{\beta,0}$ from $X^\beta$ into $X$.
	We denote by $I_{(\gamma)}$ the canonical embedding from $X^\gamma$ into $X$.
	Then the assumption is equivalent to the fact that
	$
	B_i= I_{(\gamma)}  \tilde{B}_i
	$
	for a linear and continuous function $\tilde{B}_i:\tilde{U}_i\rightarrow X^\gamma$.
\end{assumption}

\begin{theorem}[Higher regularity]\label{Thm:HigherRegularity}
	In the setting of Theorem~\ref{Thm:opt_control_boundary_opt_cond_limit_adjoint} let Assumption~\ref{Ass:regularity_of_B} hold for some $\gamma \in (0,1]$. 
	
	If $\gamma > \frac{1}{2}$, then $\overline{u}\in \mathrm{L}^\infty(J_T;\tilde{U}_i)$, $\overline{y}\in Y_{s,0}$ and $\overline{z}\in \mathrm{W}^{1,s}(J_T)$ for arbitrary $s\in (1,\infty)$.
	If $\frac{1}{2}(1+\frac{d}{p})<1$, which is the case when $d=2$ and $p>2$ in (A1) in Assumption~\ref{Ass:general_ass_and_short_notation_1}, this implies $\overline{y}\in \mathrm{C}(\overline{J_T};[\mathrm{L}^\infty(\Omega)]^m)$.
	If in addition $\Omega$ is a Lipschitz domain then $\overline{y}$ is H{\"o}lder continuous in time and space.
	
	If $\gamma \leq \frac{1}{2}$, then $\overline{u}\in \mathrm{L}^{\frac{2}{1-2s}}(J_T;\tilde{U}_i)$, $\overline{y}\in Y_{2/(1-2s),0}$ and $\overline{z}\in \mathrm{W}^{1,\frac{2}{1-2s}}(J_T)$ for arbitrary $s\in(0,\gamma)$.
	This implies $\overline{y}\in \mathrm{C}(\overline{J_T};X^\theta)$ for any $\theta \in (0,\frac{1}{2}+\gamma)$.
	If $\gamma >\frac{d}{2p}$, with $d$ and $p$ in (A1) in Assumption~\ref{Ass:general_ass_and_short_notation_1}, this implies $\overline{y}\in \mathrm{C}(\overline{J_T};[\mathrm{L}^\infty(\Omega)]^m)$.
	If in addition $\Omega$ is a Lipschitz domain then $\overline{y}$ is H{\"o}lder continuous in time and space.
\end{theorem}
\begin{proof}
	
	First note that for $0\leq \beta\leq \gamma \leq 1$ we have the compact and dense embeddings $X^\gamma\hhookrightarrow X^\beta \hhookrightarrow X$ \cite[Theorem 1.4.8]{henry}. This implies $X^*\hookrightarrow [X^\gamma]^* \hookrightarrow [X^\beta]^*$.
	By the properties of complex interpolation and with Remark~\ref{Rem:maximal parabolic regularity} there holds
	\begin{equation}
		[[\mathrm{dom}(A)]^*,X^*]_{1-\gamma}=[X^*,[\mathrm{dom}(A_p)]^*]_{\gamma}=[X,\mathrm{dom}(A_p)]_{\gamma}^*\simeq [X^\gamma]^*.\label{eq:interpolation}
	\end{equation}
	\begin{itemize}[leftmargin=0.4cm]
		\item We prove the case when Assumption~\ref{Ass:regularity_of_B} is fulfilled with $\gamma > \frac{1}{2}$:
		
		Since $1-\gamma < \frac{1}{2}$ we obtain (as in Remark~\ref{Rem:embeddings}) an embedding 
		\begin{equation*}
		Y_{2,T}^*\subset \mathrm{H}^1(J_T;[\mathrm{dom}(A)]^*)\cap\mathrm{L}^2(J_T;X^*)\hookrightarrow \mathrm{C}(\overline{J_T};[[\mathrm{dom}(A)]^*,X^*]_{1-\gamma}).
		\end{equation*}
		Therefore, by \eqref{eq:interpolation} the regularity
		$p\in Y_{2,T}^*$ in Theorem~\ref{Thm:opt_control_boundary_opt_cond_limit_adjoint} implies that we can identify the function $p\in \mathrm{L}^2(J_T;X^*)$ and the representative $\tilde{p}$ of $p$ in $\mathrm{C}(\overline{J_T};[X^\gamma]^*)$.
		This allows us to identify the function $B_i^* p \in \mathrm{L}^2(J_T;[\tilde{U}_i]^*) $ and $\tilde{B}_i^* \tilde{p} \in \mathrm{C}(\overline{J_T};[\tilde{U}_i]^*)$.
		We also have $Sq\in \mathrm{L}^\infty(J_T;X^*)$ since by Theorem~\ref{Thm:opt_control_boundary_opt_cond_limit_adjoint} there holds $q\in \mathrm{BV}(J_T)$ and because $S\in X^*$ by (A3) in Assumption~\ref{Ass:general_ass_and_short_notation_1}. 
%		So we can identify the function $Sq\in \mathrm{L}^\infty(J_T;X^*)$ with $I_{(\gamma)}^* Sq\in \mathrm{L}^\infty(J_T;[X^\gamma]^*)$ and then $B_i p \in \mathrm{L}^\infty(J_T;[\tilde{U}_i]^*) $ with $\tilde{B}_i (I_{(\gamma)}^* Sq) \in \mathrm{C}(\overline{J_T};[\tilde{U}_i]^*)$
		That is, $B_i^*Sq \in \mathrm{L}^\infty(J_T;[\tilde{U}_i]^*)$.
%		We also have that $[I_{\gamma,\beta}\tilde{B}_i]^*:[X^\beta]^*\rightarrow \tilde{U}_i^*\simeq \tilde{U}_i$ is continuous.
		Again with Theorem~\ref{Thm:opt_control_boundary_opt_cond_limit_adjoint} and the identification of $B_i^* p$ and $\tilde{B}_i^* \tilde{p}$ we arrive at
		\begin{equation*}
		\begin{aligned}
		\tilde{B}_i^* \tilde{p} + B_i^*Sq = B_i^*(p+Sq)= -\kappa \overline{u} &&\text{in } [\tilde{U}_i]^*, \text{ a.e. in } J_T.
		\end{aligned}
		\end{equation*}
		The functions on the left side are contained in $\mathrm{L}^\infty(J_T;[\tilde{U}_i]^*)$. We identify $[\tilde{U}_i]^*$ with $\tilde{U}_i$, so that $\overline{u}\in \mathrm{L}^\infty(J_T;\tilde{U}_i)$.
		Now we use the higher regularity of $\overline{u}$ to prove a better regularity also for $\overline{y}$.
		Since $\overline{u}\in \mathrm{L}^\infty(J_T;\tilde{U}_i)$, Theorem~\ref{Thm:state_equ_sol_op} yields $\overline{y}\in Y_{s,0}$ for arbitrary $s\in (1,\infty)$.
		From Remark~\ref{Rem:maximal parabolic regularity} and Remark~\ref{Rem:embeddings} we obtain $\overline{y}\in \mathrm{C}(\overline{J_T};X^\theta)$ for arbitrary $\theta \in [0,1)$.
		In \cite[Theorem 3.3]{L_infinity_estimates} it is shown that $X^\theta$ is a subset of $[\mathrm{L}^\infty(\Omega)]^m$ if $\theta>\frac{1}{2}(1+\frac{d}{p})$. By Remark~\ref{Rem:maximal parabolic regularity} we are guaranteed that we can choose $p>2$. So at least if $d=2$ there is some $\theta\in (0,1)$ with $\theta>\frac{1}{2}(1+\frac{d}{p})$ and therefore $\overline{y}\in \mathrm{C}(\overline{J_T};[\mathrm{L}^\infty(\Omega)]^m)$. 
		If $d=2$ and $p>2$ and if $\Omega$ is regular enough, for example a Lipschitz domain, then by \cite[Theorem 4.5]{disser2015h} the state $\overline{y}$ is even H{\"o}lder continuous in time and space.
	
		\item
		We prove the statement for the case when Assumption~\ref{Ass:regularity_of_B} is fulfilled with $\gamma \leq \frac{1}{2}$:

		From \cite[Theorem 3 and (22)]{amann2005nonautonomous} it follows
		\begin{equation*}
		Y_{2,T}^*\subset\mathrm{H}^1(J_T;[\mathrm{dom}(A)]^*)\cap\mathrm{L}^2(J_T;X^*)\hookrightarrow \mathrm{L}^{\frac{2}{1-2s}}(J_T;[[\mathrm{dom}(A)]^*,X^*]_{1-\gamma})
		\end{equation*}
		for arbitrary $s\in(0,\gamma)$.
		So the regularity
		$p\in Y_{2,T}^*$ in Theorem~\ref{Thm:opt_control_boundary_opt_cond_limit_adjoint} together with \eqref{eq:interpolation} implies that we can identify $p\in \mathrm{L}^2(J_T;X^*)$ and the representative $\tilde{p}$ of $p$ in $\mathrm{L}^{\frac{2}{1-2s}}(J_T;[X^\gamma]^*)$ and then  $B_i^* p \in \mathrm{L}^2(J_T;[\tilde{U}_i]^*) $ and $\tilde{B}_i^* \tilde{p} \in \mathrm{L}^{\frac{2}{1-2s}}(J_T;[\tilde{U}_i]^*)$.
		We proceed as for the case $\gamma > \frac{1}{2}$ to prove $\overline{u}\in \mathrm{L}^{\frac{2}{1-2s}}(J_T;\tilde{U}_i)$ for arbitrary $s\in(0,\gamma)$. Theorem~\ref{Thm:state_equ_sol_op} yields $\overline{y}\in Y_{2/(1-2s),0}$ for arbitrary $s\in(0,\gamma)$ and from Remark~\ref{Rem:maximal parabolic regularity} and Remark~\ref{Rem:embeddings} it follows $\overline{y}\in \mathrm{C}(\overline{J_T};X^\theta)$ for arbitrary $\theta \in [0,1-\left(\frac{2}{1-2s}\right)^{-1})=[0,\frac{1}{2}+s)$. Because $s\in(0,\gamma)$ is arbitrary, this holds for all $\theta \in [0,\frac{1}{2}+\gamma)$. The remaining statements are shown just as for $\gamma >\frac{1}{2}$.
	\end{itemize}
\end{proof}

\begin{remark}\label{Rem:Example_higher_regularity}
	For example, take $d=2$ and $p>2$ in (A1) in Assumption~\ref{Ass:general_ass_and_short_notation_1} and adopt the assumptions and the notation in Theorem~\ref{Thm:opt_control_boundary_opt_cond_limit_adjoint}. By \cite[Remark 2.7]{muench} we have the compact embedding
	\begin{equation*}
	\mathbb{W}_{\Gamma_D}^{1,p'}(\Omega)\hhookrightarrow [\mathrm{L}^2(\Omega)]^m = [\tilde{U}_1]^*
	\end{equation*}
	because $p'>1$ and then $2<\frac{dp'}{d-p'}=\frac{2p'}{2-p'}$.
	Therefore we can embed functions $u\in\tilde{U}_1$ into $\mathbb{W}_{\Gamma_D}^{-1,p}(\Omega)$ by the assignment
	$
	\int_\Omega u\cdot v\, dx,
	$
	$\forall v\in \mathbb{W}_{\Gamma_D}^{1,p'}(\Omega)$.
	We slightly reinforce Assumption~\ref{Ass:general_ass_and_short_notation_1}.
	Suppose that \cite[Assumption~2.2]{griepentrog2002interpolation} holds for $\Omega$ and for all $\Gamma_{D_j}$, $j\in \{1,\ldots m\}$. This essentially means that Assumption~\ref{Ass:domain} holds for all $x\in \partial\Omega$ and that the functional determinant of each bi-Lipschitz transformation $\phi_x$ is constant a.e.
	For example, this is the case if $\Omega$ is a Lipschitz domain \cite[Remark~2.3]{griepentrog2002interpolation}. The rest in Assumption~\ref{Ass:general_ass_and_short_notation_1} remains the same.
	With this assumption one has
	\begin{equation*}
		[\mathbb{W}_{\Gamma_D}^{-1,p}(\Omega),[\mathrm{L}^2(\Omega)]^m]_{\theta} = \mathbb{W}_{\Gamma_D}^{-\theta,p}(\Omega)
	\end{equation*}
	for $\theta\in (0,1)$ \cite[Theorem 3.1]{griepentrog2002interpolation}.
	This way we obtain an embedding $\tilde{U}_1\hookrightarrow \mathbb{W}_{\Gamma_D}^{-\theta,p}(\Omega)$.
	Furthermore, we have 
	\begin{equation*}
	\mathbb{W}_{\Gamma_D}^{-\theta,p}(\Omega) = [\mathbb{W}_{\Gamma_D}^{-1,p}(\Omega),\mathbb{W}_{\Gamma_D}^{1,p}(\Omega)]_{\gamma} \simeq X^\gamma
	\end{equation*}
	for $-\theta = -1+2\gamma$ by \cite[Theorem 3.5]{griepentrog2002interpolation} and Remark~\ref{Rem:maximal parabolic regularity}.
	For any $\gamma \in (0,\frac{1}{2})$ there holds $\theta\in (0,1)$ for $\theta=1-2\gamma$ so that we obtain an embedding $\tilde{U}_1\hookrightarrow X^\gamma$.
	Therefore, Assumption~\ref{Ass:regularity_of_B} is fulfilled for $B_1$ with any $\gamma\in (0,\frac{1}{2})$.
	By Theorem~\ref{Thm:HigherRegularity} it follows $\overline{u}\in \mathrm{L}^{\frac{2}{1-2s}}(J_T;\tilde{U}_1)$, $\overline{y}\in Y_{1/(1-2s),0}$ and $\overline{z}\in \mathrm{W}^{1,\frac{2}{1-2s}}(J_T)$ for arbitrary $s\in(0,\gamma)$.
	Since $d=2$ and $p>2$ we can choose $\gamma\in (0,\frac{1}{2})$ such that $\gamma>\frac{d}{2p}$. Theorem~\ref{Thm:HigherRegularity} yields $\overline{y}\in \mathrm{C}(\overline{J_T};[\mathrm{L}^\infty(\Omega)]^m)$. If $\Omega$ is a Lipschitz domain, then $\overline{y}$ is H{\"o}lder continuous in time and space.
\end{remark}

\section{The value function of a perturbed control problem}
In this section we analyze stability properties of the minimal value function of a perturbed problem which is similar to \eqref{opt_control_ control_problem}-\eqref{state_equ_z}.
This analysis is only relevant if the set of controls is restricted.
That is, for $i\in\{1,2\}$ we consider a convex closed subset set $C\subset U_i$ as our set of feasible controls and minimize the cost function over this set.
For given $r\in U_i$ we analyze the perturbed problem
\begin{equation}
\min_{u\in C}J(G(B_i(u+r)),u+r)=\|G(B_i(u+r))-y_d\|^2_{U_1} +\frac{\kappa}{2} \|u+r\|_{U_i}^2.\label{opt_control_prob_perturbed}
\end{equation}
We define the corresponding minimal value function
\begin{equation}
v:U_i\rightarrow \mathbb{R},\ r\mapsto v(r):=\min_{u \in C} J(G(B_i(u+r)),u+r)\label{def:val_function}
\end{equation}
and the multifunction
\begin{equation}
V: r\in U_i \mapsto V(r):=\{u\in C:\ J(G(B_i(u+r)),u+r)= v(r)\}.\label{def:set_function}
\end{equation}
We analyze the continuity properties of $v$ and $V$. The proof is quite similar to the one of \cite[Proposition 4.4]{bonnans}.

\begin{theorem}[Value function]\label{Thm:ValueFunction}
	Let Assumption~\ref{Ass:general_ass_and_short_notation_1} hold. For $i\in\{1,2\}$, let $C\subset U_i$ be convex and closed. Consider the optimal control problem \eqref{opt_control_prob_perturbed} for $r\in U_i$ together with the corresponding minimal value function $v$, defined by \eqref{def:val_function}, and the multifunction $V$ from \eqref{def:set_function}.
	Then $v$ is weakly lower semicontiuous.
	If $C$ is compact in $U_i$ then $v$ is also upper semicontiuous and therefore continuous.
	In this case, also the multifunction $V$ is upper semicontinuous, i.e. for each $r_0\in U_i$ and for any neighborhood $U_{V(r_0)}$ of $V(r_0)$ there exists a neighborhood
	$U_{r_0}$ of $r_0$ such that $V(r) \subset U_{V(r_0)}$ for all $r\in U_{r_0}$, cf. \cite[Chapter 4.1]{bonnans}.
\end{theorem}
\begin{proof}
	Note first that problem \eqref{opt_control_prob_perturbed} is well-posed. This follows essentially as Theorem~\ref{Thm:opt_control_existence} for the unperturbed problem \eqref{opt_control_ control_problem}-\eqref{state_equ_z}.
	We claim that $v$ is weakly lower semicontiuous (and then also lower semicontiuous).	
	Let $r_0\in U_i$ be given.
	We have to prove that for any sequence $\lbrace r_n \rbrace$ with $r_n\rightharpoonup r_0$, $n\rightarrow \infty$, it holds
	$
	v(r_0) \leq \liminf_{n\rightarrow \infty} v(r_n).
	$
	Let $\lbrace r_n \rbrace$ be such a sequence. Then $\lbrace r_n \rbrace$ is bounded in $U_i$.
	Let $\varepsilon>0$ be arbitrary.
	We show that for $n_0$ large enough
	$
	v(r_0)-\varepsilon \leq v(r_n)
	$
	for all $n\geq n_0$.
	Since $\lbrace r_n \rbrace$ is bounded, by definition of $J$ we can find some $R>0$ with
	$
	\mathop\cup_{n\in\mathbb{N}} V(r_n) \subset \mathrm{B}_{U_i}(0,R).
	$
	Suppose there exists a subsequence $\{r_{n_k}\} $ of $\{r_{n}\}$ such that for each $n_k$ there is some $u_{n_k}\in V(r_{n_k})$ with 	
	$
	v(r_0)-\varepsilon > J(G(B_i(u_{n_k}+r_{n_k})),u_{n_k}+r_{n_k}).
	$
	Note that $\lbrace u_{n_k} \rbrace$ is a bounded subset of $C$ and that $U_i$ is reflexive.
	Being convex and closed, $C$ is weakly compact.
	Hence, there is another subsequence (w.l.o.g. we consider the whole sequence $\lbrace u_{n_k} \rbrace$) and some $\overline{u}\in C$ such that $u_{n_k}\rightharpoonup \overline{u}$ with $k\rightarrow \infty$.
	$J(G(B_i(\cdot+r_0)),\cdot)$ is weakly lower semicontinuous.
	This follows by weak lower semicontinuity of the norm in $U_i$ and of the solution mapping $G$
%	$
%	u_n \rightharpoonup u \text{ in }U_i \Rightarrow G(B_iu_n) \rightharpoonup G(B_iu) \text{ in }U_1
%	$ 
	\cite[Lemma 5.3]{muench}.
	This implies
	\begin{equation*}
	v(r_0) \leq J(G(B_i(\overline{u}+r_0)),\overline{u}+r_0) \leq \liminf_{k\rightarrow\infty} J(G(B_i(u_{n_k}+r_{n_k})),u_{n_k}+r_{n_k}) \leq v(r_0)-\varepsilon
	\end{equation*}
	which is a contradiction.
	Therefore,
	$
	v(r_0) \leq \liminf_{n\rightarrow \infty} v(r_n).
	$	
	Now suppose that $C$ is compact.
	We have to show that for any $\varepsilon>0$ there is a neighbourhood $U_{r_0}$ of $r_0$ such that
	$
	v(r) \leq v(r_0)+\varepsilon
	$
	for all $r\in U_{r_0}$.	
	We prove that we can choose neighbourhoods $U_{V(r_0)}$ of $V(r_0)$ and $U_{r_0}$ of $r_0$ such that 
	$
	J(G(B_i(u+r)),u+r) \leq v(r_0)+\varepsilon
	$
	for all $(u,r) \in U_{V(r_0)}\times U_{r_0}$.
	Suppose that such neighbourhoods do not exist.
	Then there is a sequence $\lbrace r_n \rbrace$ with $r_n\rightarrow r_0$, $n\rightarrow\infty$, and a sequence $\lbrace u_n \rbrace \subset V(r_0)\subset C$ such that
	$
	J(G(B_i(u_n+r_n)),u_n+r_n) > v(r_0)+\varepsilon
	$
	for all $n>0$.
	Because $J$ is continuous the set $V(r_0)$ is closed and therefore compact as a closed subset of a compact set.
	Hence, there exists a subsequence $\lbrace u_{n_k} \rbrace$ and some $\overline{u}\in V(r_0)$ with $u_{n_k}\rightarrow\overline{u}$ as $k\rightarrow\infty$.
	This yields
	\begin{equation*}
	v(r_0) = J(G(B_i(\overline{u}+r_0)),\overline{u}+r_0) = \lim_{k\rightarrow\infty} J(G(B_i(u_{n_k}+r_{n_k})),u_{n_k}+r_{n_k}) > v(r_0)+\varepsilon
	\end{equation*}
	which is a contradiction.
	So the neighbourhoods $U_{V(r_0)}$ and $U_{r_0}$ do exist and for any $r\in U_{r_0}$ we obtain
	$
	v(r) \leq \inf_{u\in U_{V(r_0)}} J(G(B_i(u+r)),u+r) \leq v(r_0)+\varepsilon
	$
	which implies that $v$ is upper semicontinuous.	
	The last statement follows just as in \cite[Proposition 4.4]{bonnans}.
\end{proof}

\section*{Acknowledgement}
The author is supported by the DFG through the International Research Training Group IGDK 1754 „Optimization and Numerical Analysis for Partial Differential Equations with Nonsmooth Structures".
The author would like to thank Prof. Brokate from the Technical University of Munich and Prof. Fellner from the Karl-Franzens University of Graz for thoroughly proofreading the manuscript.

\newpage
\printbibliography
\end{document}